\documentclass[leqno]{amsart}
\usepackage[utf8]{inputenc}
\usepackage[english]{babel}
\usepackage{newclude}
\usepackage{amsfonts,amssymb,amsmath,amsgen, amsthm}
\usepackage{color, xcolor}
\usepackage{pdfsync}
\usepackage{enumerate}
\usepackage{centernot}
\usepackage{comment}
\usepackage{hyperref}
\usepackage[a4paper, total={6in, 8in}]{geometry}

\newtheorem{theorem}{Theorem}[section]
\newtheorem{lem}[theorem]{Lemma}
\newtheorem{prop}[theorem]{Proposition}
\newtheorem{cor}[theorem]{Corollary}
\theoremstyle{remark}
\newtheorem{remark}[theorem]{Remark}
\theoremstyle{definition}
\newtheorem{defin}[theorem]{Definition}

\def\R{\mathbb R}
\def\C{\mathbb C}

\def\d{\partial}

\newcommand{\ii}{\mathrm{i}\mkern1mu}

\def\QQ{\tilde{Q}_{b}}
\def\ZZ{\tilde{\zeta_b}}

\title[Blow-up of the Damped NLS Equation]{On the formation of singularities for the slightly supercritical NLS equation with nonlinear damping}

\author[P. Antonelli]{Paolo Antonelli}
\address{Gran Sasso Science Institute, viale Francesco Crispi, 7, 67100 L'Aquila, Italy}
\email{paolo.antonelli@gssi.it}

\author[B. Shakarov]{Boris Shakarov}
\address{Gran Sasso Science Institute, viale Francesco Crispi, 7, 67100 L'Aquila, Italy}
\email{boris.shakarov@gssi.it}

\begin{document}
\maketitle

\begin{abstract}
We consider the focusing, mass-supercritical NLS equation augmented with a nonlinear damping term. 
We provide sufficient conditions on the nonlinearity exponents and damping coefficients for finite-time blow-up.
In particular, singularities are formed for focusing and dissipative nonlinearities of the same power, provided that the damping coefficient is sufficiently small.
Our result thus rigorously proves the non-regularizing effect of nonlinear damping in the mass-supercritical case, which was suggested by previous numerical and formal results.
\newline
We show that, under our assumption, the damping term may be controlled in such a way that the self-similar blow-up structure for the focusing NLS is approximately retained even within the dissipative evolution.
The nonlinear damping contributes as a forcing term in the equation for the perturbation around the self-similar profile, that may produce a growth over finite time intervals. 
We estimate the error terms through a modulation analysis and a careful control of the time evolution of total momentum and energy functionals.
\end{abstract}

\section{Introduction}
In this work, we consider the NLS equation with nonlinear damping
\begin{equation} \label{eq:main}
 \begin{cases}
 & \ii \partial_t \psi + \Delta \psi + |\psi|^{2\sigma_1} \psi + \ii \eta |\psi|^{2\sigma_2} \psi = 0, \\ 
 & \psi(0) = \psi_0 \in H^1(\R^d),
 \end{cases}
\end{equation}
where $\psi: \R^+ \times \R^d \rightarrow \C$ and $\eta > 0$ is the damping coefficient. More precisely, our goal is to investigate the formation of singularities in finite time.
Equation \eqref{eq:main} arises as an effective model in various contexts, see for instance \cite{dyachenko1992optical, berge2009modeling, BaJaMa03}.
\newline
The nonlinear damping appearing in \eqref{eq:main} is usually introduced as a regularizing term of the singular dynamics provided by the focusing NLS equation. It is well known that the undamped NLS equation \eqref{eq:main} with $\eta=0$ and $\sigma_1\geq\frac{2}{d}$ may experience the formation of singularity in finite time, see \cite{Gl77, HoRo07}. From the modeling point of view, this means that the NLS effective description fails to be accurate close to the blow-up time and further effects, that were neglected in the derivation of the NLS equation, become relevant near the singularity. Several phenomena may be taken into account within the NLS description, see for instance \cite{DuLaSz16} for a quite general overview. In this work, we focus on nonlinear damping terms as in \eqref{eq:main}.
\newline
A relevant question in this perspective is to determine whether the nonlinear damping truly acts as a regularization in the vanishing dissipation regime. More precisely, given the singular dynamics for $\eta=0$, we are interested in determining whether the regularized equation \eqref{eq:main} has a global solution for any $\eta>0$, no matter how small.
\newline
The eventual (weak) limit as $\eta\to0$ of such solutions (if it exists) may be seen as a possible criterion to continue the solution beyond the singularity, in the same spirit of vanishing viscosity limits for conservation laws \cite{BiBr05}. To our knowledge, the only related rigorous results available in the literature are due to Merle \cite{merle1992uniqueness, merle1992limit}, where Hamiltonian-type regularization is adopted, see also \cite{Lan20} for a related result for the generalized KdV equation. On the other hand, several numerical simulations were performed to investigate this issue, see for instance \cite{SuSu07,Fi15,FiKl03, FiPa00} and the references therein.
\newline
The regularizing property of nonlinear damping is already established in some cases. For $\frac2d\leq\sigma_1<\sigma_2\leq\frac{2}{(d-2)^+}$ \cite{AnSp10,AnCaSp15} and for $\frac2d=\sigma_1=\sigma_2$ \cite{Da14}, the Cauchy problem \eqref{eq:main} is globally well-posed in $H^1$ for any $\eta>0$. On the other hand, there are also other cases where the damping does not act as a regularization and the dynamics remains singular for sufficiently small values of $\eta>0$. This is the case for instance of a linear damping $\sigma_2=0$ \cite{Ts84, OhTo09}. Moreover, in \cite{Da14} it is also shown that, for $0\leq\sigma_2<\sigma_1=\frac2d$, it is possible to determine an open set of initial data that develop a singularity in finite time. This is achieved by adapting the analysis developed in \cite{MeRa02,MeRa04,MeRa03,MeRa05,Ra05} by Merle and Rapha\"el for the mass-critical NLS, where they study the stability of blow-up in a self-similar regime.
\newline
In the physically relevant case $\frac2d<\sigma_1=\sigma_2\leq\frac{2}{(d-2)^+}$, this question remains unanswered in the general case. In \cite{AnCaSp15} the authors prove global well-posedness of \eqref{eq:main} in $H^1$ only for sufficiently large $\eta$, namely by requiring $\eta\geq\min(\sigma_1,\sqrt{\sigma_1})$. It is not clear whether this result is sharp, however numerical simulations \cite{SuSu07,Fi15,FiKl03} suggest that finite time blow-up still occurs for small values of $\eta$.
\newline
This work provides a rigorous answer to this question in the slightly supercritical case, confirming the numerical findings of \cite{FiKl03}. In particular, we prove that for $\frac2d<\sigma_1=\sigma_2<\frac2d+\delta^*$, where $\delta^*>0$ is sufficiently small, it is possible to provide an open set of initial data that develop a singularity in finite time in a self-similar regime. Our result shows that the self-similar blow-up regime studied in \cite{MeRaSz10} remains unaltered even under the action of the dissipative effects encoded in \eqref{eq:main}. In fact, we are going to prove a more general theorem on finite time blow-up for a larger class of nonlinear damping terms, see Theorem \ref{thm:mainBU1} below.
\newline
We first introduce the following notations. Let $s_c=s_c(\sigma_1)$ be the Sobolev critical exponent associated with $\sigma_1$,
\begin{equation}\label{eq:sc}
s_c = \frac{d}{2} - \frac{1}{\sigma_1},
\end{equation}
namely $s_c$ determines the critical regularity $\dot H^{s_c}$ for the well-posedness of equation \eqref{eq:main} with $\eta=0$. We also define
\begin{equation}\label{eq:sigma2}
\begin{aligned}
\sigma_* = \frac{2s_c}{d - 2 s_c} = s_c \sigma_1,\quad
\sigma^* = \frac{2}{d-2} - \sigma_1.
\end{aligned}
\end{equation}
Moreover, we also set 
\begin{equation}\label{eq:sigma2_max}
\sigma_{2, max}=\left\{\begin{array}{cc}
\sigma_1&\textrm{for}\,d\leq3,\\
\sigma^*&\textrm{for}\,d\geq4.
\end{array}\right.
\end{equation}
Our main result is stated as follows.
\begin{theorem}\label{thm:mainBU1}
There exists $\sigma_{crit}>\frac2d$, such that for any
$\frac2d<\sigma_1<\sigma_{crit}$, 
$\sigma_*<\sigma_2\leq\sigma_{2, max}$ the following holds true. There exists $\eta^*=\eta^*(\sigma_1)>0$ such that for any $0<\eta\leq\eta^*$,
there exists an open set $\mathcal O\subset H^1(\R^d)$ such that if $\psi_0 \in\mathcal O$ then the corresponding solution $\psi \in C([0,T_{max}),H^1(\R^d))$ to \eqref{eq:main} develops a singularity in finite time, that is $T_{max} < \infty$ and 
\begin{equation*}
	\lim_{t \to T_{max}} \| \nabla \psi(t)\|_{L^2} = \infty.
\end{equation*}
\end{theorem}
A more precise statement of our blow-up result is provided later in Theorem \ref{thm:mainBU}. In particular, the explicit blow-up rate is given there.
As previously said, our proof follows the strategy developed in \cite{MeRaSz10} for the Hamiltonian dynamics, given by \eqref{eq:main} with $\eta=0$, which in turn exploited previous results by the same authors related to the formation of singularities in the mass-critical case \cite{MeRa03, MeRa04}. More precisely, we construct a set of initial data whose evolution is almost self-similar. By a fine control of the modulation parameters entering the description of the self-similar regime, it is then possible to show the occurrence of finite-time blow-up.
\newline
Let us emphasize that the introduction of a dissipative term in the dynamics introduces further mathematical difficulties. 
First of all, the self-similar profile we consider in our analysis is determined by the undamped equations, see also \eqref{eq:Qb1} below, for instance. At present, it is not even clear whether it would be possible to determine a profile that takes into account also the dissipative term. A generalized notion of dissipative solitons is present in the literature \cite{AfAkSo96, akhmediev2001solitons, HaIbMa21}, where the profiles are determined not only by the balance between dispersion and focusing effects but also between gain and loss terms. In \eqref{eq:main} the sole presence of a nonlinear damping cannot be balanced by other effects.
The fact that the self-similar profile is determined by the Hamiltonian part of the equation, yields a non-trivial forcing term in the equation for the perturbation, given by the nonlinear damping itself. 
Through a careful modulation analysis, we determine conditions under which this forcing can be controlled, so that the perturbation is shown to be sufficiently small with respect to the self-similar profile.
In particular, in the case $\sigma_1=\sigma_2$, the control is determined by imposing a smallness condition on $\eta>0$, as stated in the main theorem above.
\newline
Moreover, a second main difficulty is that in our case the functionals related to the physical quantities such as total mass, momentum and energy are - straightforwardly - not conserved in time anymore. We thus need a suitable control on the time evolution of these quantities that will in turn provide the necessary bounds on the perturbation and the modulation parameters.
\newline
For $d\geq4$, the restriction $\sigma_2\leq\sigma^*<\sigma_1$ prevents us to consider the case $\sigma_2=\sigma_1$. This is a technical condition, needed to ensure the validity of the Sobolev embedding $H^1(\R^d) \hookrightarrow L^{2(\sigma_1 + \sigma_2 + 1)}(\R^d)$ in \eqref{eq:en_grow}, \eqref{eq:momGrow} below, see also Section \ref{ch:bootstrap}. On the other hand, the condition $\sigma_*<\sigma_2$ is motivated by the fact that we cannot control Sobolev norms $\dot H^s$ norms with $s<s_c$ in the self-similar regime, see \eqref{eq:lowerS}. We remark that for the undamped dynamics \eqref{eq:main} with $\eta=0$, it was proved in \cite{MeRa08} that all radially symmetric blowing-up solutions leave the critical space $\dot H^{s_c}$ at blow-up time.
\newline
Let us further remark that the smallness condition $\eta\leq \eta^*$ is only necessary when $\sigma_2 = \sigma_1$. 
In the case $\sigma_1 > \sigma_2$, it is also possible to show that for any $\eta$, there exists a set of initial conditions depending on $\eta$, $\sigma_1$ and $\sigma_2$ whose corresponding solutions blow-up in finite time. We will further discuss this point in Section \ref{sec:concl} and throughout this work. \newline
Moreover, as it will be clear in our analysis, the smallness of $\eta^*$ is determined in terms of the smallness of $s_c$, which is related to $\sigma_1$ through identity \eqref{eq:sc}. Indeed, we will see that $\eta^*\sim s_c^3$. 
For this reason, with some abuse of notation, in what follows we often write $\eta^*(s_c)$ instead of $\eta^*(\sigma_1)$, where $s_c$ and $\sigma_1$ are related by identity \eqref{eq:sc}.\newline
Finally, our result provides additional evidence that the mass supercritical self-similar collapse is very different from the one occurring in the mass critical case. In the latter case, indeed, any damping coefficient $\eta>0$ regularizes the dynamics and prevents the formation of finite time blow-up \cite{Da14, AnSp10}. 
In Section \ref{sec:concl} we will provide an argument explaining why the solution escapes the self-similar regime in the case $\sigma_1 = \sigma_2 = \frac{2}{d}$ for any $\eta >0$.
 We will now present a road map of how Theorem \ref{thm:mainBU} will be proved.

\subsection{Singularity formation}

We start with an initial condition which can be decomposed as a soliton profile and a small perturbation 
\begin{equation}\label{eq:deco00}
\psi_0(x) = \lambda_0^{-\frac{1}{\sigma_1}} \left(Q_{b_0}\left( \frac{x-x_0}{\lambda_0} \right) + \xi_0\left( \frac{x-x_0}{\lambda_0} \right) \right) e^{\ii \gamma_0}, 
\end{equation}
where $0 < b_0, \lambda_0 \ll 1$ are small parameters, and $Q_{b_0}$ is roughly a localized radial solution to the stationary equation 
\begin{equation} \label{eq:Qb1}
\Delta Q_{b_0} - Q_{b_0} + \ii {b_0} \left(\frac{1}{\sigma_1} Q_{b_0} + x \cdot \nabla Q_{b_0}\right) + |Q_{b_0}|^{2\sigma_1}Q_{b_0} = 0.
\end{equation}
It is known (see, for instance, \cite{SuSu07, MeRaSz10}) that self-similar blowing-up solutions of the supercritical NLS focus as a zero energy solution to \eqref{eq:Qb1} plus a non-focusing radiation. For any initial value $Q_{b_0}(0)$ and any ${b_0}\in \R$, nontrivial solutions to \eqref{eq:Qb1} exist \cite{Wa90,BuChRu99}, but any zero energy solution does not belong to $L^2(\R^d)$ \cite{JoPa93} and thus we employ a suitable localization in space. Moreover, the parameter $b_0$ is chosen to be close to a value $b^* = b^*(s_c) > 0$. By continuity, we may find a time interval where the solution can be decomposed as 
\begin{equation}\label{eq:deco0}
\psi(t,x) = \frac{1}{\lambda^\frac{1}{\sigma_1}(t)} \left(Q_{b(t)}\left( \frac{x-x(t)}{\lambda(t)} \right) + \xi\left(t, \frac{x-x(t)}{\lambda(t)} \right) \right) e^{\ii \gamma(t)}
\end{equation}
where the parameter $b$ is still close enough to $b^*$ and $\lambda$ and $\xi$ are still small enough. By perturbation techniques, we choose the parameters $b,\lambda, x,\gamma$ so that $\xi$ satisfies four suitable orthogonality conditions. \newline
Next, we will use a local virial law and a suitable Lyapunov functional to prove that the parameter $b(t)$ is trapped around the value $b^*$ for all times. This yields the following law for the scaling parameter 
\begin{equation*}
	\lambda(t) \sim \sqrt{-2b^*t + \lambda_0^2}.
\end{equation*}
In particular, there exists a time $T_{max}(s_c,\lambda_0)>0$ such that $\lambda(t) \to 0$ as $t \to T_{max}$. Consequently, the kinetic energy of the solution diverges
\begin{equation*}
 \lim_{t \to T_{max}} \| \nabla \psi(t)\|_{L^2} = \infty. 
\end{equation*}
We will use the coercivity property stated in Proposition \ref{prp:quadForm} below to prove the dynamical trapping of the parameter $b$. In order to control the six negative directions on the right-hand side of \eqref{eq:quadform}, we will use four orthogonality conditions implied by the selection of the parameters $b,\lambda, x,\gamma$ and two almost orthogonal conditions which come from suitable bounds on the energy and the momentum of the solution. \newline 
In our case, there are two more difficulties that do not arise when $\eta = 0$. First, since $Q_b$ approximates a solution of the undamped NLS, one issue will be to show that the contribution of the damping term in \eqref{eq:main} can be considered of a smaller order with respect to the rest of the dynamics. in particular, the damping term generates a forcing in the equation of the remainder, which will be controlled using the smallness of $\eta$ and $\lambda$. Second, the presence of the damping implies that the energy 
\begin{equation*}
 E[\psi(t)] = \int \frac{1}{2} |\nabla \psi(t)|^2 - \frac{1}{2\sigma + 2}|\psi(t)|^{2\sigma + 2} \, dx
\end{equation*}
and the momentum
\begin{equation*}
 P[\psi(t)] = \big( \nabla \psi(t), \ii \psi(t) \big)
\end{equation*}
are not conserved. But to show the dynamical trapping of the parameter $b$ around $b^*$, we need $E$ and $P$ to remain small enough to control two negative directions in the coercivity \eqref{eq:quadform} below. In comparison with the undamped case, this is not a trivial consequence implied by the smallness of the energy and momentum of the initial condition. Thus, we will study their time evolution and show that they remain sufficiently small until the blow-up time under the assumption $d \leq 3$, $\sigma_2 \leq \sigma_1$ and $\eta \leq \eta^*(s_c)$ or $d \geq 4$ and $\sigma_2 < \sigma^*$. 
Finally, we recall that the localized profile $Q_b$ is close in $H^1(\R^d)$ to the ground state of the undamped NLS which is the unique real-valued solution $Q \in H^1(\R^d)$ \cite{Kw89, GiNiNi79} to the elliptic equation 
\begin{equation}\label{eq:gsQ}
 \Delta Q - Q + |Q|^{2\sigma_1}Q = 0.
\end{equation}
In light of the discussion above, Theorem \ref{thm:mainBU1} will be the consequence of the following result.

\begin{theorem} \label{thm:mainBU}
There exists $s_c^* >0$ such that for any $0 <s_c <s_c^*$ and $ \sigma_*< \sigma_2 \leq \sigma_1$ when $d \leq 3$ or $ \sigma_*< \sigma_2 < \sigma_{2,max}$ when $d \geq 4$, there exists $\eta^*(s_c)>0$ such that if $\eta \leq \eta^*$, then
there exists an open set $\mathcal{O} \subset H^1(\R^d)$ such that if $\psi_0 \in \mathcal{O}$, then for any $t \in [0,T_{max})$ the corresponding maximal solution to \eqref{eq:main} can be written as 
\begin{displaymath}
 \begin{aligned}
 &\psi(t,x) = \frac{1}{\lambda^\frac{1}{\sigma_1}(t)}\left(Q \left(\frac{x - x(t)}{\lambda(t)} \right) + \zeta \left(t,\frac{x - x(t)}{\lambda(t)} \right)\right) e^{\ii \gamma(t)},
 \end{aligned}
\end{displaymath}
where $\lambda,\gamma \in C^1([0, T_{max});\R)$, 
$x \in C^1([0, T_{max}),\R^d)$, 
\begin{equation*}
 \lim_{t \to T_{max}} \lambda(t) = 0, \quad \lim_{t \to T_{max}} x(t)= x_\infty \in \R^d,
\end{equation*}
there exists $0< \delta(s_c) = \delta \ll 1$ such that 
\begin{equation*}
 \|\nabla \zeta \|_{L^\infty([0,T_{max}),L^2)} \leq \delta,
\end{equation*}
and the blow-up rate is given by
\begin{equation*}
 \| \nabla \psi(t)\|_{L^2}^2
 \sim\left(T_{max}-t\right)^{-(1-s_c)}.
\end{equation*}
\end{theorem}

More precisely, we will show that if we define 
\begin{equation}\label{eq:s_c}
 b^* = - \frac{\pi (1 + \delta)}{\ln(s_c)}
\end{equation}
then the law of the scaling parameter will satisfy the following bounds
\begin{equation}\label{eq:self-similare_speed}
\lambda^2(0) - 2(1 + \delta) b^* t \leq \lambda^2(t) \leq \lambda^2(0) -2(1 - \delta) b^* t.
\end{equation}

\begin{remark}
 Recently, in \cite{BaMaRa21}, the authors provided a complete description of a zero energy solution to \eqref{eq:Qb1} for $\sigma_1 \in \left( \frac{2}{d}, \frac{2}{d} + \varepsilon \right)$ and $\varepsilon$ small enough. 
\end{remark}

\section{Preliminaries}

We start this section with a list of notations that we are going to use throughout this work. \newline
 We use the symbol $A\sim B$, to indicate the fact that there exist two constants $C_1,C_2>0$ such that $C_1B \leq A \leq C_2B$.\newline
% For any two complex-valued functions $f,g \in L^2(\R^d)$, we define the $L^2(\R^d)$ scalar product as
% \begin{displaymath}
% (f,g) = Re \int f(x) \bar{g} (x)\, dx.
% \end{displaymath}
 For any $f \in H^1(\R^d)$, we define the operators
\begin{equation}\label{eq:LambdaD}
	\Lambda f = \frac{1}{\sigma_1} f + x \cdot \nabla f, \quad D f = \frac{d}{2} f + x \cdot \nabla f.
\end{equation}
We notice that
\begin{equation}\label{eq:Lambda-D}
 \Lambda f = D f - s_c f.
\end{equation}
Moreover, by integrating by parts we have 
\begin{equation}\label{eq:Lambdapp}
 (f, \Lambda g) = -2s_c(f,g) - (g, \Lambda f)
\end{equation}
and 
\begin{equation}\label{eq:LambLap}
 \Delta \Lambda f = 2 \Delta f + \Lambda \Delta f. 
\end{equation}
 We use the following notation 
\begin{equation*}
 \frac{1}{(d - 2)^+} = \begin{cases}
 \infty, \ \mbox{ for } \ d\leq 2, \\
 \frac{1}{(d - 2)}, \ \mbox{ for } \ d\geq 3. 
 \end{cases}
\end{equation*}
Next, we recall some preliminary results.

\begin{defin}\label{def:strich}
 We say that a pair $(q,r)$ is admissible if $2 \leq q, r \leq \infty$, $(q,r,d) \neq (2,\infty,2)$ and 
 \begin{equation*}
 \frac{2}{q} = d \left( \frac{1}{2} - \frac{1}{r} \right)
 \end{equation*}
\end{defin}

We will exploit the following Strichartz estimates \cite{GiVe85, KeTa98}.

\begin{theorem}\label{thm:strichartz}
For every $\phi \in L^2(\R^d)$ and every $(q,r)$ admissible, there exists a constant $C>0$ such that for any $t > 0$, 
\begin{displaymath}
\| e^{\ii t\Delta} \phi \|_{L^q((0,t),L^r)} \leq C \| \phi\|_{L^2}.
\end{displaymath}
Moreover, if $f \in L^{\gamma'}((s,t), L^{\rho'}(\R^d))$ where $(\gamma, \rho)$ is an admissible pair, and 
\begin{displaymath}
N(f) = \int_{s}^t e^{\ii (t-\tau) \Delta} f(\tau) \, d\tau,
\end{displaymath}
then there exists a constant $C>0$ such that for any $(q,r)$ admissible, we have
\begin{displaymath}
\| N(f)\|_{L^q((s,t), L^r)} \leq C \| f\|_{L^{\gamma'}((s,t), L^{\rho'})}.
\end{displaymath} 
\end{theorem}

By using the Strichartz estimates, it is possible to prove the local existence of solutions to equation \eqref{eq:main} (see \cite[Proposition $2.3$]{AnCaSp15}).

\begin{theorem}\label{thm:locEx}
Let $\eta \in \R$, $\sigma_1, \sigma_2 < 2/(d-2)$ if $d \geq 3$. Then for any $\psi_0\in H^1(\R^d)$ there exists $T_{max}>0$ and a unique solution $\psi$ to \eqref{eq:main} such that for any $(q,r)$ admissible, we have
\begin{equation}\label{eq:strich_sol}
 \psi,\nabla\psi \in C\left([0,T_{max}),L^2(\R^d) \right) \cap L^q_{loc}\left((0,T_{max}), L^r(\R^d)\right).
\end{equation}
Moreover, either $T_{max} = \infty$ or $T_{max} < \infty$ and 
\begin{equation*}
 \lim_{t \to T_{max}} \| \nabla \psi(t) \|_{L^2} = \infty.
\end{equation*}
\end{theorem}

The usual physical quantities (the total mass, the total energy and the momentum) are governed by the following time-dependent functions \cite{Ca03}. 

\begin{theorem}
 Let $\psi_0 \in H^1(\R^d)$ and $\psi \in C\left([0,T_{max}),H^1(\R^d)\right)$ the corresponding solution to \eqref{eq:main}. Then the total mass, the total energy and the momentum satisfy 
\begin{equation}\label{eq:mass_grow}
M[\psi(t)] = \int |\psi(t)|^2 \, dx = M[\psi_0] - 2\eta \int_0^t \| \psi(s)\|_{L^{2\sigma_2 + 2}}^{2\sigma_2 + 2} \, ds,
\end{equation}
\begin{equation}\label{eq:en_grow}
\begin{aligned}
E[\psi(t)] &= \int \frac{1}{2} |\psi(t)|^2 -\frac{1}{2\sigma_1 + 2} |\psi(t)|^{2\sigma_1 + 2} \, dx \\
& = E[\psi_0] + \eta \int_0^t \int | \psi |^{2\sigma_1 + 2\sigma_2 +2} - |\psi|^{2\sigma_2} |\nabla \psi|^2 \\ 
& - 2\sigma_2 |\psi|^{2\sigma_2 - 2} Re \left( \bar{\psi}\nabla \psi \right)^2 \, dx \, ds, 
\end{aligned}
\end{equation}
and
\begin{equation}\label{eq:momGrow}
 P[\psi(t)] = \big( \nabla \psi(t), \ii \psi(t) \big) = P[\psi_0] - 2\eta \int_0^t \int |\psi|^{2\sigma_2} Im ( \bar{\psi} \nabla \psi) \, dx \, ds, 
\end{equation} 
respectively. 
\end{theorem}

\begin{remark}
The dissipative terms appearing in \eqref{eq:mass_grow}, \eqref{eq:en_grow} and \eqref{eq:momGrow} are always well defined because of Strichartz estimates in Theorem \ref{thm:locEx}, see \eqref{eq:strich_sol}.
\end{remark}
We also notice that if $\sigma_2< \sigma_1 < \frac{1}{d-2}$, then it follows 
 \begin{equation*}
 4\sigma_2 + 2 \leq 2(\sigma_1 + \sigma_2 + 1) \leq 4\sigma_1 + 2 < \frac{2d}{(d-2)^+}.
 \end{equation*}
Thus, by the Sobolev embedding theorem, the positive term on the right-hand side of \eqref{eq:en_grow} can be bounded by
 \begin{equation*}
 \int_s^t \int |\psi|^{2(\sigma_1 + \sigma_2 + 1)} \, dx \, d\tau \leq (t-s) \| \psi\|_{L^{\infty}((s,t),H^1)}^{2(\sigma_1 + \sigma_2 + 1)}.
 \end{equation*}
Analogously, we can also estimate the term on the right-hand side of \eqref{eq:momGrow} as
\begin{equation*}
\int_s^t \int |\psi|^{2\sigma_2} Im ( \bar{\psi} \nabla \psi) \, dx \, d\tau \lesssim (t-s) \| \psi \|^{2\sigma_2 +1}_{L^{4\sigma_2 +2}} \| \nabla \psi\|_{L^2} \lesssim (t-s) \| \psi \|^{2\sigma_2 +2}_{H^1}.
\end{equation*} 
These computations will be exploited in Section \ref{ch:bootstrap}. Indeed, they will be crucial to control uniformly in time the right-hand side of \eqref{eq:en_grow} and \eqref{eq:momGrow}.

\subsection{Construction of the Approximated Soliton Core} \label{sec;profile}

In this subsection, we are going to construct a suitable localized solution to the stationary equation \eqref{eq:Qb1}, that will constitute the blowing-up soliton core in the self-similar regime. This construction already appeared in \cite[Section $2$]{MeRaSz10}, however for the sake of completeness we are going to present the main related results and properties here.
\newline
Let us notice that the damping term does not enter into the construction of the approximated blow-up core.
\newline
Let $\rho \in (0,1)$, $b > 0$, we define the following radii 
\begin{equation*}
 R_b = \frac{2}{b} \sqrt{1 - \rho}, \ R_b^- = \sqrt{1 - \rho} R_b. 
\end{equation*}
 Let $\phi_b \in C^\infty(\R^d)$ be a radially symmetric cut-off function such that
\begin{equation}\label{eq:cut-off}
\phi_b(x) = \begin{cases} &0, \mbox{ for } |x| \geq R_b, \\
& 1, \mbox{ for } |x| \leq R_b^-,
\end{cases} \quad \quad 0 \leq \phi_b(x) \leq 1.
\end{equation}
We also denote the open ball in $\R^d$ of radius $R_b$ and centered at the origin by
\begin{equation*}
 B(0,R_b) = \{ x\in \R^d: \, |x| < R_b\}.
\end{equation*}
We recall that $s_c$ denotes the Sobolev critical exponent defined in \eqref{eq:sc}.
We start with the following lemma.

\begin{lem}\label{thm:exisPb}
There exists $\sigma_c^{(1)}>\frac2d$ and $\rho^{(1)}> 0$, such that for any $\frac{2}{d} < \sigma_1 < \sigma_c^{(1)} $, $\rho \in (0, \rho^{(1)})$, there exists $b^{(1)}(\rho)>0$, such that for any $0 < b \leq b^{(1)}$, there exists a unique radial solution $P_b \in H^1_{0}(B(0, R_b))$ to the elliptic equation
\begin{equation*}
 \left(- \Delta + 1 - \frac{b^2|x|^2}{4} \right) P_b - P_b^{2\sigma_1 + 1} = 0,
\end{equation*}
with $P_b > 0 \mbox{ in } B(0, R_b)$. Moreover $P_b \in C^3(B(0,R_b))$ and 
\begin{equation*}
 \| P_b - Q \|_{C^3} \to 0
\end{equation*}
as $b \to 0$, where $Q$ the unique positive solution to \eqref{eq:gsQ}. Finally $b^{(1)}(\rho) \to 0$ as $\rho \rightarrow 0$. 
\end{lem}

This lemma was stated in \cite[Proposition $2.1$]{MeRaSz10} and its proof is a straightforward adaptation of that given for \cite[Proposition $1$]{MeRa03} in the mass-critical case. To prove the existence, one shows that $P_b$ is a suitably scaled minimizer of the functional 
\begin{equation*}
 F(u) = \int_{B(0, R_b)} |\nabla u|^2 +\left( 1 - \frac{b^2|x|^2}{4} \right) |u|^2\, dx
\end{equation*}
in the set 
\begin{equation*}
 \mathcal{U} = \left\{u \in H^1_{0,rad}(B(0,R_b)) : \, \| u\|_{L^{2\sigma_1 + 2}} = 1 \right\},
\end{equation*}
where $H^1_{0,rad}(B(0,R_b)) $ is the subset of $H^1_0(B(0,R_b)) $ containing radially symmetric functions.
Notice that the functional $F$ is bounded from below in $B(0, R_b)$ and $ - \Delta + 1 - \frac{b^2|x|^2}{4}$ is an uniformly elliptic operator. In particular, both the maximum principle and standard regularity results for
elliptic equations are satisfied, see \cite{GiTr01} for instance. \newline
Next, we define the function
\begin{equation*}
 \hat{Q}_b = \phi_b e^{-\ii b \frac{|x|^2}{4}} P_b,
\end{equation*}
 in the whole $\R^d$ space, where the cut-off $\phi_b$ is defined in \eqref{eq:cut-off}.
Notice that the profile $\hat{Q}_b \in H^1(\R^d)$ satisfies the equation
\begin{equation} \label{eq:Qb_0}
\Delta\hat{Q}_b -\hat{Q}_b + \ii b \left(\frac{d}{2}\hat{Q}_b + x \cdot \nabla\hat{Q}_b\right) + |\hat{Q}_b|^{2\sigma_1}\hat{Q}_b = - \tilde{\Psi}_b,
\end{equation}
where the remainder term $\tilde{\Psi}_b$ comes from the localization in space and is defined by
\begin{equation}\label{eq:psib0}
 \tilde{\Psi}_b = \left(2 \nabla \phi_b \cdot P_b + P_b(\Delta \phi_b) + (\phi_b^{2\sigma_1 +1} - \phi_b) P_b^{2\sigma_1 + 1}\right)e^{-\ii b \frac{|x|^2}{4}}.
\end{equation}
We recall some properties of the profile $\hat{Q}_b$ which were shown in \cite[Proposition $2.1$]{MeRaSz10}.

\begin{prop}
For any polynomial $p$ and any $k=0,1$, there exists a constant $C>0$ such that 
\begin{equation} \label{eq:Psi_grande}
\left\| p \frac{d^k}{dx^k} \tilde{\Psi}_b \right\|_{L^\infty} \leq e^{-\frac{C}{|b|}}. 
\end{equation}
Moreover, the function $\hat{Q}_b$ satisfies 
\begin{displaymath}
P[ \hat{Q}_b] = \left( \nabla\hat{Q}_b, \ii\hat{Q}_b \right)= 0, \quad \left(\Lambda \hat{Q}_b, \ii\hat{Q}_b \right) = \left(x \cdot \nabla\hat{Q}_b, \ii\hat{Q}_b \right) = - \frac{b}{2} \|x\hat{Q}_b\|_{L^2}^2,
\end{displaymath}
where $\Lambda$ is defined in \eqref{eq:LambdaD} and
\begin{displaymath}
\frac{d}{d b^2} \left( \int |\hat{Q}_b|^2 \, dx \right)_{|b=0} = C(\sigma_1),
\end{displaymath}
with $ C(\sigma_1) \rightarrow C>0$ as $\sigma_1 \rightarrow \frac{2}{d}$.
\end{prop}
Notice that heuristically $\hat Q_b$ provides an approximating solution to \eqref{eq:Qb1} in the ball $B(0, R_b)$. In the complementary region, $\R^d\setminus B(0, R_b)$ nonlinear effects become negligible, hence it is sufficient to study the outgoing radiation defined in the next lemma, that first appeared in \cite[Lemma $15$]{MeRa04}. The results in the lemma below will be exploited, in particular, to estimate the mass flux leaving the collapsing core in Lemma \ref{thm:refvirial} below.

\begin{lem}\label{lem:rad}
There exists $\rho_2 > 0$ such that for any $0<\rho < \rho_2$ there exists $b_2(\rho)> 0$ such that for any $0 < b < b_2$, there exists a unique radial solution $\zeta_b \in \dot{H}^1(\R^d)$ to 
\begin{equation}\label{eq:rad}
\Delta \zeta_b - \zeta_b + \ii b \left(\frac{d}{2} \zeta_b + x \cdot \nabla \zeta_b\right) = \tilde{\Psi}_b,
\end{equation}
where $\tilde{\Psi}_b$ is defined in \eqref{eq:psib0}. 
Moreover there exists $C>0$ such that
\begin{equation}\label{eq:Gammab}
\Gamma_b = \lim_{|x| \rightarrow + \infty} |x|^d |\zeta_b(x)|^2 \lesssim e^{-\frac{C}{b}},
\end{equation}
and there exists $c>0$ such that for $|x| \geq R_b^2$
\begin{equation}\label{eq:gammaB1}
 e^{\frac{-(1 + c\rho)\pi}{b}} \leq \frac{4}{5} \Gamma_b \leq |x|^d |\zeta_b(x)|^2 \leq e^{\frac{-(1 - c\rho)\pi}{b}}. 
\end{equation}
Furthermore, the following estimates hold
\begin{equation}\label{eq:zetaDer}
 \| \nabla \zeta_b\|_{L^2}^2 \leq \Gamma_b^{1 - c\rho},
\end{equation}
\begin{equation*}
 \left\| |y|^\frac{d}{2}( |\zeta_b| + |y| |\nabla \zeta_b|) \right\|_{L^\infty(|y| \geq R_b)} \leq \Gamma_b^{\frac{1}{2} - c\rho}.
\end{equation*}
Finally, we have that
\begin{equation*}
 \left\| |\zeta_b| e^{-|y|} \right\|_{C^2(|y| \leq R_b)} \leq \Gamma_b^{\frac{3}{5}},\quad
 \left\| \partial_b \zeta_b \right\|_{C^1} \leq \Gamma_b^{\frac{1}{2} - c\rho}.
\end{equation*}
\end{lem}

\begin{remark}
 We observe that the profile $\zeta_b \in \dot{H}^1_{rad}(\R^d)$ is not in $L^2(\R^d)$ because of the logarithmic divergence implied by \eqref{eq:Gammab}, \eqref{eq:gammaB1}.
\end{remark}

We now want to suitably modify the profile $\hat{Q}_b$ to obtain an approximating solution to 
\begin{equation}\label{eq:Qbreal}
	\ii \partial_t Q_{b} + \Delta Q_{b} - Q_{b} + \ii {b} \left(\frac{1}{\sigma_1} Q_{b}+ x \cdot \nabla Q_{b} \right) + |Q_{b}|^{2\sigma_1} Q_{b} = 0,
\end{equation}
where $b = b(t)$ is a function depending on time. Notice that if we suppose that $ \dot b(t) = 0$ and $ d \sigma_1= 2$, then $\hat Q_b$ is already an approximating solution to \eqref{eq:Qbreal}, see \eqref{eq:Qb_0} and \eqref{eq:Psi_grande}. 
In our case, we suppose that there exists $\beta(b)>0$ such that $ \dot{b}(t)= \beta s_c$.
In other words, we search for localized solutions to the following equation
\begin{equation}\label{eq:qbDer}
 \ii s_c \beta \partial_b Q_b + \Delta Q_b - Q_b + \ii b\left(\frac{1}{\sigma_1} Q_b + x \cdot \nabla Q_b \right) + | Q_b |^{2\sigma_1 }Q_b = 0.
\end{equation}

We find a solution to this equation as a suitable perturbation of the profile $\hat{Q}_b$. 
We make the following ansatz on $Q_b$,
\begin{equation*}
 Q_b =\hat{Q}_b + s_c T_b,
\end{equation*}
by requiring that it solves equation \eqref{eq:qbDer} inside the ball $|x| \leq R_b^-$, up to an error of order $s_c^{1 + C}$ for some $C>0$. Such a solution was already studied in \cite[Proposition $2.6$]{MeRaSz10}.

\begin{prop} \label{thm:profile}
There exists $\tilde\sigma_c>\frac2d$ and $\rho_3>0$ such that for any $\frac{2}{d} < \sigma_1 < \tilde\sigma_c$, $\rho \in (0, \rho_3)$ there exists $b_3(\rho)>0$ such that for any $0 < b < b_3$ there exists a radial function $T_b\in C^3(\R^d)$ and a constant $\beta> 0$ 
such that $Q_b = e^{-\ii b \frac{|y|^2}{4}} \left( P_b \phi_b + s_c T_b \right)$ satisfies 
\begin{equation} \label{eq:Qb}
\begin{aligned}
 \ii s_c \beta \partial_b Q_b + \Delta Q_b - Q_b + \ii b \Lambda Q_b + |Q_b|^{2\sigma_1 }Q_b & = - \Psi_b
\end{aligned}
\end{equation} 
where 
\begin{equation}\label{eq:psi+phi}
 \Psi_b = \tilde{\Psi}_b + \Phi_b,
\end{equation}
 and we have
\begin{equation} \label{eq:psiB}
\begin{aligned}
 &\left\| \frac{d^k}{dx^k} \Phi_b \right\|_{L^\infty(|x| \leq R^-_b)} \lesssim s_c^{1 +C}, \\
 &\left\| \frac{d^k}{dx^k} \Phi_b \right\|_{L^\infty(|x| \geq R^-_b)} \lesssim s_c,
\end{aligned}
\end{equation}
for $k = 0,1$ and $C >0$. 
Moreover, $Q_b$ satisfies
 \begin{equation}\label{eq:Qbprop}
\begin{aligned}
 &\left|E[Q_b]\right| \lesssim\Gamma_b^{1 - c\rho} + s_c, \\
 & P[Q_b] = \left( \nabla Q_b, \ii Q_b \right) = 0, \\
 & \left( \Lambda Q_b, \ii Q_b \right) = \left( x \cdot \nabla Q_b, \ii Q_b \right) = - \frac{b}{2} \|xQ\|_{L^2}^2(1 + \delta_1(s_c, b)), \\
 &\|Q_b\|_{L^2}^2 = \| Q\|_{L^2}^2 + O(s_c) + O(b^2),\\
\end{aligned}
 \end{equation}
where $\delta_1(s_c, b) \to 0$ as 
 as $b,s_c \rightarrow 0$ and $c\rho \ll 1 $ is defined in \eqref{eq:zetaDer} and $Q$ is a solution to \eqref{eq:gsQ}. The profile $Q_b$ satisfies also the uniform estimate
 \begin{equation}\label{eq:polEst}
 \left|P(x) \frac{d^k}{dx^k} Q_b(x)\right| \lesssim e^{-(1 + c) |x|}
\end{equation}
for any polynomial $P(x)$. 
 Finally, we have that
\begin{equation}\label{eq:Tb}
 \left\| e^{ \frac{(1 - \rho)\pi}{4} |x|} T_b \right\|_{C^3} + \left\| e^{\frac{(1 - \rho) \pi}{4} |x|} \partial_b T_b \right\|_{C^2} + \left| \partial_b \beta(b) \right| \lesssim 1.
\end{equation}
\end{prop}

We will now show a useful Pohozaev-type estimate for the profile $Q_b$.

\begin{prop}\label{prp:poho}
 Let $Q_b$ be a solution to \eqref{eq:Qb}. Then we have
 \begin{equation}\label{eq:poho}
 2 E[Q_b] - (\ii s_c \beta \partial_b Q_b + \Psi_b, \Lambda Q_b) = s_c\left( 2E[Q_b] + \| Q_b\|_{L^2}^2\right).
 \end{equation}
\end{prop}

\begin{proof}
 We take the scalar product of equation \eqref{eq:Qb} with $\Lambda Q_b$
 \begin{equation*}
 0 = ( \Delta Q_b - Q_b + \ii b\Lambda Q_b + |Q_b|^{2\sigma_1} Q_b, \Lambda Q_b) + (\ii s_c \beta \partial_b Q_b + \Psi_b, \Lambda Q_b).
 \end{equation*}
 Now by integration by parts, we observe that 
 \begin{equation*}
 ( \Delta Q_b + |Q_b|^{2\sigma_1} Q_b, \frac{1}{\sigma_1} Q_b + x\cdot \nabla Q_b) = 2(s_c - 1) E[Q_b],
 \end{equation*}
 and 
 \begin{equation*}
 -(Q_b, \Lambda Q_b) = s_c \| Q_b\|_{L^2}^2. 
 \end{equation*}
 as a consequence, we obtain that
 \begin{equation*}
 0 = 2(s_c - 1) E[Q_b] + s_c \| Q_b\|_{L^2}^2 + (\ii s_c \beta \partial_b Q_b + \Psi_b, \Lambda Q_b)
 \end{equation*}
 which is equivalent to \eqref{eq:poho}.
\end{proof}

Finally, by applying the operator $\Lambda$ to \eqref{eq:Qb}, we see that $\Lambda Q_b$ satisfies the following equation
\begin{equation}\label{eq:LambQb}
\begin{aligned}
 \ii \beta s_c \Lambda Q_b + \Lambda \Delta Q_b - \Lambda Q_b + \ii b \Lambda(\Lambda Q_b) + |Q_b|^{2\sigma_1} \Lambda Q_b \\ + 2\sigma_1 Re(\bar Q_b \, x \cdot \nabla Q_b) |Q_b|^{2\sigma_1 - 2} Q_b = -\Lambda \Psi_b,
\end{aligned}
\end{equation}
that will be used in the Appendix.

\subsection{Coercivity Property}

We conclude the section by discussing the coercivity properties of the linearized operator around the ground state in the mass-critical case. Since our analysis deals with the slightly mass-supercritical case, we derive a similar property by perturbative arguments. Let us first define the following quadratic form associated with the linearized operator around the ground state in the mass-critical case,
 \begin{equation}\label{eq:quadform}
 \begin{aligned}
 \mathcal{H}(f,f) & = \| \nabla f\|_{L^2}^2 + \frac{2}{d} \left( 1 + \frac{4}{d}\right) \int Q_c^{\frac{4}{d} -1} (y \cdot \nabla Q_c) Re(f)^2\, dy \\
 &+ \frac{2}{d} \int Q_c^{\frac{4}{d} -1} (y \cdot \nabla Q_c) Im(f)^2 dy.
 \end{aligned}
 \end{equation}
where $Q_c$ is the ground state profile for the mass-critical NLS, namely is the unique positive solution to
\begin{equation}
\Delta Q_c - Q_c + |Q_c|^\frac{4}{d} Q_c = 0.
\end{equation} 
We recall that
\begin{equation*}
 DQ_c = \frac{2}{d} Q_c + y\cdot \nabla Q_c.
 \end{equation*} 
We state the following coercivity property for the quadratic form $\mathcal H$. Let us remark that, even though we present it as a proposition, the following result was proved rigorously only in the one-dimensional case \cite{MeRa02}, while for the multi-dimensional case $d\leq10$ it was proved only numerically \cite{FiMeRa06, YaRoZh18}.
\begin{prop}\label{prp:quadForm}
 For any $f\in H^1(\R^d)$ there exists $c>0$ such that 
 \begin{equation} \label{eq:almost_coerc}
 \begin{aligned}
 \mathcal{H}(f,f) &\geq c \left( \| \nabla f\|_{L^2}^2 + \int |f|^2 e^{-|y|}\, dy\right) \\ 
 & - c \bigg( (f,Q_c)^2 + (f, |y|^2Q_c)^2 + (f,yQ_c)^2 + (f, \ii DQ_c)^2 \\
 & \quad \quad + (f, \ii D(DQ_c))^2 + (f, \ii \nabla Q_c)^2 \bigg).
 \end{aligned}
 \end{equation}
 \end{prop}

\section{Setting up the Bootstrap}

In this section, we are going to define the set of initial conditions that we consider for our analysis. 
These data are suitable perturbations of the scaled self-similar soliton profile defined in Proposition \ref{thm:profile} and are going to form singularities in finite time. 
In what follows, we also discuss how to estimate the solution in such a way that it retains its self-similar structure along the evolution.
Finally, we are going to show that there exists a finite time when the scaling parameter $\lambda(t)$, related to the size of the solution, goes to zero, thus exhibiting the formation of a singularity.
The idea is to find a set that is almost invariant under the dynamics of equation \eqref{eq:main} and to use a bootstrap argument to show that the solution remains inside this set until the scaling parameter becomes zero and consequently the solution blows up. \newline
We use the quantity $\rho_3$ and the profile $Q_{b}$ defined in Proposition \ref{thm:profile}, and the constant $\Gamma_b$ defined in \eqref{eq:Gammab}. Moreover, we denote by 
\begin{equation*}
 \tilde{s}_c = \frac{d}{2} - \frac{1}{\tilde{\sigma}_c},
\end{equation*}
 where $\tilde{\sigma}_c$ is determined in Proposition \ref{thm:profile}. 
In the next definition and subsequent propositions, it will be useful to write the assumptions on $\sigma_1$ by exploiting identity \eqref{eq:sc}. This means that every time a condition on $s_c$ is found, such as for instance $0 < s_c < S_c$ for some constant $S_c$, it should be interpreted as a condition on the exponent $\sigma_1 < \sigma_1(S_c)$ where $\sigma_1(S_c)$ is given by 
 \begin{equation*}
 \sigma_1(S_c) = \frac{2}{d - 2S_c}.
 \end{equation*}

\begin{defin}\label{def:deco} 
Let $0 < s_c < \tilde{s}_c$ and $0 < \rho < \rho_3$. We define the set $\mathcal{O} \subset H^1(\R^d)$ as the family of all functions $\phi \in H^1$ such that there exist $\lambda_0, b_0 >0$, $x_0 \in \R^d$, $\gamma_0 \in \R$ and $\xi_0 \in H^1(\R^d)$ such that
\begin{equation}\label{eq:in_deco} 
\phi(x) = \lambda_0^{-\frac{1}{\sigma_1}} \left(Q_{b_0} \left(\frac{x-x_0}{\lambda_0} \right) + \xi_0 \left(\frac{x-x_0}{\lambda_0} \right)\right) e^{\ii \gamma_0},
\end{equation}
with the following conditions: 
there exists $0 < \nu(\rho) \ll 1$ such that
\begin{equation}\label{eq:inB0}
 \Gamma_{b_0}^{1 + \nu^{10}} \leq s_c \leq \Gamma_{b_0}^{1 - \nu^{10}}, 
\end{equation}
the scaling parameter satisfies
\begin{equation}\label{eq:initL}
    0 < \lambda_0 < \Gamma_{b_0}^{100},
\end{equation}
the initial momentum and energy are bounded as
\begin{equation}\label{eq:initEP}
 \lambda^{2 - 2s_c} \left|E[\phi]\right| + \lambda^{1 - 2s_c} \left|P[\phi]\right|< \Gamma_{b_0}^{50}, 
\end{equation}
there exists $s= s(\sigma_1, \sigma_2)$ in the interval
\begin{equation}\label{eq:s_cond_in}
 s_c < s < \min \left(\frac{d\sigma_1}{2\sigma_1 + 2}, \frac{d\sigma_2}{2\sigma_2 + 2}, \frac{1}{2} \right)
\end{equation}
such that the following inequality is verified
\begin{equation}\label{eq:initHs}
 \int ||\nabla|^{s} \xi_0|^2 + |\nabla \xi_0|^2 + |\xi_0|^2 e^{-|y|} \, dy < \Gamma_{b_0}^{1 - \nu},
\end{equation}
where 
\begin{equation*}
 y = \frac{x - x_0}{\lambda}.
\end{equation*}
Finally, the remainder $\xi_0$ satisfies also the orthogonality conditions
\begin{equation}\label{eq:inOrtho}
	(\xi_0, |y|^2 Q_{b_0}) = (\xi_0,y Q_{b_0}) = (\xi_0, \ii \Lambda(\Lambda Q_{b_0}) ) = ( \xi_0, \ii \Lambda Q_{b_0}) = 0.
\end{equation}
\end{defin}

In this definition, the constant $\nu=\nu(\rho)$ is chosen so that 
\begin{equation}\label{eq:crhoNU}
 \Gamma_{b_0}^{1 - c\rho} \leq \Gamma_{b_0}^{1 - \nu^{50}},
\end{equation}
where $c>0$ is the constant in \eqref{eq:gammaB1}. Note that from the definition of $\Gamma_b$ in \eqref{eq:gammaB1} and from condition \eqref{eq:inB0}, it follows that 
\begin{equation*}
 s_c \sim \Gamma_b \sim e^{\frac{-(1 + c\rho)\pi}{b}}
\end{equation*}
that is 
\begin{equation*}
 b_0 \sim b^* = - \frac{\pi (1 + \delta)}{\ln(s_c)}
\end{equation*}
namely $b_0$ is chosen to be close to the value $b^*(s_c)>0$ defined in \eqref{eq:s_c}.
The constant $s = s(\sigma_1, \sigma_2)$ is chosen so that the
$L^{2\sigma_1 + 2}$ and $L^{2\sigma_2 + 2}$ norms of $\xi$ may be controlled by interpolating between $\dot H^s$ and $\dot H^1$, see \eqref{eq:nnlin} and \eqref{eq:lowerS} below. 
Finally, we recall that the set $\mathcal{O}$ is non-empty, see \cite[Remark $2.10$]{MeRaSz10}. 
We remark that this is the same set defined in \cite{MeRaSz10}, whose evolutions develop a singularity in finite time. In this work, we are going to prove that the same initial data produce singular solutions also under the dissipative dynamics \eqref{eq:main}. Thus, our result implies that the nonlinear damping is not able to regularize the dynamics and to prevent the formation of singularities, in the cases under our consideration. This is not in contradiction with the global well-posedness result proven in \cite{AnCaSp15}, where the condition $\eta\geq\min(\sigma_1, \sqrt{\sigma_1})$ was needed. In fact, here we prove our main result under a smallness assumption on $\eta$, namely we require
\begin{equation}\label{eq:initEta}
\eta \leq s_c^3.
\end{equation}
Moreover, in Section \ref{sec:concl} we are going to provide an alternative definition of the set $\mathcal{O}$, in the case $\sigma_2 < \sigma_1$. 
In particular, hypothesis \eqref{eq:initEta} will no longer be needed, even though the set will depend on $\sigma_1$, $\sigma_2$ and $\eta$. \newline
Now, let us consider an initial datum $\psi_0 \in \mathcal{O}$ and let $\psi \in C([0,T_{max}),H^1(\R^d)) $ be the corresponding solution to \eqref{eq:main}, where $T_{max} \leq \infty$ is its maximal time of existence. By continuity of the solution, there exists a time interval $[0,T_1)$, with $T_1 \leq T_{max}$, where the inequalities in Definition \ref{def:deco} are still valid but with slightly larger bounds. Moreover, by standard perturbation techniques, we can also preserve the orthogonal conditions in \eqref{eq:inOrtho} for any $t \in [0,T_1)$ by modulational analysis, see for instance \cite[Lemma $2$]{MeRa03} where this was proved for the Hamiltonian dynamics. 
The proof of a similar property for solutions to the dissipative dynamics \eqref{eq:main} follows straightforwardly.

\begin{prop}\label{prp:deco}
 There exists $0<T_1 \leq T_{max}$, $b,\lambda,\gamma \in C^1([0,T_1),\R)$, $x \in C^1([0,T_1),\R^d)$ and $\xi \in C([0,T_1),H^1(\R^d))$ such that for all $t \in [0,T_1)$, the solution $\psi$ to \eqref{eq:main} may be decomposed as
\begin{equation}\label{eq:deco1}
\psi(t,x) = \lambda(t)^{-\frac{1}{\sigma_1}}\left(Q_{b(t)}(y) + \xi(t,y )\right) e^{\ii \gamma(t)},
\end{equation}
with
\begin{equation}\label{eq:y}
 y = \frac{x - x(t)}{\lambda(t)},
\end{equation}
where $\xi$ satisfies the following orthogonal conditions:
\begin{equation}\label{eq:ortho4}
	(\xi(t), |y|^2 Q_{b(t)}) = (\xi(t),y Q_{b(t)}) = (\xi(t), \ii \Lambda(\Lambda Q_{b(t)}) ) = ( \xi(t), \ii \Lambda Q_{b(t)}) = 0,
\end{equation}
and we have
\begin{align}\label{eq:conb}
& \Gamma_{b(t)}^{1 + \nu^2} \leq s_c \leq \Gamma_{b(t)}^{1 - \nu^2}, \\ \label{eq:conL}
& 0 \leq \lambda(t) \leq \Gamma_{b(t)}^{10}, \\ \label{eq:conEP}
& \lambda^{2 - 2s_c} |E[\psi(t)]| + \lambda^{1 - 2s_c}(t) |P[\psi(t)]| \leq \Gamma_{b(t)}^2, \\ \label{eq:conHs}
& \int ||\nabla|^{s} \xi(t)|^2 dy\leq \Gamma_{b(t)}^{1 - 50\nu}, \\ \label{eq:conxi} 
& \int |\nabla \xi(t)|^2 + |\xi(t)|^2 e^{-|y|} dy \leq \Gamma_{b(t)}^{1 - 20\nu}.
\end{align}
\end{prop}
We notice that \eqref{eq:conb} and the smallness of $\nu$ and $s_c$ imply that $\Gamma_{b(t)} < 1$ for any $t \in [0,T_1)$. 
Our main goal is to prove that the solution remains in this self-similar regime until $T_{max}$, that is the decomposition of $\psi$, along with the properties of the modulation parameters as stated in Proposition \ref{prp:deco} are valid until the maximal time of existence. This is achieved by using a bootstrap argument. 
We show that the bounds in Proposition \ref{prp:deco} can be improved and thus the self-similar regime may be extended in time. This amounts to finding a dynamical trapping of the parameter $b$ to improve \eqref{eq:conb} and \eqref{eq:conxi}, to find a suitable differential equation for $\lambda$ to improve \eqref{eq:conL}, and to prove that the energy and the momentum of the solution remain small enough to improve \eqref{eq:conEP}. The trapping of the control parameter $b$ will be achieved in Section \ref{ch:parameters} by finding a lower bound for $\dot b$ using a virial-type argument, and an upper bound with a monotonicity formula. The equation satisfied by the scaling parameter $\lambda$ will be a direct consequence of the dynamical trapping of $b(t)$ around $b^*$, see Section \ref{ch:bootstrap}. Finally, the controls on $E$ and $P$ will be obtained in Section \ref{ch:bootstrap} as a consequence of the choice of the damping parameter $\eta$ in \eqref{eq:initEta}. In fact, we will show the following bootstrap result.

\begin{prop} \label{thm:boot}
There exists $s_c^*>0$, such that for any $s_c < s_c^*$, there exist $s_{max}(s_c)>s_c$, and $\nu^*(s_c) >0$ such that for any $s_c < s<s_{max}(s_c)$ and $0< \nu < \nu^*$ and for any $t \in [0,T_1),$ the following inequalities are true: 
\begin{align}\label{eq:butb}
&\Gamma_{b(t)}^{1 + \nu^4} \leq s_c \leq \Gamma_{b(t)}^{1 - \nu^4}, \\ \label{eq:butL}
&0 \leq \lambda(t) \leq \Gamma_{b(t)}^{20}, \\ \label{eq:butEP}
& \lambda^{2 - 2s_c} |E[\psi(t)]| + \lambda^{1 - 2s_c} |P[\psi(t)]| \leq \Gamma_{b(t)}^{3 - 10\nu}, \\ \label{eq:butHs}
&\int ||\nabla|^{s} \xi(t)|^2 \leq \Gamma_{b(t)}^{1 - 45\nu}, \\ \label{eq:butH1}
& \int |\nabla \xi(t)|^2 + |\xi(t)|^2 e^{-|y|} dy \leq \Gamma_{b(t)}^{1 - 10\nu}.
\end{align}
\end{prop}

As a consequence of this proposition, we see that the solution satisfies improved bounds with respect to those stated in Proposition \ref{prp:deco}. 
A standard continuity argument then implies that the same bounds are satisfied in the whole interval $[0, T_{max})$.
\newline
Let us remark that the same dynamical trapping argument to show self-similar blow-up was already exploited in \cite{MeRaSz10}. In particular, this implies that a similar argument to \cite{MeRaSz10} works also in our case for the damped NLS, under the assumptions of Theorem \ref{thm:mainBU}. On the other hand, dealing with the dissipative dynamics \eqref{eq:main} entails new mathematical difficulties with respect to \cite{MeRaSz10}. First of all, we notice that the self-similar profile $Q_b$, given by equation \eqref{eq:Qb}, does not determine an approximate solution to \eqref{eq:main}, not even close to the collapsing core. In particular, this implies that the equation for the perturbation $\xi$ bears a forcing term of order one, depending on $Q_b$. In the case under our consideration here, we can see that the forcing term produces an error that becomes non-negligible on a range of times of order $\left(\eta \lambda^{2 - 2\frac{\sigma_2}{\sigma_1}}\right)^{-1} $.
\newline
We overcome this difficulty by choosing the parameters (and in particular, the damping coefficient in the case $\sigma_1=\sigma_2$) in such a way that $T_{max} \ll \left(\eta \lambda^{2 - 2\frac{\sigma_2}{\sigma_1}}\right)^{-1}$.
\newline
The second mathematical difficulty induced by the dissipative dynamics is that the global quantities, such as total momentum and total energy, are not conserved along the flow of \eqref{eq:main}, see \eqref{eq:momGrow} and \eqref{eq:en_grow}, respectively. Consequently, while the bootstrap condition \eqref{eq:conEP} is straightforwardly satisfied in the Hamiltonian case, a fundamental step in our analysis would be to control the time evolution of these quantities.
\newline
In Section \ref{sec:concl} we will also show that in the case $\sigma_2 < \sigma_1$, it is possible to deal with these issues by choosing the scaling parameter $\lambda$ small enough and depending on $\sigma_1$, $\sigma_2$ and $\eta$ so that the initial condition is very close to the blow-up point and the damping term is not strong enough to force the solution out of the self-similar regime before the blow-up time. \newline
As the last step, we will show that inside the self-similar regime, the scaling parameter $\lambda(t)$ goes to zero. 
\newline
In particular, the following Corollary can be inferred from Proposition \ref{prp:deco}.
\begin{cor}
For any $s_c < s_c^*$ where $s_c^*$ is defined in Proposition \ref{thm:boot}, the bounds in Proposition \ref{prp:deco} are valid for any $t \in [0, T_{max})$. Moreover, there exists a constant $ 0< C = C(\nu,s_c)\ll 1$ such that for any $t \in [0,T_{max})$, 
\begin{equation}\label{eq:blowup_rate}
 \lambda_0^2 - 2(1 + C)t \leq\lambda^2(t) \leq \lambda_0^2 - 2(1 -C)t.
\end{equation} 
\end{cor}
As a consequence of the corollary above, there exists $T_{max}(\nu,s_c,\lambda_0) < \infty$ such that $\lambda \to 0$ as $t \to T_{max}(\nu,s_c,\lambda_0)$ and $\lambda$ behaves like 
\begin{equation*}
 \lambda^2(t) \sim T_{max} - t.
\end{equation*}
Moreover, the estimate \eqref{eq:blowup_rate} also provides a blow-up rate for the $L^2$-norm of the gradient of the solution, as we have
\begin{equation*}
 \big\| \nabla \psi(t) \big\|_{L^2}^2 = \lambda^{2(s_c - 1)}(t) \big\| \nabla (Q_{b(t)} + \xi(t)) \big\|_{L^2}^2 
 \sim\left(T_{max}-t\right)^{-(1-s_c)}.
\end{equation*}

\section{Estimates on the Modulation Parameters} \label{ch:parameters}

In this section, we start our analysis of solutions emanating from initial conditions in the set $\mathcal{O}$, defined in Definition \ref{def:deco}. Let $\psi_0 \in \mathcal{O}$, we denote by $\psi \in C([0,T_{max}),H^1(\R^d))$ its corresponding maximal solution, with $T_{max} \leq \infty$. Then we decompose the solution as the sum of a soliton core and remainder as in \eqref{eq:deco1}. 
By Proposition \ref{prp:deco}, the bounds \eqref{eq:conb} - \eqref{eq:conxi} are satisfied for all $t\in[0, T_1]$, together with the four orthogonal conditions in \eqref{eq:ortho4}. \newline 
The purpose of this section is to obtain preliminary estimates on the parameters that are required to prove the dynamical trapping of $b$ in the next section. A crucial step will be to use the coercivity property stated in Proposition \ref{prp:quadForm}. Even though this proposition is related to the mass critical case, we will use the smallness of $s_c$ and property \eqref{eq:Lambda-D} to obtain a similar coercivity property in the slightly super-critical case. We will then show that the negative part of the right-hand side of 
\eqref{eq:almost_coerc} is controlled by the positive part. 
When computing the quadratic form \eqref{eq:almost_coerc} on the perturbation $\xi$, we see that four of the negative terms appearing in are small enough because of the orthogonality conditions \eqref{eq:ortho4}, the smallness of $s_c$ and property \eqref{eq:Lambda-D}. We now show how to control the two remaining negative terms.
\newline
By plugging the decomposition \eqref{eq:deco1} into equation \eqref{eq:main}, we obtain the following equation for the perturbation $\xi$,
\begin{equation} \label{eq:xi1}
	\begin{split}
		& \ii \partial_\tau \xi + \mathcal{L}\xi + \ii ( \dot{ b }- \beta s_c) \partial_b Q_b 
- \ii \left( \frac{\dot{ \lambda} }{\lambda} + b \right) \Lambda (Q_b + \xi)
- \ii\frac{\dot{ x}}{\lambda} \cdot \nabla (Q_b + \xi) \\
 & - (\dot{ \gamma } -1) (Q_b + \xi)
	 + \ii \eta \lambda^{2 - 2\frac{\sigma_2}{\sigma_1}} |Q_b + \xi|^{2\sigma_2}(Q_b + \xi) + R(\xi) - \Psi_b = 0,
	\end{split}
\end{equation}
where the scaled time $\tau$ is defined according to
\begin{equation}\label{eq:scaled_t}
 \tau(t) = \int_0^t \frac{1}{\lambda^2(v)}dv,
\end{equation}
and we recall that
\begin{equation*}
\begin{aligned}
 \Lambda f = \frac{1}{\sigma_1} f + y \cdot \nabla f.
\end{aligned}
\end{equation*}
In \eqref{eq:xi1} the dot denotes the derivative with respect to $\tau$ and the space derivatives are intended to be performed with respect to the scaled space variable $y$ where 
\begin{equation}\label{eq:scaled_s}
 y = \frac{x - x(t)}{\lambda(t)}.
\end{equation}
 Moreover, we used the notation $\mathcal L$ for the linearized operator around $Q_b$,
\begin{align}\label{eq:linQb}
\mathcal{L}\xi = \Delta \xi - \xi + \ii b \Lambda \xi + 2\sigma_1 Re(\xi \bar{Q}_b) |Q_b|^{2\sigma_1 - 2} Q_b + |Q_b|^{2\sigma_1} \xi,
\end{align}
and $R$ contains all nonlinear terms in $\xi$,
\begin{equation}\label{eq:Qb_remt}
R(\xi) = |Q_b + \xi|^{2\sigma_1}(Q_b + \xi) - |Q_b|^{2\sigma_1} Q_b - 2\sigma_1 Re(\xi \bar{Q}_b) |Q_b|^{2\sigma_1 - 2} Q_b 
- |Q_b|^{2\sigma_1} \xi,
\end{equation}
whereas $\Psi_b$ is defined in \eqref{eq:psi+phi}.\newline
We notice that the damping term yields a forcing in the equation for $\xi$, given by
\begin{equation}\label{eq:forcing}
 \ii \eta \lambda^{2 - 2\frac{\sigma_2}{\sigma_1}} |Q_b|^{2\sigma_2}Q_b
\end{equation}
depending on $Q_b$, $\eta$ and $\lambda$. 
One of the main goals in our subsequent analysis is to exploit the smallness of the product $\eta \lambda^{2 - 2\frac{\sigma_2}{\sigma_1}}$ in order to obtain a suitable control on the nonlinear damping term so that the self-similar regime is maintained along the evolution.
 This task is achieved by the selection of $\eta$ for instance as in \eqref{eq:initEta} or by the smallness of $\lambda$ for $\sigma_2 < \sigma_1$, see Section \ref{sec:concl}. 
Let us notice that the assumption $\sigma_*<\sigma_2\leq\sigma_{2, max}$, see \eqref{eq:sigma2} and \eqref{eq:sigma2_max}, implies that
\begin{equation*}
0\leq2-2\frac{\sigma_2}{\sigma_1}<2(1-s_c)
\end{equation*}
for $d\leq3$, while 
\begin{equation*}
0<4\frac{(d-2)\sigma_1-1}{(d-2)\sigma_1}\leq2-2\frac{\sigma_2}{\sigma_1}<2(1-s_c),
\end{equation*}
for $d\geq4$.
\newline
In what follows, we will use the following estimates on scalar products.

\begin{lem}
 For any $t \in [0,T_1)$, any polynomial $P(y)$ and any integers $k \in \{0,1,2,3\}$ and $n\in \{0,1\}$, we have
\begin{equation}\label{eq:prod1}
	\left| \left( \xi, P \frac{d^k}{dy^k} Q_b \right)\right| \lesssim \left( \int |\xi|^2e^{-|y|} dy\right)^\frac{1}{2},
\end{equation}
\begin{equation}\label{eq:prod2}
	\left| \left( \xi, P \frac{d^k}{dy^k} \partial_b Q_b \right)\right| \lesssim \left(\int |\nabla \xi|^2 + |\xi|^2e^{-|y|} dy\right)^{\frac{1}2},
\end{equation}
\begin{equation}\label{eq:prod3}
	\left| \left( \frac{d^k}{dy^k} Q_b , P \frac{d^n}{dy^n} \Psi_b \right)\right| + \left| \left( \xi , P \frac{d^n}{dy^n} \Psi_b \right)\right| \lesssim \Gamma_b^{1 - \nu},
\end{equation}
\begin{equation}\label{eq:prod4}
	\left| \left( \partial_b Q_b, P \frac{d^k}{dy^k} Q_b \right)\right| + \left| \left(|y|^2 Q_b, P \frac{d^k}{dy^k} Q_b \right)\right| \lesssim 1.
\end{equation}
\end{lem}

These inequalities have been proven in \cite[Lemma $4$]{MeRa03}. They follow from the Cauchy-Schwarz inequality and the uniform estimate \eqref{eq:polEst}. 
In the next two Lemmas \ref{lem:neg1} and \ref{lem:neg2} we prove the smallness of the scalar products $(\xi(t), Q_{b(t)})$ and $(\xi, \ii\nabla Q_{b(t)})$, so to conclude our control of the negative terms appearing in \eqref{eq:almost_coerc} up to additional terms controlled by the smallness of $s_c$. This will be achieved by exploiting the balance laws satisfied by the total momentum and energy \eqref{eq:momGrow} and \eqref{eq:en_grow}, respectively, together with the bounds \eqref{eq:conEP}.

\begin{lem}\label{lem:neg1}
There exists $\nu_1 >0$ such that for any $\nu < \nu_1$ and $t \in [0,T_1)$, we have
\begin{equation} \label{eq:enest1}
|(\xi(\cdot, t), Q_{b(t)})| \lesssim \int |\xi|^2e^{-|y|} + |\nabla \xi|^2 \, dy + \Gamma_{b(t)}^{1 -2\nu}.
\end{equation}
\end{lem}
Let us remark that, by combining \eqref{eq:enest1} with the bound \eqref{eq:conxi}, we would obtain the following estimate
\begin{equation}\label{eq:enest}
 |(\xi(\cdot, t), Q_{b(t)})| \lesssim \Gamma_{b(t)}^{1 -20 \nu}.
\end{equation}
Although this rougher bound is sufficient to show the coercivity of the linearized operator, the estimate \eqref{eq:enest1} will be crucial to prove Proposition \ref{thm:boot}.

\begin{proof}
We plug the decomposition of the solution \eqref{eq:deco1} into the energy equation to obtain 
\begin{equation}\label{eq:enLong}
\begin{aligned}
 2\lambda^{2 - 2s_c} E[\psi] &= 2E[Q_b] + \| \nabla \xi\|_{L^2}^2 - 2\left(\Delta Q_b,\xi \right) \\ 
 &- \frac{1}{\sigma_1 + 1} \int |Q_b + \xi|^{2\sigma_1 +2} - |Q_b|^{2\sigma_1 + 2} \, dy.
\end{aligned}
\end{equation}
By using equation \eqref{eq:Qb} satisfied by $Q_b$ in the expression above, we get 
\begin{equation*}
		\begin{split}
				2\lambda^{2 - 2s_c} E[\psi] &= 2E[Q_b] + \| \nabla \xi\|_{L^2}^2 \\& + 2\left(\ii s_c \beta \partial_b Q_b - Q_b + \ii b \Lambda Q_b + \Psi_b,\xi \right) - \int R^{(2)}(\xi) \, dy,
		\end{split}
	\end{equation*}
 where $R^{(2)}(\xi)$ is defined by
\begin{equation}\label{eq:R2}
R^{(2)}(\xi) = \frac{1}{\sigma_1 + 1} \left(|Q_b + \xi|^{2\sigma_1 + 2} - |Q_b|^{2\sigma_1 + 2} - (2\sigma_1 + 2)|Q_b|^{2\sigma_1} Re(Q_b \bar{\xi})\right).
\end{equation}
Equivalently, we write the identity above as
 \begin{equation}\label{eq:entok}
		\begin{split}
				2 (Q_b,\xi) &= - 2\lambda^{2 - 2s_c} E[\psi] + 2E[Q_b] + \| \nabla \xi\|_{L^2}^2 \\
 & + 2\left(\ii s_c \beta \partial_b Q_b + \ii b \Lambda Q_b + \Psi_b,\xi \right) - \int R^{(2)}(\xi) dy.
		\end{split}
	\end{equation} 
Now we bound the terms on the right-hand side of \eqref{eq:entok}. For the first term, we use the control on the energy \eqref{eq:conEP}
\begin{equation*}
 \left|- 2\lambda^{2 - 2s_c} E[\psi] \right| \leq 2 \Gamma_b^2.
\end{equation*}
For the second term, we use the properties of $Q_b$ listed in \eqref{eq:Qbprop}, inequality \eqref{eq:crhoNU} and the control \eqref{eq:conb} on $s_c$ to obtain that
\begin{equation*}
 \big| E[Q_b] \big| \leq \Gamma_b^{1 - c\rho} + s_c \leq \Gamma_b^{1 - \nu^{50}} + \Gamma_b^{1 - \nu^{2}}.
\end{equation*}
For the fourth term, we use again the control \eqref{eq:conb} on $s_c$, estimates \eqref{eq:prod2}, \eqref{eq:prod3} and \eqref{eq:conxi} to get
\begin{equation*}
\begin{aligned}
 \left| s_c\left(\ii \beta \partial_b Q_b, \xi\right) + \left(\Psi_b, \xi \right) \right| &\lesssim \left(s_c \left( \int |\nabla \xi|^2 + |\xi|^2e^{-|y|} dy\right)^\frac{1}{2} + 
\Gamma_b^{1 - \nu}\right) \\
 &\lesssim \left(\Gamma_b^{1 - \nu^2}\Gamma_{b}^{\frac{1}{2} - 10\nu} + \Gamma_b^{1 - \nu} \right).
\end{aligned}
\end{equation*}

Moreover, we also observe that $\left( \ii b \Lambda Q_b ,\xi \right) = 0$ from \eqref{eq:ortho4}. Finally, we use estimates \eqref{eq:prod1} and \eqref{eq:conxi} to control the remainder term $R^{(2)}(\xi)$ defined in \eqref{eq:R2} as follows,
\begin{equation*}
 \left|\int R^{(2)}(\xi) dy \right| \lesssim \int | \nabla \xi|^2 + |\xi|^2e^{-|y|} dy + \| \xi\|_{L^{2\sigma_1 + 2}}^{2\sigma_1 + 2}.
\end{equation*}
In order to bound the $L^{2\sigma_1 + 2}$-norm of $\xi$, we observe that since
\begin{equation*}
 s_c < s < \frac{d}{2} - \frac{d}{2\sigma_1 + 2} = \frac{d\sigma_1}{2\sigma_1 + 2} = s(\sigma_1) 
\end{equation*}
 then by Sobolev embedding and subsequent interpolation between $H^1(\R^d)$ and $H^s(\R^d)$, we have
\begin{equation}\label{eq:nnlin}
\|\xi\|_{L^{2(\sigma_1+1)}}^{2(\sigma_1+1)}\lesssim
\||\nabla|^{s(\sigma_1)}\xi\|_{L^2}^{2(\sigma_1+1)}
\leq \| \xi\|_{\dot H^1}^{2\theta(\sigma_1 + 1)} \| \xi\|_{\dot H^s}^{2(1 - \theta)(\sigma_1 + 1)},
\end{equation}
for some $\theta=\theta(s)\in(0, 1)$, where $s$ is determined in Definition \ref{def:deco}, see \eqref{eq:s_cond_in}.
Now by using the controls \eqref{eq:conxi} and \eqref{eq:conHs} we obtain 
\begin{equation*}
\begin{aligned}
 \|\nabla \xi\|_{\dot H^1}^{2\theta(\sigma_1 + 1)} \|\nabla \xi\|_{\dot H^s}^{2(1 - \theta)(\sigma_1 + 1)} &\leq \Gamma_b^{(1 - 20\nu)\theta(\sigma_1 + 1)} \Gamma_b^{(1 - 50\nu)(1 -\theta)(\sigma_1 + 1)} \\
 &\leq \Gamma_b^{(1 - 50\nu)(\sigma_1 + 1) + 30 \nu \theta (1 + \sigma_1)}. 
\end{aligned}
\end{equation*}
By collecting everything together, we obtain that
\begin{equation*}
 \begin{aligned}
 \left|(\xi, Q_b)\right| &\lesssim \Gamma_b^2 + \Gamma_b^{1 - \nu^{50}} + \Gamma_b^{1 - \nu^{2}} + \int | \nabla \xi|^2 + |\xi|^2e^{-|y|} \, dy \\
 &+ \Gamma_b^{1 - \nu^2} \Gamma_{b}^{\frac{1}{2} - 10\nu} + \Gamma_b^{1 - \nu} + \Gamma_b^{(1 - 50\nu)(\sigma_1 + 1) + 30 \nu \theta (1 + \sigma_1)}.
 \end{aligned}
\end{equation*}
By choosing $\nu$ small enough, we have thus obtained estimate \eqref{eq:enest1}.
\end{proof}

We also exploit the equation of the momentum \eqref{eq:momGrow} to derive a bound on $(\xi(\cdot, t), \ii \nabla Q_{b(t)})$, as in the following lemma.

\begin{lem}\label{lem:neg2}
There exists $\nu^{(1)} >0$ such that for any $0< \nu < \nu^{(1)}$ and any $t \in [0,T_1)$, we have
 \begin{equation}\label{eq:momest}
|(\xi(\cdot, t), \ii \nabla Q_{b(t)})| \leq \Gamma_{b(t)}^{1 - 50 \nu}.
\end{equation}
\end{lem}

\begin{proof}
We plug decomposition \eqref{eq:deco1} into the momentum \eqref{eq:momGrow} to obtain
\begin{equation*}
	 2 (\xi, \ii \nabla Q_b)= - \lambda^{1 -2s_c} P[\psi] + P[Q_b] + P[\xi].
\end{equation*}
The first term on the right-hand side is bounded by \eqref{eq:conEP},
\begin{equation*}
 \left|- \lambda^{1 -2s_c} P[\psi] \right| \leq\Gamma_b^2.
\end{equation*}
Next, we observe that by the properties of $Q_b$, we have $P[Q_b] = 0$, see \eqref{eq:Qbprop}. For the last term, since $s < \frac{1}{2}$, we use \eqref{eq:s_cond_in} to estimate
\begin{equation*}
\begin{aligned}
 \left| P[\xi] \right| \lesssim \| \xi\|_{\dot{H}^{\frac{1}{2}}}^2 \lesssim \| \xi\|_{\dot H^1}^{2\theta} \| \xi\|_{\dot H^s}^{2(1 - \theta)}
\end{aligned}
\end{equation*}
where 
\begin{equation*}
 \theta = \frac{2- 4s}{2-s}>0.
\end{equation*}
Now we use estimates \eqref{eq:conxi} and \eqref{eq:conHs} to obtain
\begin{equation*}
 \| \xi\|_{\dot H^1}^{2\theta} \| \xi\|_{\dot H^s}^{2(1 - \theta)} \leq \Gamma_b^{(1 - 20\nu)\theta} \Gamma_b^{(1 - 50\nu)(1 -\theta)} = \Gamma_b^{(1 - 50\nu) + 30 \nu \theta}.
\end{equation*}
By combining all previous estimates, we get
\begin{equation*}
 \left| (\xi, \ii \nabla Q_b)\right| \lesssim \Gamma_b^2 + \Gamma_b^{(1 - 50\nu) + 30 \nu \theta}. 
\end{equation*}
Inequality \eqref{eq:momest} thus follows by choosing $\nu$ sufficiently small.
\end{proof}

Our next step is to obtain suitable estimates on the modulational parameters defined in the decomposition \eqref{eq:deco1}. This is accomplished by exploiting the equation \eqref{eq:xi1} satisfied by the remainder $\xi$, and the orthogonality conditions listed in \eqref{eq:ortho4}. 

\begin{lem}\label{lemma:mod_par}
For any $t \in [0,T_1)$, we have
\begin{equation}\label{eq:bLest}
\left| \frac{\dot{ \lambda}(\tau)}{\lambda(\tau)} + b(\tau) \right| + |\dot{b}(\tau)| \lesssim \Gamma_{b(\tau)}^{1 - 20 \nu},
\end{equation}
and 
\begin{equation} \label{eq:gxest}
	\begin{split}
		&\left| (\dot{ \gamma}(\tau) - 1) - \frac{1}{\| \Lambda Q_{b(\tau)}\|_{L^2}^2} (\xi(\tau), \mathcal{L} \Lambda( \Lambda Q_{b(\tau)})\right| + \left| \frac{\dot{ x}(\tau)}{\lambda(\tau)} \right| \\ &\lesssim \delta_2 \left( \int |\nabla \xi(\tau)|^2 + |\xi(\tau)|^2 e^{-|y|} dy \right)^\frac{1}{2}   + \Gamma_{b(\tau)}^{1 -20\nu},
	\end{split}
\end{equation}
where $\delta_2=\delta_2(s_c)>0$. Moreover, $\delta_2(s_c)\to0$, as $s_c\to0$.
\end{lem}

\begin{proof}
The lemma is proved by taking the scalar product of equation \eqref{eq:xi1} with suitable terms that allow us to exploit the orthogonality conditions \eqref{eq:ortho4}. 
The same approach was already used in \cite{MeRaSz10}, see Lemma 3.1 and Proposition 3.3 therein, see also \cite[Appendix A]{Ra05}, to study the undamped dynamics. For this reason we write equation \eqref{eq:xi1} as
\begin{equation}\label{eq:tokxi}
	U(\xi) + \ii \eta \lambda^{2 - 2\frac{\sigma_2}{\sigma_1}}|Q_b + \xi|^{2\sigma_2}(Q_b+\xi) = 0,
\end{equation}
so that, in what follows, we exploit the analysis already developed in \cite{MeRaSz10}.
\newline
Let us first consider the bound on $|\dot b|$. We take the scalar product of \eqref{eq:tokxi} with $\Lambda Q_b$. Following the computations in Appendix \hyperref[sec:appenA]{A}, we use estimates
\eqref{eq:prod1}, \eqref{eq:enest}, \eqref{eq:momest}, \eqref{eq:conHs} and \eqref{eq:conxi} to obtain
\begin{equation*}
\begin{aligned}
 | \dot{ b}| &\lesssim \int |\nabla \xi|^2 + |\xi|^2e^{-|y|} \, dy + \Gamma_b^{1 - 11 \nu} + \eta \lambda^{2 - 2\frac{\sigma_2}{\sigma_1}} \left|\left( \ii |Q_b + \xi|^{2\sigma_2}(Q_b+\xi) , \Lambda Q_b\right) \right|.
\end{aligned}
\end{equation*}
Similarly, we obtain the estimate for other parameters 
\begin{equation*}
 \begin{aligned}
 \left| \frac{\dot{ \lambda}}{\lambda} + b \right| &\lesssim \int |\nabla \xi|^2 + |\xi|^2e^{-|y|} dy + \Gamma_b^{1 - 11 \nu} + \eta \lambda^{2 - 2\frac{\sigma_2}{\sigma_1}} \left|\left( |Q_b + \xi|^{2\sigma_2}(Q_b+\xi) ,|y|^2 Q_b\right) \right|.
 \end{aligned}
\end{equation*}
and
\begin{equation*} 
	\begin{split}
		\left| ( \dot{ \gamma }- 1) - \frac{1}{\| \Lambda Q_b\|_{L^2}^2} (\xi, \mathcal{L} \Lambda( \Lambda Q_b)\right| + \left| \frac{\dot{ x}}{\lambda} \right| & \lesssim \delta_2 \left( \int |\nabla \xi|^2 + |\xi|^2 e^{-|y|} \, dy\right)^\frac{1}{2} \\ 
 &+ \int |\nabla \xi|^2 \, dy + \Gamma_b^{1 -11\nu}\\
 & + \eta \lambda^{2 - 2\frac{\sigma_2}{\sigma_1}} \bigg|\left( |Q_b + \xi|^{2\sigma_2}(Q_b+\xi) ,y Q_b\right) \\ 
 &+ \left( \ii |Q_b + \xi|^{2\sigma_2}(Q_b+\xi) ,\Lambda( \Lambda Q_b)\right) \bigg|.
	\end{split}
\end{equation*}
Let us now control the contributions coming from the damping term. We write 
\begin{equation*}
	|Q_b + \xi|^{2\sigma_2}(Q_b+\xi) = |Q_b|^{2\sigma_2}Q_b + R^{(1)}(\xi).
\end{equation*}
We use \eqref{eq:prod4} to obtain that
\begin{equation}\label{eq:weknow}
	\left|\left( \ii |Q_b|^{2\sigma_2}Q_b , \Lambda Q_b + \Lambda( \Lambda Q_b) + \ii yQ_b + \ii |y|^2 Q_b \right) \right| \lesssim 1. 
\end{equation}
On the other hand, by using \eqref{eq:prod1} and \eqref{eq:conxi}, we also have that 
\begin{equation*}
 \left|\left(R^{(1)}(\xi), \Lambda Q_b + \Lambda( \Lambda Q_b) + \ii yQ_b + \ii |y|^2 Q_b\right) \right| \lesssim \left(\int |\nabla \xi|^2 \, dy +\int |\xi|^2e^{-|y|} dy\right)^{\frac{1}2}. 
\end{equation*}
 Thus it follows that
 \begin{equation}\label{eq:toketa}
\begin{aligned}
 &\eta\lambda^{2 - 2\frac{\sigma_2}{\sigma_1}}\big| \left( \ii |Q_b + \xi|^{2\sigma_2}(Q_b+\xi) , \Lambda Q_b + \Lambda( \Lambda Q_b) + \ii yQ_b + \ii |y|^2 Q_b \right)\big| \\ 
 &\lesssim \eta \lambda^{2 - 2\frac{\sigma_2}{\sigma_1}} \left( 1 + \int |\nabla \xi|^2 +|\xi|^2e^{-|y|} \, dy \right)^{\frac{1}2} \lesssim \Gamma_b^{3 - 3\nu^2} \Gamma_b^{20 - 20 \frac{\sigma_2}{\sigma_1} } \leq \Gamma_b^2,
\end{aligned}
\end{equation}
where we used the hypothesis on $\eta$ \eqref{eq:initEta}, \eqref{eq:conb} and \eqref{eq:conL}. 
\end{proof}

One can see that the estimates \eqref{eq:gxest} and \eqref{eq:bLest} are the same as those in \cite[Lemma $3.1$]{MeRaSz10} for the undamped case $\eta=0$. This is because our choice of $\eta$ and $\sigma_2 \leq \sigma_1$ imply that the damping term is of lower order with respect to other terms in equation \eqref{eq:xi1}. in particular, the scalar products with the forcing term defined in \eqref{eq:forcing} are well-controlled by the smallness of $\eta \lambda^{2 -2\frac{\sigma_2}{\sigma_1}}$.

\subsection{Local Virial Law}

We will now derive suitable inequalities on $\dot b$ to prove that $b$ is trapped around the value $b^*$ defined in \eqref{eq:s_c}. 
We first obtain a lower bound, see \eqref{eq:locvi} below. This estimate is connected with the local virial law for the remainder $\xi$. For details, see \cite[Section $3$]{MeRa02} and \cite[Section $4$]{MeRa03}.

\begin{lem}\label{lem:vir1}
There exists $s_c^{(2)} > 0$ such that for any $s_c < s_c^{(2)}$ and for any $t \in [0,T_1)$, there exists $C>0$ such that 
\begin{equation} \label{eq:locvi}
\dot{b}(t) \geq C\left( s_c + \int |\nabla \xi(t)|^2 + |\xi(t)|^2 e^{-|y|} \, dy - \Gamma_{b(t)}^{1 - \nu^6}\right).
\end{equation}
\end{lem} 

\begin{proof}
	By taking the scalar product of \eqref{eq:xi1} with $\Lambda Q_b$ we obtain
	\begin{equation*}
\begin{aligned}
 0 &= (\ii \partial_\tau \xi, \Lambda Q_b)+ \dot{b} (\ii \partial_b Q_b, \Lambda Q_b) - \left( \ii \beta s_c \partial_b Q_b + \Psi_b, \Lambda Q_b \right) 
 +(\mathcal{L} \xi + R(\xi), \Lambda Q_b) \\
 & - \left( \ii \left( \frac{\dot{ \lambda} }{\lambda} + b \right) \Lambda Q_b + \ii\frac{\dot{ x}}{\lambda} \cdot \nabla Q_b + (\dot{ \gamma } -1) Q_b, \Lambda Q_b\right)\\
 &- \left( \ii \left( \frac{\dot{ \lambda} }{\lambda} + b \right) \Lambda \xi + (\dot{ \gamma } -1) \xi + \ii\frac{\dot{ x}}{\lambda} \cdot \nabla \xi, \Lambda Q_b\right) \\
 & + \eta \lambda^{2 - 2\frac{\sigma_2}{\sigma_1}}(\ii |Q_b + \xi|^{2\sigma_2}(Q_b + \xi), \Lambda Q_b).
\end{aligned}
\end{equation*}
Notice that only the last term in this equation depends on the damping. The computation involving the other terms will be shown in Appendix \hyperref[sec:appenA]{A} following the exposition of \cite[Proposition $3.3$]{MeRaSz10}. These computations yield the following inequality
\begin{equation*}
\begin{aligned}
 \dot{b} & \geq C \left(s_c + \int |\nabla \xi|^2 dy + \int |\xi|^2 e^{-|y|}dy - \Gamma_b^{1 - \nu^6}\right) \\
 &- \eta \lambda^{2 - 2\frac{\sigma_2}{\sigma_1}} \left| (\ii|Q_b +\xi|^{2\sigma_2}(Q_b + \xi), \Lambda Q_b)\right|,
\end{aligned}
\end{equation*}
for some $C>0$. On the other hand, the contribution of the damping term is negligible as we control it with the calculations in the previous lemma (see \eqref{eq:toketa}) 
\begin{equation*}
 \eta \lambda^{2 - 2\frac{\sigma_2}{\sigma_1}} \left| (\ii|Q_b +\xi|^{2\sigma_2}(Q_b + \xi), \Lambda Q_b)\right| \lesssim \Gamma_b^2.
\end{equation*}
\end{proof}

\subsection{Refined virial estimate}\label{sec:outgo}

In this subsection, we derive an upper bound for $\dot{b}$. We will study the mass flux escaping the self-similar soliton core region. 
The outgoing radiation $\zeta_b$ defined in Lemma \ref{lem:rad} will play a central role in this task. 
Let $\phi \in C^\infty_c(\R^d)$ be a radial cut-off defined by
\begin{equation*}
\phi(r) \in [0,1], \ \ \ \phi(r) = \begin{cases}
& 1 \mbox{ for } r \in [0,1), \\ &
0 \mbox{ for } r \geq 2,
\end{cases}
\end{equation*}
We also define
\begin{equation*}
 \phi_A(r) = \phi\left(\frac{r}{A}\right)
\end{equation*} 
where $A=A(t)$ is determined by 
\begin{equation} \label{eq:A}
 A(t) = \Gamma_{b(t)}^{-a}
\end{equation}
and $a=a(\nu)>0$ will be chosen later. We denote by $\tilde{\zeta}_b = \phi_A \zeta_b$ the localization of the outgoing radiation. We notice that the equation for the outgoing radiation \eqref{eq:rad} implies that $\tilde{\zeta}_b$ satisfies the equation
\begin{equation}\label{eq:zetaTilde}
\Delta \tilde{\zeta_b} - \tilde{\zeta_b} + \ii b \left(\frac{d}{2} + y \cdot \nabla\right) \tilde{\zeta}_b 
= \phi_A \tilde{\Psi}_b + F = \tilde{\Psi}_b + F
\end{equation}
where
\begin{equation}\label{eq:F}
F= (\Delta \phi_A) \zeta_b + 2\nabla \phi_A \cdot \nabla \zeta_b + \ii b y \cdot \nabla \phi_A \zeta_b,
\end{equation}
and $\tilde\Psi_b$ was defined in \eqref{eq:psib0}.
In particular, by recalling that $\mathrm{supp}\Psi_b\subset B(0, R_b)\subset B(0, \frac{2}{b})$ and by the definition of $A(t)$, we have that $\phi_A \tilde{\Psi}_b = \tilde{\Psi}_b $.
Moreover, from the properties of $\zeta_b$ established in Lemma \ref{lem:rad}, we infer that $\tilde\zeta_b\in H^1_{rad}$ and, by suitably choosing $a=a(\nu)>0$, we also have
\begin{equation}\label{eq:zetaH1}
 \| \tilde{\zeta}_b\|_{H^1}^2 \leq \Gamma_b^{1 - c\rho}.
\end{equation}
We define the following refined soliton core and remainder,
\begin{equation*}
\tilde{Q}_b = Q_b + \tilde{\zeta_b} \ \mbox{ and } \ \tilde{\xi} = \xi - \tilde{\zeta_b},
\end{equation*}
so that we decompose the solution as
	\begin{equation}\label{eq:deco2}
		\psi(t,x) = \lambda(t)^{-\frac{1}{\sigma_1}}\left(\tilde{Q}_{b(t)}\left(\frac{x-x(t)}{\lambda(t)} \right) + \tilde{\xi}\left(t,\frac{x-x(t)}{\lambda(t)} \right)\right) e^{\ii \gamma(t)}.
	\end{equation}
We now claim that also the new soliton core $\tilde Q_b$ is an approximating solution to \eqref{eq:qbDer}. By using \eqref{eq:zetaTilde} and \eqref{eq:Qb}, we obtain
\begin{equation*}
 \begin{aligned}
 \ii \beta s_c \partial_b \QQ + \Delta \QQ - \QQ + \ii b \Lambda \QQ + |\QQ|^{2\sigma_1} \QQ &= \ii \beta s_c \partial_b Q_b + \Delta Q_b - Q_b \\
 &+ \ii b \Lambda Q_b+ |Q_b|^{2\sigma_1} Q_b \\
 & + \Delta \ZZ - \ZZ + \ii \left(\frac{d}{2} \ZZ + y \cdot \nabla \ZZ\right) - \ii s_c b \ZZ \\
 & - \ii s_c b \ZZ + \ii s_c\beta \partial_b \ZZ \\
 &+ |Q_b + \ZZ|^{2\sigma_1} (Q_b + \ZZ) - |Q_b|^{2\sigma_1} Q_b \\
 &= - \Psi_b + \tilde{\Psi}_b + F - \ii s_c b \ZZ + \ii s_c \beta \partial_b \ZZ \\
 &+ |Q_b + \ZZ|^{2\sigma_1} (Q_b + \ZZ) - |Q_b|^{2\sigma_1} Q_b.
 \end{aligned}
\end{equation*}
We observe that by using \eqref{eq:psi+phi}, we get $- \Psi_b + \phi_A \tilde{\Psi}_b = \Phi_b$. Thus the new profile $\QQ$ satisfies 
\begin{equation}\label{eq:QQ}
 \ii \beta s_c \partial_b \QQ + \Delta \QQ - \QQ + \ii b \Lambda \QQ + |\QQ|^{2\sigma_1} \QQ = - \tilde{\Phi}_b + F
\end{equation}
where 
\begin{equation}\label{eq:phiTilde}
 \tilde{\Phi}_b = - \Phi_b - \ii s_c b \ZZ + \ii \beta \partial_b \ZZ + |Q_b + \ZZ|^{2\sigma_1} (Q_b + \ZZ) - |Q_b|^{2\sigma_1} Q_b,
\end{equation}
and $F$ is defined in \eqref{eq:F}.
By using the equation \eqref{eq:QQ} for $\tilde Q_b$ and decomposition \eqref{eq:deco2}, we derive the equation for $\tilde\xi$,
\begin{equation} \label{eq:xi2}
	\begin{split}
		& \ii \partial_\tau \tilde{\xi} + \mathcal{\tilde{L}}\tilde{\xi} + \ii ( \dot{ b }- \beta s_c) \partial_b \QQ - \ii \left( \frac{\dot{ \lambda} }{\lambda} + b \right) \Lambda (\QQ + \tilde{\xi})- \ii\frac{\dot{ x}}{\lambda} \cdot \nabla (\QQ + \tilde{\xi}) \\
 & - (\dot{ \gamma } -1) (\QQ + \tilde{\xi})
	 + \ii \eta \lambda^{2 - 2\frac{\sigma_2}{\sigma_1}} |\QQ + \tilde{\xi}|^{2\sigma_2}(\QQ + \tilde{\xi}) + \tilde{R}(\tilde{\xi}) - \tilde{\Phi}_b + F = 0,
	\end{split}
\end{equation}
where
\begin{align*}
\mathcal{\tilde{L}}\tilde{\xi} = \Delta \tilde{\xi} - \tilde{\xi} + \ii b \Lambda \tilde{\xi} + 2\sigma_1 Re(\tilde{\xi} \overline{\tilde{Q}}_b) |\QQ|^{2\sigma_1 - 2} \QQ + |\QQ|^{2\sigma_1} \tilde{\xi}, 
\end{align*}
and 
\begin{displaymath}
\tilde{R}(\tilde{\xi}) = |\QQ + \tilde{\xi}|^{2\sigma_1}(\QQ + \tilde{\xi}) - |\QQ|^{2\sigma_1} \QQ - 2\sigma_1 Re(\tilde{\xi}\overline{\tilde{Q}}_b) |\QQ|^{2\sigma_1 - 2} \QQ 
- |\QQ|^{2\sigma_1} \tilde{\xi}.
\end{displaymath}
We recall again that the scaled time and space denoted by $\tau$ and $y$ are defined in \eqref{eq:scaled_t} and \eqref{eq:scaled_s}, respectively. With some abuse of notations, in what follows we often explicit the $\tau-$dependence of functions. For instance, we denote $\tilde\xi(\tau)$, to actually mean $\tilde\xi(t^{-1}(\tau))$.

By exploiting again the equation \eqref{eq:xi2} for the remainder $\tilde\xi$, the controls \eqref{eq:gxest} and \eqref{eq:bLest} and estimates in Proposition \ref{prp:deco}, we obtain the following bounds.

\begin{lem} \label{thm:refvirial}
There exists $C>0$ such that any $t\in [0,T_1)$, there exists $C>0$ such that 
\begin{equation}\label{eq:refvir}
\begin{split}
C^2\left( \int |\nabla\tilde{\xi}(\tau)|^2 + |\tilde{ \xi}(\tau)|^2 e^{-|y|} dy + \Gamma_{b(\tau)} \right) & \leq C \left(\frac{d}{d\tau}f(\tau) + s_c\right) \\ 
 & + \int_{\{A(\tau)\leq|y|\leq2A(\tau)\}}|\xi(\tau)|^2 dy,
\end{split}
\end{equation}
where
\begin{equation}\label{eq:f1}
\begin{aligned}
 f(\tau) = - \frac{1}{2} (\ii \tilde{Q}_{b(\tau)}, y \cdot \nabla \tilde{Q}_{b(\tau)}) - (\ii \tilde{\xi}(\tau), \Lambda \tilde{Q}_{b(\tau)}) .
\end{aligned} 
\end{equation}
\end{lem}

\begin{proof}
As for Lemmas \ref{lemma:mod_par} and \ref{lem:vir1}, we conveniently write equation \eqref{eq:xi2} as
\begin{equation*}
U^{(1)}(\tilde\xi)+\ii \eta \lambda^{2 - 2\frac{\sigma_2}{\sigma_1}} |\QQ + \tilde{\xi}|^{2\sigma_2}(\QQ + \tilde{\xi})=0,
\end{equation*}
so that for $U^{(1)}((\tilde\xi)$ we again exploit the analysis already presented in \cite[Lemma $3.5$]{MeRaSz10}, \cite[Lemma $6$]{MeRa05} and reported in Appendix \hyperref[sec:appenB]{B} . By taking the scalar product of equation \eqref{eq:xi2} with $\Lambda \QQ$, we obtain
	\begin{equation}\label{eq:tok2BU}
		(P(\tilde{\xi}), \Lambda \QQ) + (\ii \eta \lambda^{2 - 2\frac{\sigma_2}{\sigma_1}} |\QQ + \tilde{\xi}|^{2\sigma_2}(\QQ + \tilde{\xi}), \Lambda \QQ) = 0.
 	\end{equation} 
Exploiting the computations in Appendix \hyperref[sec:appenB]{B} yields the following inequality
 \begin{equation*}
 	\begin{split}
 		\frac{d}{d\tau}f& \geq C \left( \int |\nabla\tilde{\xi}|^2 +|\tilde{ \xi}|^2 e^{-|y|} \, dy + \Gamma_b \right) - \frac{1}{C} \left( s_c + \int_A^{2A} |\xi|^2 \, dy \right) \\& - \eta \lambda^{2 - 2\frac{\sigma_2}{\sigma_1}} \left| (\ii|\tilde{Q}_b +\tilde{\xi}|^{2\sigma_2}(\tilde{Q}_b + \tilde{\xi}), \Lambda \tilde{Q}_b) \right|.
 	\end{split}
 \end{equation*}
We will now show that in this inequality, the contribution of the damping term can be considered negligible. \newline
 From \eqref{eq:zetaH1} and \eqref{eq:prod4} it follows that
 \begin{equation}\label{eq:QbZeta}
 \begin{aligned}
 \left|\left( \ii |\tilde{Q}_b|^{2\sigma_2}\tilde{Q}_b , \Lambda \tilde{Q}_b \right) \right| & = \left|\left( \ii |Q_b + \tilde{\zeta_b} |^{2\sigma_2}(Q_b + \tilde{\zeta_b} ) , \Lambda (Q_b + \tilde{\zeta_b} ) \right) \right| \\
 &\leq \left|\left( \ii |Q_b|^{2\sigma_2}Q_b , \Lambda Q_b \right) \right| + \left|\left( \ii |Q_b|^{2\sigma_2}Q_b ,\Lambda \tilde{\zeta_b} \right) \right| \\
 & + \left|\left( \ii |Q_b + \tilde{\zeta_b} |^{2\sigma_2}(Q_b + \tilde{\zeta_b} ) - \ii |Q_b|^{2\sigma_2}Q_b , \Lambda (Q_b + \tilde{\zeta_b} ) \right) \right|.
 \end{aligned}
 \end{equation}
 We already know from \eqref{eq:weknow} that 
 \begin{equation*}
 \left|\left( \ii |Q_b|^{2\sigma_2}Q_b , \Lambda Q_b \right) \right| \lesssim 1.
 \end{equation*}
 For the other terms in \eqref{eq:QbZeta}, we use \eqref{eq:zetaH1} and \eqref{eq:crhoNU} to obtain 
\begin{equation*}
\begin{aligned}
 & \left|\left( \ii |Q_b|^{2\sigma_2}Q_b ,\tilde{\zeta_b} \right) \right| + \left|\left( \ii |Q_b + \tilde{\zeta_b} |^{2\sigma_2}(Q_b + \tilde{\zeta_b} ) - \ii |Q_b|^{2\sigma_2}Q_b , \Lambda (Q_b + \tilde{\zeta_b} ) \right) \right| \lesssim \Gamma_b^{1 - \nu}. 
\end{aligned}
\end{equation*}
Thus we obtain that 
\begin{equation*}
 \eta \lambda^{2 - 2\frac{\sigma_2}{\sigma_1}} \left| (\ii|\tilde{Q}_b +\tilde{\xi}|^{2\sigma_2}(\tilde{Q}_b + \tilde{\xi}), \Lambda \tilde{Q}_b) \right| \lesssim \eta \lambda^{2 - 2\frac{\sigma_2}{\sigma_1}} \leq \Gamma_b^2
\end{equation*}
where we used again \eqref{eq:initEta}, \eqref{eq:conb} and \eqref{eq:conL}. Inequality \eqref{eq:refvir} follows straightforwardly from \eqref{eq:refvir1} and \eqref{eq:conb}.
\end{proof}

Notice that the control \eqref{eq:refvir} is the same as that in \cite[Lemma $3.5$]{MeRaSz10}. This is again because the contribution of the damping term is negligible in the self-similar regime due to the choice of $\eta$ and the fact that $\sigma_2 \leq \sigma_1$.
\newline
We proceed by finding a suitable bound for the mass flux term 
\begin{equation*}
	\int_{\{A(\tau)\leq|y|\leq2A(\tau)\}} |\xi|^2 \, dy
\end{equation*} 
that appears on the right-hand side of \eqref{eq:refvir}. To do this, we introduce a further radial, smooth cut-off $\chi \in [0,1]$ with $\chi(r) = 0$ for $r \leq \frac{1}{2}$, and $\chi(r) = 1$ for $r \geq 3$, with $\chi' \geq 0 $ and $\chi'(r)\in \left[ \frac{1}{4}, \frac{1}{2} \right]$ for $ 1 \leq r \leq 2$. Let 
\begin{equation}\label{eq:chiA}
\chi_A(r) = \chi\left(\frac{r}{A} \right)
\end{equation} where $A=A(t)$ is defined in \eqref{eq:A}. 
In the following lemma, we will choose $s >0$ depending also on $\sigma_2$ in order to be able to control the $L^{2\sigma_2 + 2}$-norm of the remainder $\xi$, see \eqref{eq:lowerS} below. 

\begin{lem} 
	 For any $t \in [0,T_1)$, we have
 \begin{equation} \label{eq:l2_flux}
 \begin{aligned}
 b(\tau) \int_{\{A(\tau)\leq|y|\leq2A(\tau)\}}|\xi(\tau)|^2\,dy
 &\lesssim \lambda^{-2s_c}(\tau)\frac{d}{d\tau} \left( \lambda^{2s_c}(\tau)\int \chi_{A}|\xi(\tau)|^2 dy \right) \\
 &+ \Gamma_{b(\tau)}^{\frac{3}{2} - 10 \nu} + \Gamma_{b(\tau)}^{2a} \| \nabla \xi(\tau) \|_{L^2}^2.
 \end{aligned} 
 \end{equation}
\end{lem}

\begin{proof}
We take the scalar product of equation \eqref{eq:main} with $\ii\chi_A\left(\frac{x-x(t)}{\lambda(t)}\right)\psi(t, x)$ and obtain 
\begin{equation*}
 \frac{1}{2} \frac{d}{dt} (\psi, \chi_A \psi) - (\psi, (\partial_t \chi_A) \psi) - ( \nabla_x \psi, \ii (\nabla_x \chi_A) \psi) + \eta(|\psi|^{2\sigma_2} \psi, \chi_A \psi) = 0.
\end{equation*}
By using decomposition \eqref{eq:deco1} and the scaled space and time variables, we rewrite the equation above as
\begin{equation}\label{eq:l2flux12}
\begin{aligned}
 0 & = \frac{1}{2\lambda^{2s_c}} \frac{d}{d\tau} \left( \lambda^{2s_c} (\xi, \chi_A \xi)\right) - (\xi, (\d_\tau\chi_A) \xi) - ( \nabla \xi, \ii (\nabla \chi_A) \xi) \\&+ \eta \lambda^{2 - 2\frac{\sigma_2}{\sigma_1}}(|\xi|^{2\sigma_2} \xi, \chi_A \xi) 
 + R^{(1)}(Q_b,\xi),
\end{aligned}
\end{equation}
where the gradient and the scalar products are taken with respect to the variable
\begin{equation*}
 y(\tau) = \frac{x - x(\tau)}{\lambda(\tau)}
\end{equation*}
and
\begin{equation*}
 \begin{aligned}
 R^{(1)}(Q_b,\xi) &= \frac{1}{2\lambda^{2s_c}} \frac{d}{d\tau} \left(\lambda^{2s_c} \left( (Q_b, \chi_A Q_b) + 2(Q_b, \chi_A \xi) \right) \right) - (Q_b, (\partial_\tau \chi_A) Q_b) \\ & -2 (Q_b, (\partial_\tau \chi_A) \xi) - ( \nabla Q_b, \ii (\nabla \chi_A) Q_b) - 2( \nabla Q_b, \ii (\nabla \chi_A) \xi) \\
 &+ \eta \lambda^{2 - 2\frac{\sigma_2}{\sigma_1}}\left((|Q_b + \xi|^{2\sigma_2} (Q_b + \xi), \chi_A (Q_b+ \xi)) - (|\xi|^{2\sigma_2} \xi, \chi_A \xi) \right). \\
 \end{aligned}
\end{equation*}
We now estimate the reminder term $R^{(1)}(Q_b,\xi)$. Let us recall that for $|y|\geq\frac2b$ we have $Q_b = s_c T_b$, where $T_b$ is exponentially decreasing, see Proposition \ref{thm:profile}. We claim that, for the terms depending only on $Q_b$, we have
\begin{equation*}
\begin{aligned}
 \left|R^{(2)}(Q_b) \right| & = \bigg| \frac{1}{2\lambda^{2s_c}} \frac{d}{d\tau} \left(\lambda^{2s_c} \left( Q_b, \chi_A Q_b \right) \right)- (Q_b, (\partial_\tau \chi_A) Q_b) - ( \nabla Q_b, \ii (\nabla \chi_A) Q_b) \\
 &+ \eta \lambda^{2 - 2\frac{\sigma_2}{\sigma_1}}\left(|Q_b|^{2\sigma_2} Q_b , \chi_A Q_b) \right) \bigg| \\
 &\lesssim s_c^2 + \eta \lambda^{2 - 2\frac{\sigma_2}{\sigma_1}} s_c^{2\sigma_2 + 2}.
\end{aligned}
\end{equation*}
Here we used the bound 
\begin{equation*}
 \begin{aligned}
 \frac{d}{d\tau} \left(\lambda^{2s_c} \left( Q_b, \chi_A Q_b \right) \right) & = s_c^2\frac{d}{d\tau}\left(\lambda^{2s_c}\int\chi_A|T_b|^2\,dy\right) \\
 &\lesssim s_c^3 \lambda^{2 s_c}\left( \frac{\dot{\lambda}}{\lambda} - b + b\right) \| Q_b\|_{L^2}^2 + s_c^2 \lambda^{2 s_c} \dot{b} \partial_b \| Q_b\|_{L^2}^2 \\
 &\lesssim 
 s_c^3 \lambda^{2s_c} \left( \Gamma_b^{\frac{1}{2} - 10\nu} - \frac{c}{\ln{s_c}}\right) + s_c^2 \lambda^{2s_c}\Gamma_b^{\frac{1}{2} - 10\nu} \lesssim s_c^2.
 \end{aligned}
\end{equation*}
that follows from the properties of $Q_b$ \eqref{eq:Qbprop}.
All other terms may be estimated by
\begin{equation*}
 \begin{aligned}
 \left| R^{(1)}(Q_b,\xi) - R^{(2)}(Q_b) \right| \lesssim s_c \left( \int |\nabla \xi|^2 + |\xi|^2e^{-|y|} dy \right)^\frac{1}{2}.
 \end{aligned}
\end{equation*}
To prove the two claims above, we use the following property
\begin{equation*}
 \left|(T_b,\xi) \right| \leq \left( \int |\nabla \xi|^2 + |\xi|^2e^{-|y|} dy \right)^\frac{1}{2},
\end{equation*}
that can be proven similarly to \eqref{eq:prod1} and \eqref{eq:prod2}. 
Thus by using \eqref{eq:conb}, \eqref{eq:conL} and \eqref{eq:initEta}, we obtain that
\begin{equation*}
 \left| R^{(1)}(Q_b,\xi) \right| \lesssim s_c^2 + \eta \lambda^{2 - 2\frac{\sigma_2}{\sigma_1}} s_c^{2\sigma_2 + 2} + s_c \left( \int |\nabla \xi|^2 + |\xi|^2e^{-|y|} dy \right)^\frac{1}{2} \lesssim \Gamma_b^{\frac{3}{2} - 10 \nu}.
\end{equation*}
In particular, from \eqref{eq:l2flux12} we get the following inequality
\begin{equation*}
\begin{aligned}
 \frac{1}{2\lambda^{2s_c}} \frac{d}{d\tau} \left( \lambda^{2s_c} (\xi, \chi_A \xi)\right) &\geq (\xi, (\partial_\tau \chi_A) \xi) + ( \nabla \xi, \ii (\nabla \chi_A) \xi) \\ 
 &- \eta \lambda^{2 - 2\frac{\sigma_2}{\sigma_1}}(|\xi|^{2\sigma_2} \xi, \chi_A \xi) - \Gamma_b^{\frac{3}{2} - 10 \nu}.
\end{aligned}
\end{equation*}
From straightforward computations, we see that
\begin{equation*}
\d_\tau\chi_{A(\tau)}\left(\frac{x-x(\tau)}{\lambda(\tau)}\right)
=-\frac1A\left(\frac{\dot x}{\lambda}+\left(\frac{\dot\lambda}{\lambda}+\frac{\dot A}{A}\right)y\right)\cdot\nabla\chi\left(\frac{x-x(\tau)}{\lambda(\tau)A(\tau)}\right).
\end{equation*}
and consequently, we can rewrite the inequality above as
\begin{equation}\label{eq:chiAxi}
\begin{aligned}
 \frac{1}{2\lambda^{2s_c}}\frac{d}{d\tau} \left(\lambda^{2s_c} (\xi, \chi_A\xi)\right)
 &\geq b (\xi,y \cdot \nabla \chi_A \xi) \\
 &- \left( \xi, \frac{1}{A} \left( \frac{\dot{x}}{\lambda} + \left( \frac{\dot{\lambda}}{\lambda} + b\right) y + \frac{\dot{A}}{A} y \right) \cdot (\nabla \chi) \xi \right)\\
 & + (\nabla \xi, \ii (\nabla \chi_A) \xi ) - \eta \lambda^{2 - 2\frac{\sigma_2}{\sigma_1}} \int \chi_A | \xi|^{2\sigma_2 +2} dy - \Gamma_b^{\frac{3}{2} - 10 \nu}.
\end{aligned}
\end{equation}
The definition of $\chi_A$ \eqref{eq:chiA} implies the following chain of estimates 
\begin{equation}\label{eq:kot6}
\begin{aligned}
\frac18\int_{\{A\leq|y|\leq2A\}}|\xi|^2\,dy
\leq&\frac12\int_{\{A\leq|y|\leq2A\}}\chi'(\frac{|y|}{A})|\xi|^2\,dy
\leq\frac12\int_{\{A\leq|y|\leq2A\}}\frac{|y|}{A}\chi'(\frac{|y|}{A})|\xi|^2\,dy\\
=&\frac12\int_{\{A\leq|y|\leq2A\}}\frac{y}{A}\cdot\nabla\chi(\frac{|y|}{A})|\xi|^2\,dy.
\end{aligned}
\end{equation}
We will now exploit it to bound the terms in inequality \eqref{eq:chiAxi}. For the first term, we can easily infer the following bound
\begin{equation*}
 b \int y \cdot \nabla \chi_A |\xi|^2 dy \geq \frac{b}{8} \int_{\{A\leq|y|\leq2A\}} |\xi|^2 dy.
\end{equation*}
For the second term, the definition of $A(t)$ \eqref{eq:A} and estimate \eqref{eq:gammaB1} for $\Gamma_b$ allow us to infer
\begin{equation*}
\frac{\dot A}{A}=-ac\frac{\dot b}{b^2},
\end{equation*}
which yields 
\begin{equation*}
\frac1A\left|\left(\frac{\dot\lambda}{\lambda}+b+\frac{\dot A}{A}\right)\int y\cdot\nabla\chi|\xi|^2\,dy\right|
\lesssim\Gamma_b^a\Gamma_b^{1-20\nu}
\int_{A\leq|y|\leq2A}|\xi|^2\,dy,
\end{equation*}
where we have used \eqref{eq:bLest} and \eqref{eq:kot6}.
By using \eqref{eq:gxest} we may analogously estimate
\begin{equation*}\begin{aligned}
\frac1A\left|\int \frac{\dot x}{\lambda}\cdot\nabla\chi|\xi|^2\,dy\right|
&\lesssim\Gamma_b^a\left((\int|\nabla\xi|^2+|\xi|^2e^{-|y|}\,dy)^{1/2}+\Gamma_b^{1-20\nu}\right)
\int_{A\leq|y|\leq2A}|\xi|^2\,dy\\
&\lesssim\Gamma_b^a\Gamma_b^{\frac12-10\nu}
\int_{A\leq|y|\leq2A}|\xi|^2\,dy,
\end{aligned}\end{equation*}
where we used \eqref{eq:conxi} in the last inequality.
\newline
For the third term on the right-hand side of \eqref{eq:chiAxi} we use Young's inequality and get
\begin{equation*}
\begin{aligned}
    \frac1A\left|(\nabla\xi,\ii(\nabla\chi)\xi)\right|
\leq\frac1A\|\nabla\xi\|_{L^2}
\left(\int_{A\leq|y|\leq2A}|\xi|^2\,dy\right)^{1/2}
& \leq\frac{40}{bA^2}\|\nabla\xi\|_{L^2}^2 \\
&+\frac{b}{40} \int_{A\leq|y|\leq2A}|\xi|^2\,dy.
\end{aligned}
\end{equation*}
Finally, we consider the contribution coming from the nonlinear damping. By recalling that $\sigma_2 > \sigma_* = 2s_c/(d - 2s_c)$, we have
\begin{equation*}
 s_c < \frac{d}{2} - \frac{d}{2\sigma_2 + 2} = \frac{d\sigma_2}{2\sigma_2 + 2}.
\end{equation*}
This implies that we can choose $s$ such that 
\begin{equation*}
 s_c < s < \frac{d\sigma_2}{2\sigma_2 + 2} = s(\sigma_2).
\end{equation*}
Consequently, we can interpolate the space $\dot H^{s(\sigma_2)}(\R^d)$ between $\dot H^1(\R^d)$ and $\dot H^s(\R^d)$ and obtain that 
\begin{equation}\label{eq:lowerS}
 \| \xi\|_{L^{2\sigma_2 +2}}^{2\sigma_2 + 2} \lesssim \| |\nabla|^{s(\sigma_2)}\xi\|_{L^2}^{2\sigma_2 + 2} \lesssim 
 \| \xi\|_{\dot{H}^1}^{\theta(2\sigma_2 + 2)} \| \xi\|_{\dot{H}^{s}}^{(1 - \theta)(2\sigma_2 + 2)}
\end{equation} 
for some $\theta(s) \in (0,1)$. 
Now from \eqref{eq:conxi} and \eqref{eq:conHs} it follows that
\begin{equation*}
 \| \xi\|_{\dot{H}^1}^{\theta(2\sigma_2 + 2)} \| \xi\|_{\dot{H}^{s}}^{(1 - \theta)(2\sigma_2 + 2)} \leq \Gamma_b^{(1 - 50\nu)(\sigma_2 + 1) + 30 \nu \theta (1 + \sigma_2)}.
\end{equation*}
Thus, by collecting everything, we have that
\begin{equation}\label{eq:xiSigma2}
 \begin{aligned}
 \frac{1}{2\lambda^{2s_c}}\frac{d}{d\tau} \left(\lambda^{2s_c} (\xi, \chi_A\xi)\right)
 & \geq \left(\frac{b}{8} - c \Gamma_b^{a} 
\Gamma_b^{\frac12-10\nu}- \frac{b}{40} \right) \int_{\{A\leq|y|\leq2A\}} |\xi|^2 dy \\
&- \Gamma_b^{2a} \| \nabla \xi\|_{L^2}^2 - \Gamma_b^{\frac{3}{2} - 10 \nu} \\ 
 & - \eta \lambda^{2 - 2\frac{\sigma_2}{\sigma_1}} \Gamma_b^{(1 - 50\nu)(\sigma_2 + 1) + 30 \nu \theta (1 + \sigma_2)}.
 \end{aligned}
\end{equation}
Inequality \eqref{eq:l2_flux} is a simple consequence of the choice of $\eta$ \eqref{eq:initEta}. 
\end{proof}
Let us emphasize that the assumption $\sigma_2>\sigma_*$, see \eqref{eq:sigma2}, is required to obtain the bound \eqref{eq:lowerS}. As already remarked, this assumption is related to the fact that it is not possible to control Sobolev norms of $\xi$ rougher than the critical norm $\dot{H}^{s_c}$. Consequently, our argument cannot be applied for instance to the linearly damped NLS equation (equation \eqref{eq:main} with $\sigma_2=0$) since we would need to estimate the term
\begin{equation*}
 \eta \lambda^2 \int \chi_A |\xi|^2 dy.
\end{equation*}
By using estimates \eqref{eq:refvir} and \eqref{eq:l2_flux}, it is possible to define a Lyapunov functional that provides an upper bound for $\dot b$. In the next lemma, it will be fundamental to exploit the decay of the total mass. We notice that this fact suggests that a different regularization of the focusing NLS dynamics could not yield the same result. 
\newline
In what follows we define
\begin{equation}\label{eq:Jclean}
 \begin{aligned}
 J(\tau)&= \int(1 - \chi_{A(\tau)} ) |\xi(\tau)|^2 dy + \|Q_{b(\tau)}\|_{L^2}^2 - \| Q\|_{L^2}^2 + 2(\xi(\tau), Q_{b(\tau)}) \\ &- b(\tau)f(\tau) + \int_0^{b(\tau)} f(v)\, dv,
 \end{aligned}
\end{equation}
where $f$ is defined in \eqref{eq:f1} and $Q$ is the unique positive solution to \eqref{eq:gsQ}.
\begin{lem} \label{thm:global_virial}
 There exist $a_1 = a_1(\nu) >0$
 such that for any $a < a_1$ and any $t \in [0,T_1)$, there exist $c >0$ such that
\begin{equation}\label{eq:second_monotonicity}
\begin{aligned}
 \frac{d}{d\tau} J(\tau) &\lesssim b(\tau) s_c + \Gamma_{b(\tau)}^{2a} \| \nabla \xi(\tau) \|_{L^2}^2 \\ &- b(\tau) \left( \Gamma_{b(\tau)} + \int |\nabla \tilde{\xi}(\tau)|^2 + |\tilde{\xi}(\tau)|^2 e^{-|y|} dy\right).
\end{aligned}
\end{equation}
\end{lem}

\begin{proof} 
By multiplying inequality \eqref{eq:refvir} by $b$, we obtain 
\begin{equation}\label{eq:kot7}
 C^2 b\left(\int |\nabla \tilde{\xi}|^2 + |\tilde{\xi}|^2 e^{-|y|} dy + \Gamma_b \right) \leq C b \left( \frac{d}{d\tau} f + s_c\right) + b\int_{\{A\leq|y|\leq2A\}} |\xi|^2 dy.
\end{equation}
The last term on the right-hand side of \eqref{eq:kot7} may be estimated by \eqref{eq:l2_flux}, so that
\begin{equation}\label{eq:massFlux}
 b \int_{\{A(\tau)\leq|y|\leq2A(\tau)\}}|\xi|^2dy \lesssim \lambda^{-2s_c} \frac{d}{d\tau}\left( \lambda^{2s_c} \int \chi_A |\xi|^2 dy \right) + \Gamma_b^{\frac{3}{2} - 10 \nu} + \Gamma_{b}^{2a} \| \nabla \xi \|_{L^2}^2.
\end{equation}
In order to find a satisfactory bound on the first term on the right-hand side of the inequality above, we exploit the fact that the total mass is non-increasing. By writing
\begin{equation*}
\begin{aligned}
 \frac{d}{d\tau} \| \psi \|_{L^2}^2 = \frac{d}{d\tau} \left( \lambda^{2s_c} \int(\chi_A + (1 - \chi_A) ) |\xi|^2 + |Q_b|^2 dy + 2\lambda^{2s_c}(\xi, Q_b) \right) \leq 0,
\end{aligned}
\end{equation*}
we have that
\begin{equation}\label{eq:dissImp}
\begin{aligned}
 \lambda^{-2s_c} \frac{d}{d\tau}\left( \lambda^{2s_c} \int \chi_A |\xi|^2 dy \right) &\leq - \frac{d}{d\tau} \left( \int(1 - \chi_A ) |\xi|^2 + |Q_b|^2 dy + 2(\xi, Q_b) \right) \\ &
 - 2s_c \frac{\dot{\lambda}}{\lambda} \left( \int(1 - \chi_A) |\xi|^2 + |Q_b|^2 dy + 2(\xi, Q_b) \right). 
\end{aligned}
\end{equation}
From the definition of $\chi_A$ given in \eqref{eq:chiA} and by Hardy's inequality, we have
\begin{equation*}
 \int(1 - \chi_A) |\xi|^2 \leq \int_{|y| \leq 3A} \frac{|y|^2}{|y|^2} |\xi|^2 dy \lesssim A^2 \| \nabla \xi\|_{L^2}^2. 
\end{equation*}
Similarly, we can obtain a comparable bound in dimensions one and two, see \cite[Appendix $C$]{MeRa05}.
In particular, for any dimension we conclude that
\begin{equation}\label{eq:hardyXi}
 \int(1 - \chi_A) |\xi|^2 \lesssim A^2 \int |\nabla \xi|^2 + |\xi| e^{-|y|} dy. 
\end{equation}
Thus, by using estimates \eqref{eq:bLest}, \eqref{eq:enest} and \eqref{eq:conxi} we can bound the second term on the right-hand side of \eqref{eq:dissImp} by
\begin{equation*}
\begin{aligned}
 &\left| s_c \frac{\dot{\lambda}}{\lambda} \left( \int(1 - \chi_A) |\xi|^2 + |Q_b|^2 dy + 2(\xi, Q_b) \right) \right| \\
 &\leq \left|s_c \left(\frac{\dot{\lambda}}{\lambda} + b\right) \left( \int(1 - \chi_A) |\xi|^2 + |Q_b|^2 dy + 2(\xi, Q_b) \right) \right| \\ 
 & + \left| s_c b \left( \int(1 - \chi_A) |\xi|^2 + |Q_b|^2 dy + 2(\xi, Q_b) \right) \right| \\
 & \lesssim s_c \left( \Gamma_b^{1 - 20 \nu} + b \right)\left( (1 + A^2)\int |\nabla \xi|^2 + |\xi| e^{-|y|} dy + \| Q_b\|_{L^2}^2 \right).
\end{aligned}
\end{equation*}
Now we want that
\begin{equation}\label{eq:hardyXi2}
 A^2 \int |\nabla \xi|^2 + |\xi| e^{-|y|} dy \leq A^2 \Gamma_b^{1 - 20\nu} = \Gamma_{b}^{-2a + 1 - 20\nu} \lesssim 1,
\end{equation}
where we used the definition of $A$ in \eqref{eq:A} and \eqref{eq:conxi}. That is $a$ must satisfy the following inequality 
\begin{equation*}
 a \leq \frac{1 - 20\nu}{2}. 
\end{equation*}
This implies that
\begin{equation*}
 \left| s_c \frac{\dot{\lambda}}{\lambda} \left( \int(1 - \chi_A) |\xi|^2 + |Q_b|^2 dy + 2(\xi, Q_b) \right) \right| \lesssim s_c b.
\end{equation*}
By combining \eqref{eq:massFlux} with \eqref{eq:dissImp}, we obtain
\begin{equation*}
 \begin{aligned}
 b \int_{\{A\leq|y|\leq2A\}} |\xi|^2dy &\lesssim - \frac{d}{d\tau} \left( \int(1 - \chi_A ) |\xi|^2 + |Q_b|^2 \, dy + 2(\xi, Q_b) \right) \\ &+ s_c b + \Gamma_b^{\frac{3}{2} - 10 \nu} + \Gamma_{b}^{2a} \| \nabla \xi\|_{L^2}^2. 
 \end{aligned}
\end{equation*}
By plugging the above inequality into \eqref{eq:kot7}, we infer
\begin{equation}\label{eq:primeJ}
 \begin{aligned}
 b\left(\int |\nabla \tilde{\xi}|^2 + |\tilde{\xi}|^2 e^{-|y|} dy + \Gamma_b \right) &\lesssim b \frac{d}{d\tau} f \\ 
 &- \frac{d}{d\tau} \left( \int(1 - \chi_A ) |\xi|^2 dy + \|Q_b\|_{L^2}^2 + 2(\xi, Q_b) \right) \\ 
 &+ s_c b + \Gamma_b^{\frac{3}{2} - 10 \nu} + \Gamma_{b}^{2a} \| \nabla \xi\|_{L^2}^2. 
 \end{aligned}
\end{equation}
Finally, let us consider $f=f(\tau)$ as defined in \eqref{eq:f1}. By the monotonicity property of $b$, see \eqref{eq:locvi} for instance, we denote - by some abuse of notation - $f=f(b(\tau))$. In this way, we may write
\begin{equation*}
 b \frac{d}{d\tau} f = \frac{d}{d\tau} (bf) - \frac{d}{d\tau} \int_0^b f(v)\, dv.
\end{equation*}
Let us now recall the definition of $J(\tau)$ given in \eqref{eq:Jclean}, we have
\begin{equation*}
\begin{aligned}
 J(\tau)&= \int(1 - \chi_{A(\tau)} ) |\xi(\tau)|^2 \, dy + \|Q_{b(\tau)}\|_{L^2}^2 - \| Q\|_{L^2}^2 + 2(\xi(\tau), Q_{b(\tau)}) \\ &- b(\tau)f(\tau) + \int_0^{b(\tau)} f(v)\, dv.
\end{aligned}
\end{equation*}
By using the previous identities and \eqref{eq:conb}, we see that \eqref{eq:primeJ} implies the following estimate
\begin{equation}\label{eq:secondJ}
 \begin{aligned}
 b\left(\int |\nabla \tilde{\xi}|^2 + |\tilde{\xi}|^2 e^{-|y|} \, dy + \Gamma_b \right) \lesssim - \frac{d}{d\tau}J + bs_c +\Gamma_{b}^{2a} \| \nabla \xi \|_{L^2}^2,
 \end{aligned}
\end{equation}
which readily gives \eqref{eq:second_monotonicity}.
\end{proof}	

Let us now discuss how the functional $J$ is related to the control parameter $b$. First, we define $K(\tau)$ as
\begin{equation}\label{eq:Kdef}
 K(\tau) = \|Q_{b(\tau)}\|_{L^2}^2 - \| Q \|_{L^2}^2 - b(\tau)f(\tau) + \int_0^{b(\tau)} f(v)\, dv 
\end{equation}
where $Q$ is the ground state profile to \eqref{eq:main}, see \eqref{eq:gsQ}. In this way, we obtain that
\begin{equation*}
\begin{aligned}
 J(\tau) - K(\tau) = \int(1 - \chi_{A(\tau)} ) |\xi(\tau)|^2 \, dy + 2(\xi(\tau), Q_{b(\tau)}).
\end{aligned}
\end{equation*}
Thus, by using \eqref{eq:hardyXi}
 and \eqref{eq:enest1} we get
\begin{equation}\label{eq:j-kupper}
 J(\tau) - K(\tau) \lesssim (1 + A^2) \int |\xi|^2 e^{-|y|} + |\nabla \xi|^2 \, dy + \Gamma_b^{1 - \nu}.
\end{equation}
We stress that to provide a satisfactory upper bound on the difference $J-K$, it is not sufficient to consider the weaker bound \eqref{eq:enest} in the inequality above. Indeed, in the proof of the bootstrap in the next section, we will first show that the control \eqref{eq:conxi} is not satisfactory to obtain the better bound \eqref{eq:butH1}. \newline
To obtain a lower bound we observe that, by using \eqref{eq:entok}, we have 
\begin{equation*}
 \begin{aligned}
 J - K & = \int(1 - \chi_{A} ) |\xi|^2 dy - 2\lambda^{2 - 2s_c} E[\psi] + 2E[Q_b] + \| \nabla \xi\|_{L^2}^2 \\
 & + 2\left(\ii s_c \beta \partial_b Q_b + \ii b \Lambda Q_b + \Psi_b,\xi \right) - \int R^{(2)}(\xi) dy
 \end{aligned}
\end{equation*}
where $R^{(2)}(\xi)$ is defined in \eqref{eq:R2}. We rewrite the equation above as 
\begin{equation}\label{eq:j-k11}
 \begin{aligned}
 J - K = - 2\lambda^{2 - 2s_c} E[\psi] + 2E[Q_b] + 2\left(\ii s_c \beta \partial_b Q_b + \Psi_b,\xi \right) \\ 
 + ( \mathcal{L}^{(1)} \xi, \xi) - \int \chi_A |\xi|^2\, dy - \frac{1}{\sigma_1 + 1} \int R^{(3)}(\xi) dy
 \end{aligned}
 \end{equation}
 where
 \begin{equation}\label{eq:linMR}
 \mathcal{L}^{(1)} \xi = - \Delta \xi + \xi - 2\sigma_1 Re(\xi \bar{Q}_b) |Q_b|^{2\sigma_1 - 2} Q_b - |Q_b|^{2\sigma_1}\xi
 \end{equation}
 and 
\begin{equation*}
 \begin{aligned}
 R^{(3)}(\xi) = |Q_b + \xi|^{2\sigma_1 + 2} - |Q_b|^{2\sigma_1 + 2} - (2\sigma_1 + 2)|Q_b|^{2\sigma_1} Re(Q_b \bar{\xi}) \\
 - (2\sigma_1 + 2) |Q_b|^{2\sigma_1 -2}\left(|Q_b|^2 |\xi|^2 + 2\sigma_1 Re(Q_b \bar{\xi})^2 \right).
 \end{aligned}
\end{equation*}
We recall the coercivity property 
\begin{equation*}
 ( \mathcal{L}^{(1)} \xi, \xi) - \int \chi_A |\xi|^2\, dy\geq C\int |\xi|^2 e^{-|y|} + |\nabla \xi|^2\, dy - \Gamma_b^{\frac{3}{2} - 10 \nu}
\end{equation*}
which was proved in \cite[Appendix $D$]{MeRa05}. Then in \eqref{eq:j-k11} by further using \eqref{eq:conEP}, \eqref{eq:Qbprop}, \eqref{eq:prod2} and \eqref{eq:prod3} to bound the rests as
\begin{equation*}
 \left|2\left(\ii s_c \beta \partial_b Q_b + \Psi_b,\xi \right) 
 - \frac{1}{\sigma_1 + 1} \int R^{(3)}(\xi) dy\right| \lesssim \Gamma_b
\end{equation*}
we obtain
\begin{equation}\label{eq:j-klower}
 J - K \gtrsim \int |\xi|^2 e^{-|y|} + |\nabla \xi|^2\, dy- \Gamma_b - s_c.
\end{equation}
By combining \eqref{eq:j-klower} and \eqref{eq:j-kupper}, we obtain that 
\begin{equation}\label{eq:j-k}
\begin{aligned}
 \int |\xi|^2 e^{-|y|} + |\nabla \xi|^2\, dy- \Gamma_b - s_c \lesssim J-K \lesssim (1 + A^2) \int |\xi|^2 e^{-|y|} + |\nabla \xi|^2\, dy+ \Gamma_b^{1 - \nu}.
\end{aligned}
\end{equation}
Now we want to prove that $K(\tau)$ is a small perturbation of $b^2(\tau)$. By slightly abusing notations again, we write $K(\tau) = K(b(\tau))$. The following lemma can be found in \cite[Section $4$]{MeRaSz10}.
\begin{lem}\label{eq:kPerB}
 There exists $b^{(1)} >0$ such that for any $0 < b < b_1 < b^{(1)}$, we have
 \begin{equation}\label{eq:k-k}
 K(b) - K(b_1) \lesssim s_c
 \end{equation}
 and 
 \begin{equation}\label{eq:k-b}
 b^2 - s_c \lesssim K(b) \lesssim b^2 + s_c. 
 \end{equation}
\end{lem}

The proof of this lemma follows from the properties of $Q_b$ stated in \eqref{eq:Qbprop}, the definition of $f$ in \eqref{eq:f1} the decomposition \eqref{eq:deco2}, and the smallness of the outgoing radiation \eqref{eq:zetaH1}. Observe that by collecting together \eqref{eq:j-k} and \eqref{eq:k-b}, we have that $J(\tau)$ is close to $b^2(\tau)$ up to smaller order corrections and up to the term $A^2 \int |\xi|^2 e^{-|y|} + |\nabla \xi|^2 dy$. Consequently, if $a$ and $b$ are small enough, we obtain the following inequality
\begin{equation}\label{eq:j-b}
 \left| J(\tau) - db^2(\tau)\right| \lesssim \Gamma_{b(\tau)}^{1 - 20\nu} + A^2 \int |\xi|^2 e^{-|y|} + |\nabla \xi|^2\, dy\leq \nu b^2(\tau).
\end{equation}

\section{Proof of the Bootstrap} \label{ch:bootstrap}

In this section, we are going to prove Proposition \ref{thm:boot}. We proceed in the following order.

\begin{enumerate}
\item Using the monotonicity properties \eqref{eq:locvi} and \eqref{eq:second_monotonicity}, we refine the control over the remainder $\xi$ as in \eqref{eq:butH1}.
\item Inequalities \eqref{eq:locvi} and \eqref{eq:second_monotonicity} also imply the dynamical trapping of $b$ \eqref{eq:butb}. In particular, $b$ is almost constant and close to the value $b^*(s_c)>0$ defined in \eqref{eq:s_c}.
\item From \eqref{eq:bLest}, we obtain the equation for the scaling parameter
\begin{equation*}
	\frac{\dot{\lambda}}{\lambda} \sim - b + \Gamma_b.
\end{equation*}
The previous points yield \eqref{eq:butL} and a precise law for $\lambda$.
\item By finding suitable bounds on the time derivatives of the energy and the momentum, we will prove \eqref{eq:butEP}. 
\item Finally, we will deduce the $\dot{H}^s$-norm control of $\xi$ in \eqref{eq:butHs}.
\end{enumerate}

We shall stress that the first three points of our scheme are consequences of the local virial law \eqref{eq:locvi} and the monotonicity formula \eqref{eq:second_monotonicity} and their proofs are very similar to those in \cite{MeRaSz10}. On the other hand, in the undamped case, the fourth point comes naturally from the conservation of the energy and the momentum. In our case, we will show that the growth of the time-dependent energy and momentum (see equations \eqref{eq:momGrow} and \eqref{eq:en_grow}) can be controlled until the blow-up time by our choice of $\eta$. Finally, with respect to the proof in \cite{MeRaSz10}, the fifth point requires an additional change to treat the new dissipative term. For the convenience of the reader, we will restate the bootstrap proposition below.

\begin{prop} \label{thm:boot1}
There exists $s_c^*>0$, $s^*>s_c^*$, $\nu^* >0$ and $a^*(\nu^*)> a_*(\nu^*) >0$, such that for any $s_c < s_c^*$, $s_c < s < s^*$, $\nu < \nu^*$ and $ a_* < a < a^*$ and for any $t \in [0,T_1),$ the following inequalities are true: 
\begin{align}\label{eq:butb1}
&\Gamma_{b(t)}^{1 + \nu^4} \leq s_c \leq \Gamma_{b(t)}^{1 - \nu^4}, \\ \label{eq:butL1}
&0 \leq \lambda(t) \leq \Gamma_{b(t)}^{20}, \\ \label{eq:butEP1}
& \lambda^{2 - 2s_c} |E[\psi(t)]| + \lambda^{1 - 2s_c} |P[\psi(t)]| \leq \Gamma_{b(t)}^{3 - 10\nu}, \\ \label{eq:butHs1}
&\int ||\nabla|^{s} \xi(t)|^2 \leq \Gamma_{b(t)}^{1 - 45\nu}, \\ \label{eq:butH11}
& \int |\nabla \xi(t)|^2 + |\xi(t)|^2 e^{-|y|}\, dy\leq \Gamma_{b(t)}^{1 - 10\nu}.
\end{align}
\end{prop}

\begin{proof}
We choose $s_c^*$ to be the minimum of all the conditions for $s_c$ found in previous sections. Then we fix $s_c < s_c^*$ and we choose $s$ such that
\begin{equation*}
 s_c < s < \min \left(\frac{d\sigma_1}{2\sigma_1 + 2}, \frac{d\sigma_2}{2\sigma_2 + 2}, \frac{1}{2} \right)
\end{equation*}
 such that \eqref{eq:lowerS}, \eqref{eq:nnlin} and \eqref{eq:momest} are true. Now we choose an initial condition $\psi_0 \in \mathcal O$, where the set $\mathcal O$ is defined in \ref{def:deco}. Consequently, there exists a time interval $[0,T_1]$ where the estimates in Proposition \ref{prp:deco} are true. In particular, by choosing $s_c$ small enough we obtain that $b(t)$ is small for any $t \in [0,T_1]$ and also that $\Gamma_{b(t)} < 1$. \newline
We start by proving inequality \eqref{eq:butH11}. We proceed by contradiction. First, suppose that there exists $\tau_0 \in [0, \tau^{-1}(T_1))$ such that 
\begin{equation*}
 \int |\nabla \xi(\tau_0)|^2 + |\xi(\tau_0)|^2 e^{-|y|}\, dy> \Gamma_{b(\tau_0)}^{1 - 10 \nu}.
\end{equation*}
Notice that the initial condition satisfies \eqref{eq:initHs}, and in particular, since $\Gamma_b < 1$, for any $t \in [0,T_1)$, see \eqref{eq:conb}, we have
\begin{equation*}
 \int |\nabla \xi(0)|^2 + |\xi(0)|^2 e^{-|y|}\, dy < \Gamma_{b(0)}^{1 - 7 \nu},
\end{equation*}
Then by continuity of $\xi$, $b$ and $\Gamma_b$, there exists a time interval $[\tau_1, \tau_2] \subset [0, \tau_0)$ such that 
\begin{equation} \label{eq:s3_s4}
\begin{aligned}
 \int |\nabla \xi(\tau_1)|^2 + |\xi(\tau_1)|^2 e^{-|y|}\, dy= \Gamma_{b(\tau_1)}^{1 - 7 \nu}, \\ 
 \int |\nabla \xi(\tau_2)|^2 +|\xi(\tau_2)|^2 e^{-|y|}\, dy= \Gamma_{b(\tau_2)}^{1 - 10 \nu},
\end{aligned}
\end{equation}
and for any $\tau \in [\tau_1,\tau_2]$ 
\begin{equation} \label{eq:proof_1}
\int |\nabla \xi(\tau)|^2 + |\xi(\tau)|^2 e^{-|y|}\, dy\geq \Gamma_{b(\tau)}^{1 - 7\nu}. 
\end{equation} 
The virial estimate \eqref{eq:locvi}, controls \eqref{eq:conb}, \eqref{eq:conL} and \eqref{eq:conEP} imply that for any $\tau \in [\tau_1,\tau_2]$, we have
\begin{equation*}
\begin{split}
\dot{b} & \gtrsim s_c +\int |\nabla \xi(\tau)|^2 + |\xi(\tau)|^2 e^{-|y|}\, dy- \Gamma_b^{1 - \nu^6} \gtrsim \Gamma_b^{1 + \nu^2} + \Gamma_{b}^{1 - 7\nu} - \Gamma_b^{1 - \nu^6} > 0,
\end{split}
\end{equation*}
for $\nu$ small enough and hence
\begin{equation} \label{eq:monotonicity_b}
b(\tau_2) \geq b(\tau_1).
\end{equation}
We notice that the definition of $\Gamma_b$ \eqref{eq:gammaB1} implies also that
\begin{equation}\label{eq:monGamma}
 \Gamma_{b(\tau_2)} \geq \Gamma_{b(\tau_1)}.
\end{equation}
On the other hand, from the smallness of $\ZZ$ \eqref{eq:zetaH1}, we obtain the inequality 
\begin{equation}\label{eq:xitildeTOXI}
\int |\nabla \tilde{\xi}|^2 + |\tilde{\xi}|^2 e^{- |y|}\, dy\geq \frac{1}{2} \int |\nabla \xi|^2 + |\xi|^2 e^{-|y|} \, dy - \Gamma_b^{1 - \nu}
\end{equation}
which combined with the second monotonicity estimate \eqref{eq:second_monotonicity} and \eqref{eq:conb} yields 
\begin{displaymath}
\begin{aligned}
 \frac{d}{d\tau} J & \lesssim b s_c + \Gamma_{b}^{2a} \| \nabla \xi\|_{L^2}^2- b \Gamma_b - b \int |\nabla \tilde{\xi}(\tau)|^2 + |\tilde{\xi}(\tau)|^2 e^{-|y|}\, dy \\
 & \lesssim b\left(s_c - \Gamma_b - \frac{1}{2}\int |\nabla \xi|^2 + |\xi|^2 e^{-|y|}dy + \Gamma_b^{1 - \nu} + \frac{\Gamma_{b}^{2a}}{b} \| \nabla \xi\|_{L^2}^2 \right) \\
 & \lesssim b \left(\Gamma_b^{1 - \nu^2} - \Gamma_b - \Gamma_b^{1 - 7\nu} + \Gamma_b^{1 - \nu} \right) \leq 0,
\end{aligned}
\end{displaymath} 
for $\nu$ small enough where we have chosen $a>0$ so that
\begin{equation*}
 \frac{\Gamma_{b}^{2a}}{b} \leq \frac{1}{4}. 
\end{equation*}
Thus we obtain that
\begin{equation} \label{eq:monotonicity_J}
J(\tau_2) \leq J(\tau_1).
\end{equation} 
Next, using \eqref{eq:j-k} and \eqref{eq:monotonicity_J} we get 
\begin{equation*}
 \begin{aligned}
 K(\tau_2) + \int |\xi(\tau_2)|^2 e^{-|y|} + |\nabla \xi(\tau_2)|^2\, dy- \Gamma_{b(\tau_2)} - s_c \lesssim J(\tau_2) \leq J(\tau_1) \\ \lesssim K(\tau_1) + (1 + A^2) \int |\xi(\tau_1)|^2 e^{-|y|} + |\nabla \xi(\tau_1)|^2\, dy+ \Gamma_{b(\tau_1)}^{1 - \nu}.
 \end{aligned}
\end{equation*}
Equivalently, by using \eqref{eq:s3_s4} we get 
\begin{equation*}
 \begin{aligned}
 \int |\xi(\tau_2)|^2 e^{-|y|} + |\nabla \xi(\tau_2)|^2\, dy &= \Gamma_{b(\tau_2)}^{1 - 10\nu} \\ & \lesssim K(\tau_1) - K(\tau_2) + s_c + \Gamma_{b(\tau_2)} 
 + (1 + A^2) \Gamma_{b(\tau_1)}^{1 - 7\nu} + \Gamma_{b(\tau_1)}^{1 - \nu}.
 \end{aligned}
\end{equation*}
On the right-hand side of the equation above, we use \eqref{eq:k-k} to get
\begin{equation*}
 \left| K(\tau_1) - K(\tau_2)\right| \lesssim s_c.
\end{equation*}
Thus using the inequality $s_c \leq \Gamma_{b(\tau_1)}^{1 - \nu^2}$ \eqref{eq:conb} and the definition of $A$ \eqref{eq:A}, we obtain 
\begin{equation*}
 \Gamma_{b(\tau_2)}^{1 - 10\nu} \lesssim \Gamma_{b(\tau_1)}^{1 - \nu^2} + \Gamma_{b(\tau_2)} + \Gamma_{b(\tau_1)}^{1 - 7\nu - 2a} + \Gamma_{b(\tau_1)}^{1 - \nu}. 
\end{equation*}
Now we suppose that $a \leq \frac{\nu}{2}$. This implies that there exists $C>0$ such that
\begin{equation*}
 \Gamma_{b(\tau_2)}^{1 - 10\nu} \leq C \Gamma_{b(\tau_1)}^{1 - 8\nu}.
\end{equation*}
By exploiting \eqref{eq:monGamma}, and for $b$ small enough, this implies that
\begin{equation*}
 \Gamma_{b(\tau_2)}^{1 - 10\nu} \leq C \Gamma_{b(\tau_1)}^{1 - 8\nu} \leq C \Gamma_{b(\tau_2)}^{1 - 8\nu} = C \Gamma_{b(\tau_2)}^{\nu} \Gamma_{b(\tau_2)}^{1 - 9\nu} \leq \Gamma_{b(\tau_2)}^{1 - 9\nu}, 
\end{equation*}
which is a contradiction because $\Gamma_{b(\tau_2)} < 1$. 
\newline
As our next step, we will prove \eqref{eq:butb1}. Assume by contradiction that there exists $\tau_0 \in [0, \tau^{-1}(T_1)]$ such that $s_c > \Gamma_{b(\tau_0)}^{1- \nu^4}.$ By continuity, this implies that there exists $\tau_1 < \tau_0$ such that 
\begin{equation*}
 s_c = \Gamma_{b(\tau_1)}^{1 - \nu^5}
\end{equation*}
and for any $\tau \in [\tau_1,\tau_0]$, 
\begin{equation*}
 s_c > \Gamma_{b(\tau)}^{1 - \nu^5}.
\end{equation*}
In particular, we have that 
\begin{equation*}
 \frac{d}{d\tau} \Gamma_{b(\tau_1)} < 0
\end{equation*}
which is equivalent to $ \dot{b}(\tau_1) < 0$ from the definition of $\Gamma_b$ in \eqref{eq:gammaB1}. On the other hand, the local virial inequality \eqref{eq:locvi} implies that
\begin{displaymath}
\dot{b}(\tau_1) \gtrsim s_c - \Gamma_{b(s_6)}^{1 - \nu^6} \gtrsim \Gamma_{b(s_6)}^{1 - \nu^5} - \Gamma_{b(s_6)}^{1 - \nu^6} > 0,
\end{displaymath}
which is a contradiction. \newline
Similarly, suppose by contradiction that there exists $\tau_0 \in [0, \tau^{-1}(T_1))$ such that $s_c < \Gamma_{b(\tau_0)}^{1 + \nu^4}.$ Then, by using \eqref{eq:j-b}, we get that
\begin{equation*}
 s_c \leq \Gamma_{\sqrt{\frac{J(\tau_0)}{d}}}^{1 + \nu^5}
\end{equation*}
Now let $\tau_1\in [0,\tau_0]$ be the largest time such that 
\begin{equation*}
 s_c = \Gamma_{\sqrt{\frac{J(\tau_1)}{d}}}^{1 + \nu^6}
\end{equation*}
Thus, by definition of $\tau_1$ and $\Gamma_b$ \eqref{eq:gammaB1} we obtain that 
\begin{equation*}
 \frac{d}{d\tau}J(\tau_1) \geq 0,
\end{equation*}
while from the monotonicity formula \eqref{eq:second_monotonicity} and \eqref{eq:xitildeTOXI} we obtain 
\begin{equation*}
\begin{aligned}
 \frac{d}{d\tau} J(\tau_1) &\lesssim b(\tau_1)\left(s_c - \Gamma_{b(\tau_1)} - \frac{1}{2}\int |\nabla \xi(\tau_1)|^2 + |\xi(\tau_1)|^2 e^{-|y|}dy + \Gamma_{b(\tau_1)}^{1 - \nu} + \frac{\Gamma_{b(\tau_1)}^{2a}}{b(\tau_1)} \| \nabla \xi(\tau_1)\|_{L^2}^2 \right) \\
 &\lesssim b(\tau_1) \left( \Gamma_{\sqrt{\frac{J(\tau_1)}{d}}}^{1 + \nu^6} - \Gamma_{\sqrt{\frac{J(\tau_1)}{d}}}^{1 + \nu^2}\right) \leq 0,
\end{aligned}
\end{equation*}
which is a contradiction. \newline
The next step will be to prove inequality \eqref{eq:butL1}. From the control \eqref{eq:butb1}, it follows that the parameter $b$ is dynamically trapped around $b_0$, that is, for any $t \in [0,T_1)$, we have the following upper and lower bounds
\begin{displaymath}
 \frac{1 - \nu^4}{1 + \nu^{10}} b_0 \leq b(t) \leq \frac{1 + \nu^4}{1 - \nu^{10}} b_0.
\end{displaymath}
Thus, we notice that the control on the parameter \eqref{eq:bLest}, and inequality \eqref{eq:butH1} imply that for any $ t \in [0,T_1)$
\begin{displaymath}
0 < \frac{1 - 2\nu^4}{1 + \nu^{10}} b_0 \leq - \frac{\dot{\lambda}}{\lambda} = - \frac{1}{2} \frac{d}{dt}\lambda^2 \leq \frac{1 + 2\nu^4}{1 - \nu^{10}} b_0,
\end{displaymath}
where we used that 
\begin{equation*}
 \frac{d}{d\tau} \lambda = \dot{\lambda} = \lambda^2 \frac{d}{dt} \lambda 
\end{equation*}
from \eqref{eq:scaled_t}.
Equivalently, we obtain the law of the parameter $\lambda$
\begin{equation} \label{eq:L_equation}
\lambda_0^2 - 2\left(\frac{1 + 2\nu^4}{1 - \nu^{10}} \right) b_0 t \leq \lambda^2(t) \leq \lambda_0^2 - 2\left(\frac{1 - 2\nu^4}{1 + \nu^{10}} \right) b_0 t.
\end{equation}
in particular, $\lambda(t)$ is a non-increasing function of time. This estimate and the dynamical trapping of the parameter $b$ \eqref{eq:butb1} imply inequality \eqref{eq:butL1}. \newline
The next step is to obtain the bound on the total energy $E$ and momentum $P$ \eqref{eq:butEP1}. We start by deriving a suitable bound for the energy functional. We use \eqref{eq:en_grow} to obtain that
\begin{equation*}
\begin{aligned}
 E[\psi(t)] &= E[\psi_0] + \eta \int_0^t \int |\psi|^{2(\sigma_1 + \sigma_2+1)} - |\psi|^{2\sigma_2} |\nabla \psi|^2 - 2\sigma_2 |\psi|^{2\sigma_2 - 2} Re \left( \bar{\psi}\nabla \psi \right)^2 \, dx \, dv \\
 &\leq E[\psi_0] + \eta \int_0^t \lambda^{2s_c - 2\frac{\sigma_2}{\sigma_1} -2} \int |Q_b + \xi|^{2(\sigma_1 + \sigma_2 +1)}\, dy \, dv.
\end{aligned}
\end{equation*}
We notice that if $d\leq3$ and $ \sigma_2 \leq \sigma_1 < \frac{1}{(d -2)^+} $ or if $d \geq 4$, and $\sigma_2 < \sigma^*$, then $H^1(\R^d) \hookrightarrow L^{2(\sigma_1 + \sigma_2 +1)}(\R^d)$. 
Thus, for $s_c$ small enough we can use the Jensen inequality, \eqref{eq:butH1}, \eqref{eq:conHs} and interpolation 
\begin{equation*}
 \begin{aligned}
 \| \xi \|_{L^{2(\sigma_1 + \sigma_2 +1)}}^{2(\sigma_1 + \sigma_2 +1)} \lesssim \| \xi \|_{\dot H^1}^{2\theta(\sigma_1 + \sigma_2 +1)} \| \xi \|_{\dot H^s}^{2(1 - \theta)(\sigma_1 + \sigma_2 +1)} \lesssim 1
 \end{aligned}
\end{equation*}
for some $\theta(\sigma_1,\sigma_2,s) \in (0,1)$, to obtain 
\begin{equation*}
 \| Q_b + \xi \|_{L^{2(\sigma_1 + \sigma_2 +1)}}^{2(\sigma_1 + \sigma_2 +1)} \leq C \left( \| Q_b\|_{L^{2(\sigma_1 + \sigma_2 +1)}}^{2(\sigma_1 + \sigma_2 +1)} + \| \xi \|_{L^{2(\sigma_1 + \sigma_2 +1)}}^{2(\sigma_1 + \sigma_2 +1)} \right) \leq C
\end{equation*}
where $C = C(\sigma_1,\sigma_2,b^*)>0$. If follows that 
\begin{equation}\label{eq:enDerBu1}
	\begin{split}
			\lambda^{2 - 2s_c} E[\psi(t)] &\leq \lambda^{2 - 2s_c} E[\psi_0] + C \eta \lambda^{2 - 2s_c}\int_0^t \lambda^{2s_c - 2\frac{\sigma_2}{\sigma_1} -2}\, dv.
	\end{split}
\end{equation} 
Now we observe that, from \eqref{eq:L_equation}, there exists a constant $0< c=c(\nu, b_0)\ll1$ such that, for any $t \in [0,T_1)$, 
\begin{equation*}
	\lambda_0^2 - 2(1 + c)t \leq \lambda^2(t) \leq \lambda_0^2 - 2(1 - c) t.
\end{equation*} 
Integrating $\lambda$ in time, it is straightforward to see that for any $\alpha \in \R$, we have
\begin{equation}\label{eq:lamint}
\int_0^t \lambda(\tau)^\alpha\, dv \lesssim \lambda(t)^{\alpha + 2}.
\end{equation}
Thus, from \eqref{eq:enDerBu1}, it follows that
\begin{equation}\label{eq:enDerbu2}
\begin{aligned}
 \lambda^{2 - 2s_c} E[\psi(t)] &\leq \lambda^{2 - 2s_c} E[\psi_0] + C \eta \lambda^{2 -2\frac{\sigma_2}{\sigma_1}}.
\end{aligned}
\end{equation}
From \eqref{eq:butL} and \eqref{eq:initEP} we have
\begin{equation*}
 \lambda^{2 - 2s_c} \left| E[\psi_0] \right| \leq \Gamma_b^{50}.
\end{equation*}
Moreover, we bound the last term using the smallness of $\eta$ \eqref{eq:initEta}, \eqref{eq:butb} and \eqref{eq:butb} 
\begin{equation*}
 C \eta \lambda^{2 -2\frac{\sigma_2}{\sigma_1}} \leq Cs_c^3 \leq \Gamma_b^{3 - 9 \nu}
\end{equation*}
for $\nu$ small enough. 
This implies the bootstrapped control on the energy
\begin{equation*}
		\lambda^{2 - 2s_c} |E[\psi(t)]| \leq \Gamma_b^{3 - 10\nu}.
\end{equation*}
We use the same procedure to bound the momentum functional $P$. From \eqref{eq:momGrow}, we have
\begin{equation*}
\begin{aligned}
 P[\psi(t)] & = P[\psi_0] + 2\eta \int_0^t \int |\psi|^{2\sigma_2} Im ( \bar{\psi} \nabla \psi)\, dx \, d\tau \\
 & = P[\psi_0] + 2\eta \int_0^t \lambda^{-2\frac{\sigma_2}{\sigma_1} + 2s_c -1} \| Q_b + \xi\|_{L^{4\sigma_2 + 2}}^{2\sigma_2 + 1} \| \nabla (Q_b + \xi)\|_{L^2}\, dv.
\end{aligned}
\end{equation*}
Again, if $d\leq3$, $s_c$ is small enough and $ \sigma_2 \leq \sigma_1$ or if $d \geq 4$, and $\sigma_2 < \sigma^*$, then we use use the Jensen inequality, \eqref{eq:butH1}, \eqref{eq:conHs}, interpolation and \eqref{eq:lamint} to prove that there exists $C = C(\sigma_2, b^*)>0$ such that 
\begin{equation}\label{eq:momDerbu2}
\begin{aligned}
 \lambda^{1 - 2s_c} P[\psi(t)] \leq \lambda^{1 - 2s_c} P[\psi_0] + C \eta \lambda^{2 - 2\frac{\sigma_2}{\sigma_1}}.
\end{aligned}
\end{equation}
Consequently, \eqref{eq:initEP}, \eqref{eq:initEta} and \eqref{eq:butb} imply that
\begin{equation*}
 \lambda^{1 - 2s_c} P[\psi] \leq \Gamma_b^{3 - 10\nu},
\end{equation*}
for $\nu$ small enough.
This concludes the proof of \eqref{eq:butEP}. \newline
The last step is to obtain the $\dot{H}^s$-norm control \eqref{eq:butHs1}. We define
\begin{displaymath}
\begin{aligned}
 &\hat{Q}(t,x) = \lambda^{-\frac{1}{\sigma_1}} Q_b\left(\frac{x - x(t)}{\lambda}\right) e^{\ii \gamma}, \\ 
 &\hat{\xi}(t,x) = \lambda^{-\frac{1}{\sigma_1}} \xi\left(t,\frac{x - x(t)}{\lambda}\right) e^{\ii \gamma},
\end{aligned}
\end{displaymath}
we decompose the solution as 
\begin{equation*}
 \psi = \hat{Q}+ \hat{\xi}.
\end{equation*}
From \eqref{eq:main}, we see that the function $\hat{\xi}$ satisfies the equation
\begin{equation}\label{eq:main_tilde}
\ii\hat{\xi}_t + \Delta \hat{\xi} = - \mathcal{E}(\hat{Q}) -N_1(\hat{\xi}) - \ii \eta N_2(\hat{\xi}) , 
\end{equation}
where $\mathcal E$ does not depend on $\hat{\xi}$ 
\begin{align*}
& \mathcal{E}(\hat{Q}) = \ii \partial_t \hat{Q}+ \Delta \hat{Q}+ |\hat{Q}|^{2\sigma_1} \hat{Q}+ \ii \eta |\hat{Q}|^{2\sigma_2}\tilde{Q}
\end{align*}
and 
\begin{align*}
N_1(\hat{\xi}) = |\hat{Q}+ \hat{\xi}|^{2\sigma_1}( \hat{Q}+ \hat{\xi}) - |\hat{Q}|^{2\sigma_1} \hat{Q}, 
\end{align*}
\begin{equation*}
	 N_2(\hat{\xi}) = |\hat{Q}+ \hat{\xi}|^{2\sigma_2}(\hat{Q}+ \hat{\xi}) - |\hat{Q}|^{2\sigma_2}\hat{Q}.
\end{equation*}
Using the equation satisfied by $Q_b$ \eqref{eq:Qb}, we obtain that
\begin{align*}
 \mathcal{E}(\hat{Q}) &= \frac{1}{\lambda^{2 + \frac{1}{\sigma_1}}} e^{\ii \gamma} \left( -\Psi_b + \ii (\dot{b} - \beta s_c) \partial_b Q_b -\ii\left(\frac{\dot{\lambda}}{\lambda} + b \right) \Lambda Q_b -\ii\frac{\dot{x}}{\lambda} \cdot \nabla Q_b\right. \\ &- (\dot{\gamma} -1) Q_b 
+ \left. \ii \eta \lambda^{2-\frac{2\sigma_2}{\sigma_1}} |Q_{b}|^{2\sigma_2}Q_{b} \right).
\end{align*}
We notice that inequality \eqref{eq:butHs1} is equivalent to
\begin{equation}\label{eq:butHs_2}
\int ||\nabla|^{s} \hat{\xi}|^2 \, dx \leq {\lambda^{2(s_c - s)}} \Gamma_b^{1 - 45\nu}
\end{equation}
since 
\begin{equation*}
 \int ||\nabla|^{s} \hat{\xi}|^2\, dx = {\lambda^{2(s_c - s)}} \int ||\nabla|^{s} \xi|^2 dy,
\end{equation*}
and thus we will now prove \eqref{eq:butHs_2}. 
We define for $j=1,2$ where
\begin{equation}\label{eq:strichartz_pair}
r_{j} = \frac{d(2\sigma_{j}+ 2)}{d + 2s\sigma_j}, \quad \gamma_j = \frac{4(\sigma_j+ 1)}{\sigma_j(d - 2s)}, \quad \frac{2}{\gamma_j} = \frac{d}{2} - \frac{d}{r_j}.
\end{equation}
We notice that $(\gamma_j,r_j)$ are Strichartz admissible pairs (see Definition \ref{def:strich}). 
Now we write equation \eqref{eq:main_tilde} in the integral form
\begin{equation*}
 \hat{\xi}(t) = e^{\\i \Delta t} \hat{\xi}(0) + \int_0^t e^{\ii (t-\tau) \Delta} \left(\mathcal{E} (Q_b) + N_1(\hat{\xi}) + \ii \eta N_2(\hat{\xi}) \right) d\tau.
\end{equation*}
and we use the Strichartz estimates in Theorem \ref{thm:strichartz} to obtain
\begin{equation}\label{eq:Hs_strichartz}
 \begin{aligned}
 \| |\nabla|^{s} \hat{\xi}\|_{L^\infty([0,T_1],L^2)} & \lesssim \| |\nabla|^{s} \hat{\xi}_0 \|_{L^2} + \| |\nabla|^{s} \mathcal{E} \|_{L^1([0,T_1],L^2)} \\ &+ \| |\nabla|^{s} N_1(\hat{\xi}) \|_{L^{\gamma_1}[0,T_1],L^{r_1}}
 + \eta \| |\nabla|^{s}N_2(\hat{\xi}) \|_{L^{\gamma_2}([0,t],L^{r_2})}.
 \end{aligned}
\end{equation}
From the initial bound \eqref{eq:initHs}, we have 
\begin{equation*}
 \lambda^{2(s - s_c)} \int ||\nabla|^{s} \hat{\xi}(0)|^2dx = \int ||\nabla|^{s} \xi_0|^2dy \leq \Gamma_{b_0}^{1 - \nu},
\end{equation*}
and thus the bootstrapped controls on b, see \eqref{eq:butb} and $\lambda$, see \eqref{eq:butL} imply
\begin{equation}\label{eq:i_c_Hs}
\| |\nabla|^{s} \hat{\xi}_0 \|_{L^2}^2 \leq \Gamma_{b_0}^{1 - \nu}{ \lambda^{2(s_c - s)}}.
\end{equation} 
We claim that the remaining terms in \eqref{eq:Hs_strichartz} are bounded as
\begin{equation}\label{eq:big_epsilon}
\| |\nabla|^{s} \mathcal{E} \|_{L^1([0,T_1],L^2)}^2 \leq \Gamma_{b}^{1 - 15 \nu}\lambda^{2(s_c - s)},
\end{equation}
and
\begin{equation}\label{eq:nonlinearities_Hs}
\| |\nabla|^{s} N_1(\hat{\xi}) \|_{L^{\gamma_1}([0,T_1],L^{r_1})} + \eta \| |\nabla|^{s}N_2(\hat{\xi}) \|_{L^{\gamma_2}([0,T_1],L^{r_2})} \leq \Gamma_{b}^{\frac{1}{2}(1 - 41 \nu)} \lambda^{s_c-s}.
\end{equation}
Notice that \eqref{eq:big_epsilon} and the bound on $N_1(\hat{\xi}) $ in \eqref{eq:nonlinearities_Hs} has been already proven in \cite[Section $4$]{MeRaSz10}, up to the term coming from the damping in $\mathcal{E}$. In particular, from the estimates on the parameters \eqref{eq:gxest}, \eqref{eq:bLest}, the bound on $\Psi_b$ \eqref{eq:psiB}, and the bootstrap bounds for $b$, $\lambda$ and $\xi$, see \eqref{eq:butb}, \eqref{eq:butL} and \eqref{eq:butH1} respectively, and from the smallness condition on $\eta$ in \eqref{eq:initEta} there holds for any $t \in [0,T_1)$ 
\begin{equation}\label{eq:big_E_bound}
	\begin{split}
		\| |\nabla|^{s} \mathcal{E}(\hat{Q}) \|_{L^2} & \lesssim \frac{1}{\lambda^{2 + s - s_c}} \left( \int |\nabla \xi|^2\, dy+ \int |\xi|^2 e^{-|y|}\, dy+ \Gamma_b^{1 - 11\nu} + \eta\lambda^{2 - 2\sigma_2/\sigma_1} \right)^\frac{1}{2} \\
		& \lesssim \frac{\Gamma_{b}^{\frac{1}{2}(1 - 12 \nu)}}{ \lambda^{2 + s - s_c}}.
	\end{split}
\end{equation}
We use inequality \eqref{eq:lamint} to integrate \eqref{eq:big_E_bound} in time which yields 
\begin{equation*}
\int_0^t \| |\nabla|^{s} \mathcal{E}(\hat{Q}) \|_{L^2} d\tau \lesssim \int_0^t \frac{\Gamma_{b_0}^{\frac{1}{2}(1 - 12 \nu)}}{ \lambda^{2 + (s - s_c)}} d\tau \lesssim \frac{\Gamma_{b_0}^{\frac{1}{2}(1 - 12 \nu)}}{ \lambda^{s - s_c}}. 
\end{equation*}
Lastly, we will show inequality \eqref{eq:nonlinearities_Hs}. The main ingredient is the inequality 
\begin{equation}\label{eq:nonlinearities_finale}
 \| |\nabla|^{s} N_j(\hat{\xi}) \|_{L^{r'_j}} \lesssim \lambda^{(2\sigma_j +1)(s_c - s + 2/\gamma_j)} \| |\nabla|^{s + 2/\gamma_j } \xi \|_{L^2}
\end{equation}
for any $j= 1,2$ which has been proven in \cite[Appendix]{MeRaSz10} and will be shown in Appendix $C$. 
Notice that since
\begin{equation*}
 s + 2/\gamma_j = s + \frac{(d-2s_c)\sigma_j}{2\sigma_j +2},
\end{equation*}
 for any $j=1,2$, then if we choose $s$ to be close enough to $s_c$, we have
$s < s + 2/\gamma_j < 1$. In particular, we can interpolate and use \eqref{eq:butH1} and \eqref{eq:conHs} to obtain 
\begin{displaymath}
\| |\nabla|^{s + 2/\gamma_j} \xi \|_{L^2}^2 \lesssim \| |\nabla|^{s} \xi\|_{L^2}^{2\theta_j} \| \nabla \xi\|_{L^2}^{2(1 - \theta_j)} \leq \Gamma_b^{ 1 -(10 +40 \theta_j) \nu}
\end{displaymath}
where 
\begin{equation*}
 \theta_j = \frac{1 - s - 2/\gamma_j}{1 - s}.
\end{equation*}
For $j=1$, we choose $s_c\ll 1$ small enough and $s$ close enough to $s_c$ to have
\begin{displaymath}
\| |\nabla|^{s + 2/\gamma_j} \xi \|_{L^2} \lesssim \Gamma_{b_0}^{\frac{1}{2} (1 - 40 \nu)}. 
\end{displaymath}
 It follows that
\begin{equation*}
 \| |\nabla|^{s} N_1(\hat{\xi}) \|_{L^{\gamma_1}[0,T_1],L^{r_1}} \lesssim \Gamma_{b_0}^{\frac{1}{2} (1 - 40 \nu)} \left( \int \lambda^{\gamma_1'(2\sigma_1 +1)(s_c - s + 2/\gamma_j)} \right)^\frac{1}{\gamma_1'}.
\end{equation*}
Since
\begin{equation*}
 \gamma_1' (2\sigma_1 + 1) (s + 2/\gamma_1) = 2 + \gamma_1' (s - s_c),
\end{equation*}
we can use \eqref{eq:lamint} to integrate in time and obtain the first inequality in \eqref{eq:nonlinearities_Hs}
\begin{displaymath}
\| |\nabla|^{s} N_1(\hat{\xi}) \|_{L^{\gamma_1}[0,T_1],L^{r_1}} \lesssim \frac{\Gamma_{b_0}^{\frac{1}{2}(1 - 41 \nu)}}{ \lambda^{(s - s_c)}}.
\end{displaymath}
Moreover, we make the same computations for $j =2$ and use the control on $\eta$ \eqref{eq:initEta} and \eqref{eq:butb} to conclude that
\begin{equation}\label{eq:etaSmall}
 \eta \| |\nabla|^{s}N_2(\hat{\xi}) \|_{L^{\gamma_2}([0,T_1],L^{r_2})} \leq \eta \Gamma_b^{ 1 -(10 +40 \theta_2) \nu} \lambda^{s_c-s} \leq \Gamma_b^{ 2 -50 \nu} \lambda^{s_c-s}.
\end{equation}
\end{proof}

Thus Proposition \ref{thm:boot} has been proven for any $t \in [0,T_1)$. Consequently, we can extend by continuity the time interval of the self-similar regime, that is the time interval where the bounds in Proposition \ref{prp:deco} are true. Recursively, we extend the self-similar regime to the whole time interval $[0, T_{max})$ where $T_{max}(\psi_0)$ is the maximal time of existence of the solution stemming from $\psi_0$. Now we show that $\psi$ experiences a collapse in finite time.

\begin{cor}
There exists a time $T_{max}<\infty$ such that
\begin{displaymath}
\lim_{t \rightarrow T_{max}} \lambda(t) = 0.
\end{displaymath}
Moreover, there exists $x_\infty \in \R^d$ such that 
\begin{displaymath}
\lim_{t \rightarrow T_{max}} x(t) \rightarrow x_\infty,
\end{displaymath}
where $x(t)$ is defined in \eqref{eq:y}.
\end{cor}

\begin{proof}
Let $\psi_0 \in \mathcal{O}$ where $\mathcal O$ is defined in \ref{def:deco} and let $T_{max} \leq \infty$ be the maximal time of existence of the corresponding solution to \eqref{eq:main}. Then Proposition \ref{thm:boot} implies that for any $t\in [0,T_{max})$, the bounds on the scaling parameter in \eqref{eq:L_equation} are true, namely, there exists $0< c=c(\nu, b_0)\ll1$ such that
\begin{equation}\label{eq:lambEq}
	\lambda_0^2 - 2(1 + c)t \leq \lambda^2(t) \leq \lambda_0^2 - 2(1 - c) t.
\end{equation}
Consequently, there exists a time $0<T=T(\lambda_0,\nu,b_0) <\infty $ such that $\lim_{t \to T} \lambda(t) = 0$. Furthermore, by using decomposition \eqref{eq:deco1} we also obtain that 
\begin{equation*}
 \lim_{t \to T} \| \nabla \psi(t)\|_{L^2}^2 = \lim_{t \to T} \lambda^{2(s_c - 1)}(t) \| \nabla (Q_{b(t)} + \xi(t))\|_{L^2}^2 = \infty.
\end{equation*}
Indeed, in the limit above, the exponent $2(s_c - 1)$ is negative because $s_c <1$, and for any $t \in [0,T]$, $Q_{b(t)} \in H^1(\R^d)$,
\begin{equation*}
 \big\| \nabla Q_{b(t)} \big\|_{L^2} \geq C >0,
\end{equation*}
while \eqref{eq:butb1} and \eqref{eq:butH11} imply that 
\begin{equation*}
 \big\| \nabla \xi(t) \big\|_{L^2}^2 \leq \Gamma_{b}^{1 - 10\nu} \lesssim s_c^\frac{1 - 10\nu}{1 + \nu^4}.
\end{equation*}
in particular, the maximal time of existence is given by $T < \infty$. Finally, we prove the convergence to a blow-up point. By exploiting \eqref{eq:gxest} and \eqref{eq:lambEq}, we get
\begin{displaymath}
\begin{aligned}
 &\big|x(T) - x(t)\big| = \left|\int_t^{T} \frac{dx}{ds} ds \right| \leq \int_t^{T} \frac{1}{\lambda} \left| \frac{1}{\lambda} \frac{d}{d\tau} x \right| ds \\
 & \lesssim \left\|\left( \int |\nabla \xi(t)|^2 + |\xi(t)|^2 e^{-|y|}\, dy\right)^\frac{1}{2} + \Gamma_{b(t)}^{1 -20\nu} \right\|_{L^\infty([0,T])} \int_t^{T} \left(\lambda_0^2 - 2(1 +c) s\right)^{-\frac{1}{2}} ds.
\end{aligned}
\end{displaymath}
We use \eqref{eq:butb} and \eqref{eq:butH1} to obtain
\begin{equation*}
 \left\|\left( \int |\nabla \xi(t)|^2 + |\xi(t)|^2 e^{-|y|}\, dy\right)^\frac{1}{2} + \Gamma_{b(t)}^{1 -20\nu} \right\|_{L^\infty([0,T])} \lesssim 1.
\end{equation*}
Moreover, we compute the integral as follows
\begin{equation*}
 \int_t^{T} \left(\lambda_0^2 - 2(1 +c) s\right)^{-\frac{1}{2}} ds = -\frac{2}{2 + c} (\lambda_0^2 - 2(1 +c) s)^{\frac{1}{2}}|_t^T.
\end{equation*}
 This implies that 
\begin{equation*}
 \lim_{t \to T} \big|x(T) - x(t)\big| \lesssim \lim_{t \to T} (\lambda_0^2 - 2(1 +c) s)^{\frac{1}{2}}|_t^T = 0.
\end{equation*}
\end{proof}

\section{The case $\sigma_2 < \sigma_1$} \label{sec:concl}

In this section, we discuss the case where the exponent of the damping term is strictly smaller than that of the power-type nonlinearity. \newline
First, we observe that the condition which we have chosen for the damping parameter $\eta \leq s_c^3$ is not strictly necessary. In fact, this condition was used to prove that the damping term is of smaller order with respect to other terms while the solution is in the self-similar regime. In particular, it was used for instance in \eqref{eq:xiSigma2}, \eqref{eq:toketa}, \eqref{eq:big_E_bound}, \eqref{eq:etaSmall}, \eqref{eq:enDerbu2} and \eqref{eq:momDerbu2}, where the damping term contributions can be always bounded by
\begin{equation*}
 \eta \lambda^{2 - 2\frac{\sigma_2}{\sigma_1}}(t) \left(C(Q_{b(t)}) + c \left(\int |\nabla \xi(t)|^2 + |\xi(t)|^2 e^{-|y|} dy\right)^{\frac{1}{2}} \right),
\end{equation*}
where $C(Q_b) >0$ is a constant depending only on $Q_b$. We know that in the self-similar regime, from \eqref{eq:conxi}, we have that
\begin{equation*}
 \left(\int |\nabla \xi(t)|^2 + |\xi(t)|^2 e^{-|y|} dy\right)^{\frac{1}{2}} \lesssim \Gamma_{b(t)}^{\frac{1}{2} - 10 \nu} \lesssim C(Q_{b(t)}) 
\end{equation*}
for $b(t)$ small enough. Thus, in order to prove that the damping term is smaller than the leading order terms in the equation referred to above, which are of the size $\Gamma_b$, we shall have for example that
\begin{equation*}
 \eta \lambda^{2 - 2\frac{\sigma_2}{\sigma_1}}(t) \leq \Gamma_{b(t)}^{2}
\end{equation*}
for any $t \in [0,T_1)$. This is equivalent to asking that the scaling parameter satisfies the inequality 
 \begin{equation*}
\lambda(t) < \eta^{-1} \Gamma_{b(t)}^\frac{\sigma_1}{\sigma_1 - \sigma_2}. 
 \end{equation*}
Notice that the right-hand side of this inequality degenerates for $\sigma_2 \to \sigma_1$, and consequently, for $\sigma_2 = \sigma_1$, the supposition on the smallness of $\eta$ becomes necessary to carry on with the bootstrap. On the other hand, for $\sigma_2 < \sigma_1$, we can use this condition to proceed with the bootstrap instead of that in \eqref{eq:conL}. We observe that in this case, the set of initial conditions would depend on $\eta$, $\sigma_1$ and $\sigma_2$. For instance, it would be necessary to choose the initial condition satisfying, say
\begin{equation*}
 \lambda_0 < \eta^{-1} \Gamma_{b(t)}^\frac{10 \sigma_1}{\sigma_1 - \sigma_2}
\end{equation*}
instead of the weaker bound \eqref{eq:initL}. In particular, the bigger $\eta$ is and the closer $\sigma_2$ is to $\sigma_1$, the more focused the initial condition needs to be to experience a collapse in finite time. This implies the following.
\begin{theorem}\label{thm:mainBU2}
There exists $s_c^* >0$ such that for any $0 < s_c < s_c^*$, any $ \sigma_*< \sigma_2 < \sigma_1$ when $d \leq 3$ or $ \sigma_*< \sigma_2 < \sigma^*$ when $d \geq 4$ and any $\eta >0$ there exists a set $\mathcal{O}_{\eta,\sigma_1,\sigma_2} \subset H^1(\R^d) $ such that if $\psi_0 \in \mathcal{O}_{\eta,\sigma_1,\sigma_2}$ then the corresponding solution $\psi \in C([0,T_{max}),H^1(\R^d))$ to \eqref{eq:main} develops a singularity in finite time.
\end{theorem}

We emphasize again that our argument works only when $\sigma_1 >\sigma_2$. When $\sigma_1 = \sigma_2$, it is crucial that $\sigma_1 > \frac{2}{d}$ and $\eta$ to be small enough. On the other hand, it is already known that when $\sigma_1 = \sigma_2 = \frac{2}{d}$, solutions are global in time. This means that a solution escapes the self-similar regime no matter how small $\eta$ and $\lambda_0$ are. Heuristically, this is explained by the fact that, in the self-similar regime, when $ \sigma_1 = \frac{2}{d}$, the control parameter converges to zero in time, that is $b(t) \to 0$ for $t \to T_{max}$ (see \cite{MeRa05} for example). Consequently, for $\sigma_2 = \sigma_1 =\frac{2}{d}$, there would exist a time $T = T(\eta, b_0, \lambda_0) >0$ such that for some $t > T$, the critical inequality $\eta \lesssim \Gamma_{b(t)}$ is not true. Thus the damping term would become of the leading order inside the self-similar dynamics and interrupt it.
\newline 
As our final remark, we notice that if $\sigma_2> \sigma_1$, then heuristically the self-similar regime would be interrupted. Indeed the convergence of $\lambda$ to zero implies that contributions of damping terms will eventually be of the leading order in the dynamics because 
 \begin{equation*}
 \lambda^{2 - 2\frac{\sigma_2}{\sigma_1}}(t) \to \infty.
 \end{equation*}

\section{Appendix}

\subsection{Appendix A}\label{sec:appenA}

In this appendix we will report the proof of \cite[Proposition $3.3$]{MeRaSz10} that contains computations used in Lemma \ref{lem:vir1}. We state the result below for the reader's convenience.

\begin{lem}\label{lem:locvirEta}
 Suppose that $\eta = 0$. Then there exists $s_c^{(1)} > 0$ such that for any $s_c < s_c^{(1)}$ and for any $t \in [0,T_1)$, there exists $C>0$ such that 
\begin{equation}\label{eq:virNoEta}
\dot{b} \geq C\left( s_c + \int |\nabla \xi|^2 + |\xi|^2 e^{-|y|}dy - \Gamma_b^{1 - \nu^6}\right).
\end{equation}
\end{lem}

We now present some preliminary computations. 
We start with the following.

\begin{lem}\label{lem:vir11}
 Let $Q_b$ be a solution to \eqref{eq:Qb}. Then
 \begin{equation}\label{eq:vir11}
 ( \ii \partial_b Q_b, \Lambda Q_b) = -\frac{1}{4} \| xQ\|_{L^2}^2(1 +\delta_1(s_c,b)) + s_c(\ii Q_b, \partial_b Q_b),
\end{equation}
where $\delta_1(s_c,b) \to 0 $ as $s_c, b \to 0$ and $Q$ is the unique positive solution to \eqref{eq:gsQ}.
\end{lem}

\begin{proof}
 We recall the third property in \eqref{eq:Qbprop}
\begin{equation}\label{eq:b11}
 \begin{aligned}
 (\ii Q_b, x \cdot \nabla Q_b) = ( \ii Q_b, \Lambda Q_b) = - \frac{b}{2} \| xQ\|_{L^2}^2(1 +\delta_1(s_c,b))
 \end{aligned}
\end{equation}
where $\delta_1(s_c,b) \to 0 $ as $s_c, b \to 0$. 
By taking the derivative of \eqref{eq:b11} with respect to $b$ we obtain 
\begin{equation*}
 \frac{d}{db} ( \ii Q_b, \Lambda Q_b) = 
( \ii \partial_b Q_b, \Lambda Q_b) + ( \ii Q_b, \Lambda \partial_b Q_b) =
-\frac{1}{2} \| xQ\|_{L^2}^2(1 +\delta_1(s_c,b)).
\end{equation*}
Let us also recall identity \eqref{eq:LambLap},
\begin{equation*}
 ( \ii Q_b, \Lambda \partial_b Q_b) = - 2s_c(\ii Q_b, \partial_b Q_b) + ( \ii \partial_b Q_b, \Lambda Q_b).
\end{equation*}
By plugging it into the previous formula, we obtain
\begin{equation*}
 ( \ii \partial_b Q_b, \Lambda Q_b) = -\frac{1}{4} \| xQ\|_{L^2}^2(1 +\delta_1(s_c,b)) + s_c(\ii Q_b, \partial_b Q_b).
\end{equation*}
\end{proof}

Next, we recall the definition of the linearized operator around $Q_b$, see \eqref{eq:linQb}
\begin{equation}\label{eq:lin_Qb_App}
\mathcal{L}\xi = \Delta \xi - \xi + \ii b \Lambda \xi + 2\sigma_1 Re(\xi \bar{Q}_b) |Q_b|^{2\sigma_1 - 2} Q_b + |Q_b|^{2\sigma_1} \xi, 
\end{equation}
and the nonlinear remainder term in the equation \eqref{eq:xi1} for the perturbation, namely
\begin{displaymath}
R(\xi) = |Q_b + \xi|^{2\sigma_1}(Q_b + \xi) - |Q_b|^{2\sigma_1} Q_b - 2\sigma_1 Re(\xi \bar{Q}_b) |Q_b|^{2\sigma_1 - 2} Q_b 
- |Q_b|^{2\sigma_1} \xi.
\end{displaymath}

\begin{lem}
 We have that 
 \begin{equation}\label{eq:vir12}
 \begin{aligned}
 (\mathcal{L} \xi + R(\xi), \Lambda Q_b) &= -2s_c b(\ii \xi, \Lambda Q_b) - \beta s_c(\xi, \ii \Lambda \partial_b Q_b) - (\xi, \Lambda \Psi_b) \\
 &- 2\lambda^{2 - 2s_c} E[\psi] + 2E[Q_b] + \mathcal H (\xi,\xi) + \mathcal{E}(\xi,\xi) \\
 &+ \frac{1}{\sigma_1 + 1}\int R^{(3)}(\xi)\, dy+ (R_3(\xi), \Lambda Q_b).
 \end{aligned}
 \end{equation}
where $\mathcal{H}$ is the quadratic form defined in \eqref{eq:quadform}, $\mathcal{E}(\xi,\xi)$ satisfies the following bound
\begin{equation*}
 \left|\mathcal{E}(\xi,\xi) \right|\leq \delta_2(s_c) \int |\xi|^2 e^{-|y|} + |\nabla \xi|^2 dy,
\end{equation*}
with $\delta_2(s_c) \to 0$ as $s_c \to 0$ and the remainder terms $R^{(3)}(\xi)$, $R_3(\xi)$, are given by
\begin{equation}\label{eq:R3}
 \begin{aligned}
 R^{(3)}(\xi) = |Q_b + \xi|^{2\sigma_1 + 2} - |Q_b|^{2\sigma_1 + 2} - (2\sigma_1 + 2)|Q_b|^{2\sigma_1} Re(Q_b \bar{\xi}) \\
 - (2\sigma_1 + 2) |Q_b|^{2\sigma_1 -2}\left(|Q_b|^2 |\xi|^2 + 2\sigma_1 Re(Q_b \bar{\xi})^2 \right),
 \end{aligned}
\end{equation}
and 
\begin{equation}\label{eq:R_3}
\begin{aligned}
 R_3(\xi) = R(\xi) - 2\sigma_1 |Q_b|^{2\sigma_1 -2} \bigg( Re(Q_b \bar{\xi})\xi + (2\sigma_1 - 1) |Q_b|^{-2} Re(Q_b \bar{\xi})^2 Q_b + Q_b|\xi|^2 \bigg).
\end{aligned}
\end{equation}
\end{lem}

\begin{proof}
 We observe that by using properties \eqref{eq:Lambdapp}, \eqref{eq:LambLap} we have 
\begin{equation*}
 (\Delta \xi, \Lambda Q_b) = 2(\xi, \Delta Q_b) + ( \xi, \Lambda \Delta Q_b) 
\end{equation*}
and
\begin{equation*}
 ( \ii b \Lambda \xi, \Lambda Q_b) = - 2s_c b(\ii \xi, \Lambda Q_b) + (\xi, \ii b \Lambda (\Lambda Q_b)). 
\end{equation*}
Consequently, by using the equation \eqref{eq:LambQb} satisfied by $\Lambda Q_b$ and the definition \eqref{eq:lin_Qb_App}, we obtain that
\begin{equation*}
\begin{aligned}
 (\mathcal{L} \xi, \Lambda Q_b) & = 2(\xi, \Delta Q_b) - 2s_c b(\ii \xi, \Lambda Q_b) \\
 & + (\xi, \Lambda \Delta Q_b - \Lambda Q_b + \ii b \Lambda(\Lambda Q_b) + |Q_b|^{2\sigma_1} \Lambda Q_b + 2\sigma_1 Re(\bar Q_b \Lambda Q_b) |Q_b|^{2\sigma_1 - 2} Q_b) \\
 &= 2(\xi, \Delta Q_b) - 2s_c b(\ii \xi, \Lambda Q_b) + 2\sigma_1 (\xi,Re(\bar Q_b, \frac{1}{\sigma_1} Q_b) |Q_b|^{2\sigma_1 - 2} Q_b) \\
 &+ (\xi, \Lambda \Delta Q_b - \Lambda Q_b + \ii b \Lambda(\Lambda Q_b) + |Q_b|^{2\sigma_1} \Lambda Q_b + 2\sigma_1 Re(\bar Q_b \, y \cdot \nabla Q_b) |Q_b|^{2\sigma_1 - 2} Q_b) \\ 
 &= 2(\xi, \Delta Q_b + |Q_b|^{2\sigma_1} Q_b) - 2s_c b(\ii \xi, \Lambda Q_b) -\beta s_c (\xi, \ii \Lambda \partial_b Q_b) - (\xi, \Lambda \Psi_b).
 \end{aligned}
\end{equation*}
Now for the first term on the right-hand side of the equation above we notice that \eqref{eq:enLong} implies
\begin{align*}
 2(\Delta Q_b + |Q_b|^{2\sigma_1} Q_b, \xi) = - 2\lambda^{2 - 2s_c} E[\psi] + 2E[Q_b] + \| \nabla \xi\|_{L^2}^2 - \int R^{(2)}(\xi) dy,
\end{align*}
where \begin{equation*}
R^{(2)}(\xi) = \frac{1}{\sigma_1 + 1} \left(|Q_b + \xi|^{2\sigma_1 + 2} - |Q_b|^{2\sigma_1 + 2} - (2\sigma_1 + 2)|Q_b|^{2\sigma_1} Re(Q_b \bar{\xi})\right).
\end{equation*}
Thus we arrive to the preliminary equation
\begin{align*}
 (\mathcal{L} \xi + R(\xi), \Lambda Q_b) = -2s_c b(\ii \xi, \Lambda Q_b) - \beta s_c(\xi, \ii \Lambda \partial_b Q_b) - (\xi, \Lambda \Psi_b) \\
 - 2\lambda^{2 - 2s_c} E[\psi] + 2E[Q_b] + \| \nabla \xi\|_{L^2}^2 - \int R^{(2)}(\xi)\, dy+ (R(\xi), \Lambda Q_b).
\end{align*}
It remains to extract the quadratic terms in $-\int R^{(2)}(\xi)\, dy+ (R(\xi), \Lambda Q_b)$. We write
\begin{equation*}
 \int R^{(2)}(\xi)\, dy= 2 \int |Q_b|^{2\sigma_1 -2}\left(|Q_b|^2 |\xi|^2 + 2\sigma_1 Re(Q_b \bar{\xi})^2 \right) + \frac{1}{\sigma_1 + 1} \int R^{(3)}(\xi)\, dy
\end{equation*}
where $R^{(3)}(\xi)$ is the rest defined in \eqref{eq:R3}. Then we also write that
\begin{equation*}
 \begin{aligned}
 (R(\xi), \Lambda Q_b) &= (R_3(\xi), \Lambda Q_b) \\
 & + 2\sigma_1 \bigg( |Q_b|^{2\sigma_1 -2} \bigg( Re(Q_b \bar{\xi})\xi + (2\sigma_1 - 1) |Q_b|^{-2} Re(Q_b \bar{\xi})^2 Q_b + Q_b|\xi|^2 \bigg) , \\ 
 & \frac{1}{\sigma_1} Q_b + y \cdot \nabla Q_b \bigg)
 \end{aligned}
\end{equation*}
where $R_3(\xi)$ is defined in \eqref{eq:R_3}. In the equation above, we notice that
\begin{equation*}
\begin{aligned}
 &2\sigma_1 \bigg( |Q_b|^{2\sigma_1 -2} \bigg( Re(Q_b \bar{\xi})\xi + (2\sigma_1 - 1) |Q_b|^{-2} Re(Q_b \bar{\xi})^2 Q_b + Q_b|\xi|^2 \bigg) , \frac{1}{\sigma_1} Q_b\bigg) \\ 
 &= 2 \int |Q_b|^{2\sigma_1 -2}\left(|Q_b|^2 |\xi|^2 + 2\sigma_1 Re(Q_b \bar{\xi})^2 \right) \, dy
\end{aligned}
\end{equation*}
that is we can write that
\begin{equation*}
\begin{aligned}
 -\int R^{(2)}(\xi)\, dy+ (R(\xi), \Lambda Q_b) & = - \frac{1}{\sigma_1 + 1} \int R^{(3)}(\xi)\, dy+ (R_3(\xi), \Lambda Q_b)\\
 &+ 2\sigma_1 \bigg( |Q_b|^{2\sigma_1 -2} \bigg( Re(Q_b \bar{\xi})\xi + (2\sigma_1 - 1) |Q_b|^{-2} Re(Q_b \bar{\xi})^2 Q_b \\ 
 & + Q_b|\xi|^2 \bigg),y \cdot \nabla Q_b \bigg).
\end{aligned}
\end{equation*}
Consequently, by replacing $Q_b$ with the ground state $Q$ generating an error which we denote by $\mathcal{E}(\xi,\xi)$, one obtains that 
\begin{equation*}
\begin{aligned}
 \| \nabla \xi\|_{L^2}^2 - \int R^{(2)}(\xi)\, dy+ (R(\xi), \Lambda Q_b) &= \mathcal H (\xi,\xi) + \mathcal{E}(\xi,\xi) \\
 &- \frac{1}{\sigma_1 + 1} \int R^{(3)}(\xi)\, dy+ (R_3(\xi), \Lambda Q_b)
\end{aligned}
\end{equation*}
where $\mathcal{H}(\xi,\xi)$ is the quadratic form defined in \eqref{eq:quadform}. For $s_c$ small enough, we also obtain that $\mathcal{E}(\xi,\xi)$ can be bounded by 
\begin{equation}\label{eq:mathE}
 \left| \mathcal{E}(\xi,\xi) \right| \leq \delta_2(s_c) \int |\xi|^2 e^{-|y|} + |\nabla \xi|^2 dy,
\end{equation}
where $ \delta_2(s_c) \to 0$ as $s_c \to 0$. 
\end{proof}

For details of the proof above, we refer to \cite[Proposition $3.3$]{MeRaSz10}. 
We are now ready to prove the local virial property \ref{lem:vir1} when $\eta = 0$.

\begin{proof}[Proof of Lemma \ref{lem:locvirEta}.]
By taking the scalar product of equation \eqref{eq:xi1} with $\Lambda Q_b$, we obtain
\begin{equation}\label{eq:bLong}
\begin{aligned}
 0 &= (\ii \partial_\tau \xi, \Lambda Q_b)+ \dot{b} (\ii \partial_b Q_b, \Lambda Q_b) + \left(2E[Q_b] - \left( \ii \beta s_c \partial_b Q_b + \Psi_b, \Lambda Q_b \right) \right) \\ 
 & +(\mathcal{L} \xi + R(\xi), \Lambda Q_b) - 2E[Q_b] \\
 & - \left( \ii \left( \frac{\dot{ \lambda} }{\lambda} + b \right) \Lambda (Q_b + \xi) + \ii\frac{\dot{ x}}{\lambda} \cdot \nabla Q_b + (\dot{ \gamma } -1) Q_b, \Lambda Q_b\right)\\
 &- \left( \ii \left( \frac{\dot{ \lambda} }{\lambda} + b \right) \Lambda \xi + (\dot{ \gamma } -1) \xi + \ii\frac{\dot{ x}}{\lambda} \cdot \nabla \xi, \Lambda Q_b\right). 
\end{aligned}
\end{equation}
We will study the contributions of the terms in \eqref{eq:bLong} separately. For the first term, we observe that
\begin{equation}\label{eq:ft1}
 (\ii \partial_\tau \xi, \Lambda Q_b) = \frac{d}{d\tau} (\ii \xi, \Lambda Q_b) - (\ii \xi, \partial_\tau \Lambda Q_b) = \frac{d}{d\tau} (\ii \xi, \Lambda Q_b) - \dot b (\ii \xi, \Lambda \partial_b Q_b).
\end{equation}
For the second term, we use \eqref{eq:vir11} to get 
\begin{equation*}
 \dot b ( \ii \partial_b Q_b, \Lambda Q_b) = -\dot b\frac{ 1}{4} \| xQ\|_{L^2}^2(1 +\delta_1(s_c,b)) + s_c(\ii Q_b, \partial_b Q_b).
\end{equation*}
 Moreover, from \eqref{eq:prod2} and \eqref{eq:conxi} , we bound the second term on the right-hand side of \eqref{eq:ft1} as
\begin{equation*}
 \left|(\ii \xi, \Lambda \partial_b Q_b) \right| \lesssim \left( \int |\xi|^2 e^{-|y|} + |\nabla \xi|^2\, dy\right)^\frac{1}{2} \leq \Gamma_b^{\frac{1}{2} - 10\nu},
\end{equation*}
which implies that there exists $c>0$ such that
\begin{equation*}\begin{aligned}
 \dot{b} &\left((\ii \partial_b Q_b, \Lambda Q_b) - (\ii \xi, \Lambda \partial_b Q_b) \right) = \dot{b} s_c(\partial_b Q_b, \ii Q_b) - \dot{b}\left(\frac{\| y Q\|_{L^2}^2}{4} (1 + \delta_1(s_c,b)) - c \Gamma_b^{\frac{1}{2} - 10\nu}\right). 
 \end{aligned}
\end{equation*}
For the third term in \eqref{eq:bLong}, we use the Pohozaev-type estimate \eqref{eq:poho} to obtain 
\begin{equation*}
 2E[Q_b] - \left( \ii \beta s_c \partial_b Q_b + \Psi_b, \Lambda Q_b \right) = s_c(2 E[Q_b] + \| Q_b\|_{L^2}^2).
\end{equation*}
For the fourth term, we use \eqref{eq:vir12} to obtain that 
\begin{align*}
 (\mathcal{L} \xi + R(\xi), \Lambda Q_b) - 2E[Q_b] & = -2s_c b(\ii \xi, \Lambda Q_b) - \beta s_c(\xi, \ii \Lambda \partial_b Q_b) - (\xi, \Lambda \Psi_b) \\
 &- 2\lambda^{2 - 2s_c} E[\psi] + \mathcal H (\xi,\xi) + \mathcal{E}(\xi,\xi) \\
 &- \frac{1}{\sigma_1 + 1}\int R^{(3)}(\xi) dy+ (R_3(\xi), \Lambda Q_b)
\end{align*}
where $R^{(3)}$ is defined in \eqref{eq:R3}. For the fifth term in \eqref{eq:bLong}, by straightforward computations we get
\begin{equation*}
 \begin{aligned}
 & -\left( \ii \left( \frac{\dot{ \lambda} }{\lambda} + b \right) \Lambda Q_b + \ii\frac{\dot{ x}}{\lambda} \cdot \nabla Q_b + (\dot{ \gamma } -1) Q_b, \Lambda Q_b\right) = s_c (\dot{\gamma} - 1) \| Q_b\|_{L^2}^2.
 \end{aligned}
\end{equation*}
By combining the previous computations, we have
\begin{equation}\label{eq:bVeryLong}
\begin{aligned}
& - \dot{b} \left((\ii \partial_b Q_b, \Lambda Q_b) - (\ii \xi, \Lambda \partial_b Q_b) \right) = \frac{d}{d\tau} (\ii \xi, \Lambda Q_b) \\
 &+ s_c\left( - 2b(\ii \xi, \Lambda Q_b) - \beta (\xi, \ii \Lambda \partial_b Q_b) + 2E[Q_b] + \| Q_b\|_{L^2}^2 + (\dot{\gamma} - 1)\| Q_b\|_{L^2}^2 \right) \\
 & - (\xi, \Lambda \Psi_b) - 2\lambda^{2 - 2s_c} E[\psi] + \mathcal H (\xi,\xi) + \mathcal{E}(\xi,\xi) + \frac{1}{\sigma_1 + 1}\int R^{(3)}(\xi)\, dy\\
 &- \left( \ii \left( \frac{\dot{ \lambda} }{\lambda} + b \right) \Lambda \xi + (\dot{ \gamma } -1) \xi + \ii\frac{\dot{ x}}{\lambda} \cdot \nabla \xi, \Lambda Q_b\right),
\end{aligned}
\end{equation}
or equivalently 
\begin{equation}\label{eq:bVeryLong22}
\begin{aligned}
 &\dot{b}\left(\frac{\| y Q\|_{L^2}^2}{4} (1 + \delta_1(s_c,b)) - c \Gamma_b^{\frac{1}{2} - 10\nu}\right) \geq \frac{d}{d\tau} (\ii \xi, \Lambda Q_b) \\
 &+ s_c\left( \dot{b} (\partial_b Q_b, \ii Q_b) - 2b(\ii \xi, \Lambda Q_b) - \beta (\xi, \ii \Lambda \partial_b Q_b) + 2E[Q_b] + \| Q_b\|_{L^2}^2 + (\dot{\gamma} - 1)\| Q_b\|_{L^2}^2 \right) \\
 & - (\xi, \Lambda \Psi_b) - 2\lambda^{2 - 2s_c} E[\psi] + \mathcal H (\xi,\xi) + \mathcal{E}(\xi,\xi) + \frac{1}{\sigma_1 + 1}\int R^{(3)}(\xi)\, dy\\
 &- \left( \ii \left( \frac{\dot{ \lambda} }{\lambda} + b \right) \Lambda \xi + (\dot{ \gamma } -1) \xi + \ii\frac{\dot{ x}}{\lambda} \cdot \nabla \xi, \Lambda Q_b\right). 
\end{aligned}
\end{equation}
We will now study term by term the right-hand side of \eqref{eq:bVeryLong22}. 
The first term vanishes for any $t \in [0,T_1)$ because of one orthogonality condition in \eqref{eq:ortho4}. Furthermore, by choosing $s_c$ small enough, we notice that 
\begin{equation*}
 \left(\frac{\| y Q\|_{L^2}^2}{4} (1 + \delta_1(s_c,b)) - c \Gamma_b^{\frac{1}{2} - 10\nu}\right) \geq \frac{ \| y Q\|_{L^2}^2}{8}. 
\end{equation*}
For the second term, we observe that by using estimates \eqref{eq:conb}, \eqref{eq:conxi} and \eqref{eq:bLest}, we obtain that 
\begin{equation*}
\begin{aligned}
 &s_c\left( \dot{b} (\partial_b Q_b, \ii Q_b) - 2b(\ii \xi, \Lambda Q_b) - \beta (\xi, \ii \Lambda \partial_b Q_b) + 2E[Q_b] + \| Q_b\|_{L^2}^2 + (\dot{\gamma} - 1)\| Q_b\|_{L^2}^2 \right) \\ &\geq s_c \frac{\| Q_b\|_{L^2}^2}{2},
\end{aligned}
\end{equation*}
for $b$ small enough. Furthermore, we use \eqref{eq:prod2} to bound the term 
\begin{equation*}
 \left| (\xi, \Lambda \Psi_b) \right| \lesssim \Gamma_b^{1 - \nu},
\end{equation*}
and the control on the energy \eqref{eq:conEP} to bound
\begin{equation*}
 \left|2\lambda^{2 - 2s_c} E[\psi] \right| \lesssim \Gamma_b^2. 
\end{equation*}
All the rests are bounded using \eqref{eq:mathE} and \eqref{eq:prod1} 
\begin{equation*}
 \left|\mathcal{E}(\xi,\xi) - \frac{1}{\sigma_1 + 1}\int R^{(3)}(\xi)\, dy+(R_3(\xi), \Lambda Q_b) \right| \leq \delta_4(s_c) \int |\xi|^2 e^{-|y|} + |\nabla \xi|^2\, dy
\end{equation*}
for some $\delta_4(s_c) >0$ with $\delta_4(s_c) \to 0$ as $s_c \to 0$. For the last term on the right-hand side of \eqref{eq:bVeryLong}, we notice that using \eqref{eq:Lambdapp} and two of the orthogonality conditions in \eqref{eq:ortho4}, we have
\begin{equation*}
 (\ii \Lambda \xi, \Lambda Q_b) = - (\ii \xi, \Lambda(\Lambda Q_b)) - 2s_c (\ii \xi, \Lambda Q_b) = 0.
\end{equation*}
Moreover, by using \eqref{eq:gxest}, Cauchy-Schwartz inequality and \eqref{eq:conxi}, we obtain
\begin{equation*}
\begin{aligned}
 \left| \left( \ii \frac{\dot{ x}}{\lambda} \cdot \nabla \xi, \Lambda Q_b \right)\right| &\lesssim \| \nabla \xi\|_{L^2} \left(\delta_2(s_c) \left( \int |\nabla \xi|^2 + |\xi|^2 e^{-|y|}\, dy\right)^\frac{1}{2} + \int |\nabla \xi|^2 dy+ \Gamma_b^{1 -11\nu}\right) \\
 &\lesssim \delta_2(s_c) \Gamma_b^{1 - 20\nu} + \Gamma_b^{\frac{3}{2} - 31\nu}.
 \end{aligned}
\end{equation*}
In the same way, we observe that
\begin{equation*}
\begin{aligned}
 - \left( (\dot{ \gamma } -1) \xi , \Lambda Q_b\right) & = - \left( \left( (\dot{ \gamma} - 1) - \frac{1}{\| \Lambda Q_b\|_{L^2}^2} (\xi, \mathcal{L} \Lambda( \Lambda Q_b) \right) \xi , \Lambda Q_b \right) \\ 
 &- \left( \left( \frac{1}{\| \Lambda Q_b\|_{L^2}^2} (\xi, \mathcal{L} \Lambda( \Lambda Q_b) \right) \xi , \Lambda Q_b \right)
 \end{aligned}
\end{equation*}
and we use again \eqref{eq:gxest} to obtain that 
\begin{equation*}
\begin{aligned}
 \left|- \left( \left( (\dot{ \gamma} - 1) - \frac{1}{\| \Lambda Q_b\|_{L^2}^2} (\xi, \mathcal{L} \Lambda( \Lambda Q_b) \right) \xi , \Lambda Q_b \right)\right| \lesssim \Gamma_b.
 \end{aligned}
\end{equation*}
Now we use the coercivity property in Proposition \ref{prp:quadForm}
to obtain that there exists $c>0$ such that
\begin{equation*}
\begin{aligned}
 H(\xi,\xi) - \left( (\dot{ \gamma } -1) \xi , \Lambda Q_b\right) & \geq c \int|\xi|^2e^{-|y|} + |\nabla \xi|^2\, dy\\
 &- c \big( (\xi,Q)^2 + (\xi, |y|^2Q)^2 + (\xi,yQ)^2 \\
 &+ (\xi, \ii DQ)^2 + (\xi, \ii D(DQ))^2 + (\xi, \ii \nabla Q)^2 \big),
 \end{aligned}
\end{equation*}
where $D\xi$ is defined in \eqref{eq:LambdaD}. By using property \eqref{eq:Lambda-D}, \eqref{eq:conxi} and the closeness of $Q_b$ with $Q$, we have that there exists a constant $\delta_5(s_c) >0$ such that
\begin{equation*}
\begin{aligned}
 &(\xi,Q)^2 + (\xi, |y|^2Q)^2 + (\xi,yQ)^2 + (\xi, \ii DQ)^2 + (\xi, \ii D(DQ))^2 + (\xi, \ii \nabla Q)^2 \\
 &\geq(\xi,Q_b)^2 + (\xi, |y|^2Q_b)^2 + (\xi,yQ_b)^2 + (\xi, \ii \Lambda Q_b)^2 + (\xi, \ii \Lambda(\Lambda Q_b))^2 + (\xi, \ii \nabla Q_b)^2 \\
 &- \delta_5(s_c)(s_c + \Gamma_b)
\end{aligned}
\end{equation*}
where $\delta_5(s_c) \to 0$ as $s_c \to 0$. Finally, we use the orthogonality conditions \eqref{eq:ortho4}, inequalities \eqref{eq:enest} and \eqref{eq:momest} to obtain 
\begin{align*}
 \left|(\xi,Q_b)^2 + (\xi, |y|^2Q_b)^2 + (\xi,yQ_b)^2 + (\xi, \ii \Lambda Q_b)^2 + (\xi, \ii \Lambda(\Lambda Q_b))^2 + (\xi, \ii \nabla Q_b)^2 \right|\lesssim \Gamma_b^{\frac{3}{2} - 20\nu}. 
\end{align*}
Thus \eqref{eq:bVeryLong} implies \eqref{eq:virNoEta}.
\end{proof}

\subsection{Appendix B}\label{sec:appenB}

In this appendix we report the computations needed in Lemma \ref{thm:refvirial}, based on \cite[Section $3.3$]{MeRaSz10} 
and \cite[Lemma $6$]{MeRa05} and stated in Lemma \ref{thm:refvirial1} below. We recall that the remainder term $\tilde{\xi}$ in \eqref{eq:deco2} satisfies 
\begin{equation} \label{eq:xi21}
	\begin{split}
		& \ii \partial_\tau \tilde{\xi} + \mathcal{\tilde{L}}\tilde{\xi} + \ii ( \dot{ b }- \beta s_c) \partial_b \QQ - \ii \left( \frac{\dot{ \lambda} }{\lambda} + b \right) \Lambda (\QQ + \tilde{\xi})- \ii\frac{\dot{ x}}{\lambda} \cdot \nabla (\QQ + \tilde{\xi}) \\
 & - (\dot{ \gamma } -1) (\QQ + \tilde{\xi})
	 + \ii \eta \lambda^{2 - 2\frac{\sigma_2}{\sigma_1}} |\QQ + \tilde{\xi}|^{2\sigma_2}(\QQ + \tilde{\xi}) + \tilde{R}(\tilde{\xi}) - \tilde{\Phi}_b + F = 0,
	\end{split}
\end{equation}
where
\begin{align*}
\mathcal{\tilde{L}}\tilde{\xi} = \Delta \tilde{\xi} - \tilde{\xi} + \ii b \Lambda \tilde{\xi} + 2\sigma_1 Re(\tilde{\xi} \overline{\tilde{Q}}_b) |\QQ|^{2\sigma_1 - 2} \QQ + |\QQ|^{2\sigma_1} \tilde{\xi}, 
\end{align*}
\begin{displaymath}
\tilde{R}(\tilde{\xi}) = |\QQ + \tilde{\xi}|^{2\sigma_1}(\QQ + \tilde{\xi}) - |\QQ|^{2\sigma_1} \QQ - 2\sigma_1 Re(\tilde{\xi}\overline{\tilde{Q}}_b) |\QQ|^{2\sigma_1 - 2} \QQ 
- |\QQ|^{2\sigma_1} \tilde{\xi}.
\end{displaymath}
\begin{equation}\label{eq:F1}
F= (\Delta \phi_A) \zeta_b + 2\nabla \phi_A \cdot \nabla \zeta_b + \ii b y \cdot \nabla \phi_A \zeta_b,
\end{equation}
and
\begin{equation}\label{eq:phiTilde1}
 \tilde{\Phi}_b = - \Phi_b - \ii s_c b \ZZ + \ii \beta \partial_b \ZZ + |Q_b + \ZZ|^{2\sigma_1} (Q_b + \ZZ) - |Q_b|^{2\sigma_1} Q_b.
\end{equation}
Here $\Phi_b$ is defined in \eqref{eq:psi+phi}, $\zeta_b$ is the outgoing radiation of Lemma \ref{lem:rad}, $\tilde{\zeta}_b = \phi_A \zeta_b$ is the localized outgoing radiation where $\phi_A$ is a smooth cut-off defined in the beginning of Subsection \ref{sec:outgo}. By construction, $\ZZ$ satisfies the following inequality
\begin{equation}\label{eq:zetaLong}
 \left\|(1 + |y|)^{10} ( |\ZZ| + |\nabla \ZZ|^2 \right\|_{L^2} + \left\|(1 + |y|)^{10} ( |\partial_b \ZZ| + |\nabla \partial_b \ZZ|^2\right\|_{L^2}^2 \leq \Gamma_b^{1 - c\rho},
\end{equation}
where $c\rho \ll 1$ is defined in Lemma \ref{lem:rad}. 
By exploiting equation \eqref{eq:xi21} and reproducing the steps in Lemma \ref{lem:locvirEta}, we obtain the following.

\begin{lem} \label{thm:refvirial1}
Let $\eta = 0$. Then for any $t\in [0,T_1)$, there exists $C>0$ such that 
\begin{equation}\label{eq:refvir1}
\begin{split}
\frac{d}{d\tau}f(\tau) & \geq C \left( \int |\nabla\tilde{\xi}|^2 + |\tilde{ \xi}|^2 e^{-|y|}\, dy+ \Gamma_b \right) - \frac{1}{C} \left( s_c + \int_A^{2A} |\xi|^2\, dy\right), 
\end{split}
\end{equation}
where
\begin{equation}\label{eq:f11}
\begin{aligned}
 f = - \frac{1}{2}Im \left( \int y \cdot \nabla \tilde{Q}_b \overline{\tilde{Q}_b}\, dy\right) - \left( \tilde{\zeta}_b, \ii \Lambda \tilde{Q}_b \right) + (\xi, \ii \Lambda \tilde{\zeta}).
\end{aligned} 
\end{equation}
\end{lem}

\begin{proof}
By taking the scalar product of equation \eqref{eq:xi21} with $\Lambda \QQ$, we obtain that 
\begin{equation}\label{eq:bLong1} 
\begin{aligned}
 0 &= (\ii \partial_\tau \tilde{\xi}, \Lambda \QQ)+ \dot{b} (\ii \partial_b \QQ, \Lambda \QQ) + \left(2E[\QQ] - \left( \ii \beta s_c \partial_b \QQ + \tilde{\Phi}_b - F, \Lambda \QQ \right) \right) \\ 
 & +(\tilde{\mathcal{L}} \tilde{\xi} + \tilde{R}(\tilde{\xi}), \Lambda \QQ) - 2E[\QQ] \\
 & - \left( \ii \left( \frac{\dot{ \lambda} }{\lambda} + b \right) \Lambda (\QQ + \tilde{\xi}) + \ii\frac{\dot{ x}}{\lambda} \cdot \nabla \QQ + (\dot{ \gamma } -1) \QQ, \Lambda \QQ\right)\\
 &- \left( \ii \left( \frac{\dot{ \lambda} }{\lambda} + b \right) \Lambda \tilde{\xi} + (\dot{ \gamma } -1) \tilde{\xi} + \ii\frac{\dot{ x}}{\lambda} \cdot \nabla \tilde{\xi}, \Lambda \QQ\right). 
\end{aligned}
\end{equation}
Now we repeat the steps in Lemma \ref{lem:locvirEta}, and arrive to the preliminary estimate 
\begin{equation}\label{eq:bVeryLong1}
\begin{aligned}
 &- \dot{b} \left((\ii \partial_b \QQ, \Lambda \QQ) - (\ii \tilde{\xi}, \Lambda \partial_b \QQ) \right) - \frac{d}{d\tau} (\ii \tilde{\xi}, \Lambda \QQ) \\
 &= s_c\left( - 2b(\ii \tilde{\xi}, \Lambda \QQ) - \beta (\tilde{\xi}, \ii \Lambda \partial_b \QQ) + 2E[\QQ] + \| \QQ\|_{L^2}^2 + (\dot{\gamma} - 1)\| \QQ\|_{L^2}^2 \right) \\
 & - (\tilde{\xi}, \Lambda (\tilde{\Phi}_b - F)) - 2\lambda^{2 - 2s_c} E[\psi] + \mathcal H (\tilde{\xi},\tilde{\xi}) + \mathcal{E}(\tilde{\xi},\tilde{\xi}) + \frac{1}{\sigma_1 + 1}\int R^{(3)}(\tilde{\xi})\, dy\\
 &- \left( \ii \left( \frac{\dot{ \lambda} }{\lambda} + b \right) \Lambda \tilde{\xi} + (\dot{ \gamma } -1) \tilde{\xi} + \ii\frac{\dot{ x}}{\lambda} \cdot \nabla \tilde{\xi}, \Lambda \QQ\right). 
\end{aligned}
\end{equation}
The most important difference between \eqref{eq:bVeryLong1} and \eqref{eq:bVeryLong} is the leading order term $(\tilde{\xi}, \Lambda (\tilde{\Phi}_b - F))$ instead of $(\tilde{\xi}, \Lambda \Psi_b)$. 
First, by integration by parts and by using \eqref{eq:Lambdapp}, we have that
\begin{equation*}
\begin{aligned}
 (\ii \partial_b \QQ, \Lambda \QQ) & = \partial_b (\ii \QQ, \Lambda \QQ) - (\ii \QQ, \Lambda \partial_b \QQ) \\
 &= \partial_b (\ii \QQ, y \cdot \nabla \QQ) + (\ii \Lambda \QQ, \partial_b \QQ) + 2s_c( \ii \QQ, \partial_b \QQ),
\end{aligned}
\end{equation*}
which equivalently implies that
\begin{equation*}
 \begin{aligned}
 - \dot{b} (\ii \partial_b \QQ, \Lambda \QQ) = - \frac{1}{2} \frac{d}{d\tau} (\ii \QQ, y \cdot \nabla \QQ) - \dot b s_c ( \ii \QQ, \partial_b \QQ).
 \end{aligned}
\end{equation*}
Thus we can write the left-hand side of \eqref{eq:bVeryLong1} as 
\begin{equation*}
 \begin{aligned}
 - \dot{b} \left((\ii \partial_b \QQ, \Lambda \QQ) - (\ii \tilde{\xi}, \Lambda \partial_b \QQ) \right) - \frac{d}{d\tau} (\ii \tilde{\xi}, \Lambda \QQ) =
 \frac{d}{d\tau} f(\tau) + \dot b (\ii \tilde{\xi}, \Lambda \partial_b \QQ) - \dot b s_c ( \ii \QQ, \partial_b \QQ)
 \end{aligned}
\end{equation*}
where 
\begin{equation*}
 f(\tau) = - \frac{1}{2} (\ii \QQ, y \cdot \nabla \QQ) - (\ii \tilde{\xi}, \Lambda \QQ) .
\end{equation*}
We notice that by using the estimate on $|\dot b|$ \eqref{eq:bLest}, the smallness of $\ZZ$ \eqref{eq:zetaLong}, \eqref{eq:prod1} and the controls \eqref{eq:conxi}, \eqref{eq:conb} we obtain 
 \begin{equation*}
 \begin{aligned}
 - \dot b (\ii \tilde{\xi}, \Lambda \partial_b \QQ) &=- \dot b (\ii (\xi - \ZZ), \Lambda \partial_b (Q_b + \ZZ)) \\ 
 & \gtrsim \Gamma_{b}^{1 - 20\nu} \left(\left( \int|\nabla \xi|^2 + |\xi|^2 e^{-|y|}\, dy\right)^\frac{1}{2} + \Gamma_b^{1 - c\rho} \right)
 \gtrsim \Gamma_b^{\frac{3}{2} - 31\nu}. 
 \end{aligned}
 \end{equation*}
Next, all the terms on the right-hand side of \eqref{eq:bVeryLong1} are bounded in the same way as in Lemma \ref{lem:locvirEta} except for the scalar product $(\tilde{\xi}, \Lambda (\tilde{\Phi}_b - F))$. In this way, we obtain that there exists a constant $C>0$ such that
\begin{equation}\label{eq:fsmall}
 \begin{aligned}
 \frac{d}{d\tau}f \geq C\int |\tilde{\xi}|^2 e^{-|y|} + |\nabla \tilde{\xi}|^2\, dy- C (s_c + \Gamma_b^{\frac{3}{2} - 50\nu}) - (\tilde{\xi}, \Lambda (\tilde{\Phi}_b - F)).
 \end{aligned}
\end{equation}
and $\Phi_b$ is defined in \eqref{eq:Qb} and $F$ in \eqref{eq:F}. We notice that from the definition of $\tilde{\Phi}_b$ in \eqref{eq:phiTilde} and the bound on $\Phi_b$ in \eqref{eq:psiB} and the control on the outgoing radiation \eqref{eq:zetaLong} we get
\begin{equation*}
 \left\|(1 + |y|^{10})(|\tilde{\Phi}_b| + |\nabla \tilde{\Phi}_b| ) \right\|_{L^2}^2 \leq \Gamma_b^{1 + \nu^4} + s_c,
\end{equation*}
which implies that 
\begin{equation*}
 \left|( \tilde{\xi}, \Lambda \tilde{\Phi}_b) \right| \lesssim \Gamma_b^{1 + \nu^4}.
\end{equation*}
Finally, it remains to study the last term $(\tilde{\xi}, \Lambda F)$ in \eqref{eq:fsmall}. We observe that
\begin{equation*}
 (\tilde{\xi}, \Lambda F) = ( \xi - \tilde{\zeta}_b, \Lambda F) = ( \xi, \Lambda F) - (\tilde{\zeta}_b, \Lambda F).
\end{equation*}
For the second term on the right-hand side of the equation above, we notice that
\begin{equation*}
 - (\tilde{\zeta}_b, \Lambda F) = - (\tilde{\zeta}_b, DF) + s_c(\tilde{\zeta}_b, F)
\end{equation*}
and from \eqref{eq:zetaLong}
\begin{equation*}
 \left| s_c(\tilde{\zeta}_b, F) \right| \leq s_c\Gamma^{1 - c\rho} \leq \Gamma^{\frac{3}{2}}, 
\end{equation*}
while using \eqref{eq:gammaB1}
\begin{equation*}
 - (\tilde{\zeta}_b, DF) \geq C_1 \Gamma_b
\end{equation*}
for some $C_1 >0$. Finally, we use the Young inequality and \eqref{eq:zetaLong} to obtain
\begin{equation*}
\begin{aligned}
 \left|( \xi, D F)\right| &\lesssim \left( \int_A^{2A} |\xi|^2\, dy\right)^\frac{1}{2} \left( \int_A^{2A} |F|^2 \right)^\frac{1}{2} \leq \frac{1}{C_2} \int_A^{2A} |\xi|^2\, dy+ \frac{C_1}{2} \Gamma_b. 
\end{aligned}
\end{equation*}
for some $C_2 >0$. Collecting everything we see that from \eqref{eq:fsmall} we obtain
\begin{equation*}
 \begin{aligned}
 \frac{d}{d\tau}f & \geq C\int |\tilde{\xi}|^2 e^{-|y|} + |\nabla \tilde{\xi}|^2\, dy- C (s_c + \Gamma_b^{1 + \nu^2}) - (\tilde{\xi}, \Lambda (\tilde{\Phi}_b - F)) \\
 & \geq C\int |\tilde{\xi}|^2 e^{-|y|} + |\nabla \tilde{\xi}|^2\, dy+ C_1 \Gamma_b - \frac{C_1}{2} \Gamma_b - C_3\left(s_c + \int_A^{2A} |\xi|^2\, dy\right)
 \end{aligned}
\end{equation*}
for some $C_3 >0$. This is equivalent to \eqref{eq:refvir1}
 \end{proof}

\subsection{Appendix C}

In this appendix, we will report the proof inequality \eqref{eq:nonlinearities_finale} which has already been done in the appendix of \cite{MeRaSz10}.
\begin{proof}
 We give a proof only for the term $N_1(\xi)$, since that for $N_2(\xi)$ is equivalent. We define the function $F: \C \to \C$ as
\begin{displaymath}
F(z) = |z|^{2\sigma_1}z. 
\end{displaymath}
From the definition of $r_1$ \eqref{eq:strichartz_pair} it follows that
\begin{displaymath}
\| |\nabla|^{s} N_1(\hat{\xi}) \|_{L^{r'}} \lesssim \frac{1}{\lambda^{(2\sigma_1 + 1)(\tilde{s} - s_c)}} \| |\nabla|^{s} (F(Q_b+ \xi) - F(Q_b)) \|_{L^{r'}}
\end{displaymath}
where
\begin{equation}\label{eq:def_s}
\tilde{s} = s + \frac{d}{2} - \frac{d}{r}.
\end{equation}
We claim the following estimate:
\begin{equation*}
 \| |\nabla|^{s} (F(Q_b+ \xi) - F(Q_b)) \|_{L^{r'}} \lesssim \| |\nabla|^{\tilde{s}} \xi\|_{L^2}.
\end{equation*}
Indeed observe that 
\begin{equation} \label{eq:nnlin:lagrange}
F(Q_b + \xi) - F(Q_b) = \left( \int_0^1 \partial_z F(Q_b + \tau \xi) d\tau \right) \xi + \left( \int_0^1 \partial_{\bar{z}} F(Q_b + \tau \xi) d\tau \right) \bar{\xi}.
\end{equation}
Both terms on the right-hand side of \eqref{eq:nonlinearities_finale} are treated in the same way. We define $q \in \R$ such that
\begin{equation}\label{eq:def_q}
\frac{1}{q} = \frac{1}{r'} - \frac{1}{r}.
\end{equation}
For any function $h$, by the definition of $r$ \eqref{eq:strichartz_pair}, $\tilde{s}$ \eqref{eq:def_s} and by Sobolev embedding we have
\begin{equation}\label{eq:sobolev}
\| h\|_{L^{2\sigma_1q}} \lesssim \| |\nabla|^{s} h\|_{L^r} \lesssim \||\nabla|^{\tilde{s}} h \|_{L^2}.
\end{equation}
By the fractional Leibniz rule (see for example \cite{Gr09}), it follows that
\begin{align}\label{eq:KP}
\bigg \| |\nabla|^{s} \left( \xi \int_0^1 \partial_z F(Q_b + \tau \xi) d\tau \right) \bigg \|_{L^{r'}} &\lesssim \| |\nabla|^{s} \xi\|_{L^2} \bigg \| \int_0^1 \partial_z F(Q_b + \tau \xi) d\tau \bigg \|_{L^q} \\ \nonumber
& + \| \xi\|_{L^{2\sigma_1 q}} \bigg \| | \nabla |^{s} \int_0^1 \partial_z F(Q_b + \tau \xi) d\tau \bigg \|_{L^u},
\end{align}
where 
\begin{equation}\label{eq:def_u}
\frac{1}{u} = \frac{1}{r'} - \frac{1}{2\sigma_1 q}.
\end{equation}
Then, from \eqref{eq:sobolev}, we have
\begin{align}\label{eq:KP_1}
\bigg \| |\nabla|^{s} \left( \xi \int_0^1 \partial_z F(Q_b + \tau \xi) d\tau \right) \bigg \|_{L^{r'}} &\lesssim \| |\nabla|^{\tilde{s}} \xi\|_{L^2} \bigg ( \int_0^1 \| \partial_z F(Q_b + \tau \xi) \|_{L^q} d\tau\\ \nonumber
& + \int_0^1 \| | \nabla |^{s}\partial_z F(Q_b + \tau \xi)\|_{L^u} d\tau \bigg ),
\end{align}
Now it remains to prove that 
\begin{equation}
\bigg ( \int_0^1 \| \partial_z F(Q_b + \tau \xi) \|_{L^q}d\tau
+ \int_0^1 \| | \nabla |^{s}\partial_z F(Q_b + \tau \xi)\|_{L^u} d\tau \bigg ) \lesssim 1.
\end{equation}
By homogeneity 
\begin{displaymath}
\forall \, \tau \in [0, 1], \ |\partial_z F(Q_b + \tau \xi)|\lesssim |Q_b|^{2\sigma_1} + |\xi|^{2\sigma_1},
\end{displaymath}
and so 
\begin{align*}
\int_0^1 \| \partial_z F(Q_b + \tau \xi) \|_{L^q}d\tau \lesssim \int_0^1 (|Q_b|^{2\sigma_1} + |\xi|_{L^{2\sigma_1}q}^{2\sigma_1}) d\tau \lesssim \int_0^1 (|1 + |\xi|_{L^2}^{2\sigma_1}) d\tau \lesssim 1.
\end{align*}

Moreover, recall \cite{Gr09} the equivalent definition of the homogeneous Besov norm:
\begin{displaymath}
\forall \, 0 < \tilde{s} < 1, \ \| u\|_{\dot{B}^{\tilde{s}}_{q,2}}^2 \sim \int_0^\infty \big ( R^{-\tilde{s}} \sup_{|y|\leq R} | u(. - y) - u(.) |_{L^q} \big )^2 \frac{1}{R} dR.
\end{displaymath}
Recall also that $\| |\nabla|^{\tilde{s}} \psi \|_{L^q} \lesssim \| \psi \|_{\dot{B}^{\tilde{s}}_{q,2}}.$
\newline
Observe that, for $ 1 \leq d \leq 3$, $2\sigma_1 > 2$ and it follows by homogeneity, 
\begin{displaymath}
|\partial_z F(u) - \partial_z F(v) |\lesssim |u - v| (|u|^{2\sigma_1 - 1} + |v|^{2\sigma_1 - 1}).
\end{displaymath}
We define 
\begin{displaymath}
h_\tau = Q_b + \tau \xi, \quad 0 \leq \tau \leq 1.
\end{displaymath}
We estimate from H\"older and \eqref{eq:sobolev} 
\begin{align*}
 |\partial_z F(h_\tau) (. - y) - \partial_z F(h_\tau) (.) \|_{L^u} 
&\lesssim \big \| ((h_\tau)(. - y) - (h_\tau)(.)|\, )( |h_\tau(.-y)|^{2\sigma_1 - 1} + |h_\tau(.)|^{2\sigma_1 + 1} )\big \|_{L^u} \\
& \lesssim \| h_\tau (. - y) - h_\tau (.)\|_{L^r} \| h_\tau\|_{L^{2\sigma_1 q}}^{2\sigma_1 - 1} \\
&\lesssim \| h_\tau (. - y) - h_\tau (.)\|_{L^r} \| |\nabla|^{\tilde{s}}h_\tau\|_{L^2}^{2\sigma_1 - 1},
\end{align*}
and so we have that 
\begin{align*}
 \| |\nabla|^{s} \partial_z F(h_\tau) \|_{L^u} &\lesssim \int_0^\infty \big ( R^{-\tilde{s}} \sup_{|y|\leq R} | \partial_z F(h_\tau) (. - y) - \partial_z F(h_\tau) (.) |_{L^u} \big )^2 \frac{1}{R} dR \\ 
& \lesssim \| |\nabla|^{\tilde{s}}h_\tau\|_{L^2}^{2\sigma_1 - 1} \| |\nabla|^{s} h_\tau \|_{L^r} \lesssim \| |\nabla |^{\tilde{s}} h_\tau \|_{L^2}^{2\sigma_1}, 
\end{align*}
and this concludes the proof of \eqref{eq:nonlinearities_finale}. 
\end{proof}

\bibliographystyle{plain}
\bibliography{supercritical}

\end{document}